\documentclass[a4paper,12pt]{article}
\usepackage{amssymb}
\usepackage{color}
\usepackage{amsmath,amssymb, amscd}

\DeclareMathOperator{\tr}{tr}
\DeclareMathOperator{\Tr}{Tr}
\DeclareMathOperator{\ad}{ad}
\DeclareMathOperator{\Ad}{Ad}
\DeclareMathOperator{\Ker}{Ker}

\DeclareMathOperator{\can}{can}

\DeclareMathOperator{\res}{res}

\newcommand{\eqa}{\begin{eqnarray}}
\newcommand{\eeqa}{\end{eqnarray}}
\newcommand{\beq}{\begin{equation}}
\newcommand{\eeq}{\end{equation}}
\newcommand{\hg}{\hat{\mathfrak{g}}}
\newcommand{\g}{\mathfrak{g}}

\newcommand{\n}{\mathfrak{n}}

\newcommand{\h}{\mathfrak{h}}
\newcommand{\fb}{\mathfrak{b}}
\newcommand{\hb}{\hat{\mathfrak{b}}}

\newcommand{\kk}{{\rm{k}}}
\newcommand{\la}{\langle}
\newcommand{\ra}{\rangle}
\newcommand{\e}{\epsilon}
\newcommand{\lm}{\lambda}
\newcommand{\al}{\alpha}
\newcommand{\dl}{\delta}

\newcommand{\LM}{\Lambda}
\newcommand{\tLM}{\tilde{\Lambda}}
\newcommand{\p}{\partial}
\newcommand{\nn}{\nonumber}

\numberwithin{equation}{section}

\newtheorem{theorem}{Theorem}[section]
\newtheorem{prop}[theorem]{Proposition}
\newtheorem{lem}[theorem]{Lemma}

\newtheorem{defn}[theorem]{Definition}
\newtheorem{rem}[theorem]{Remark}
\newenvironment{proof}{\noindent {\it Proof }}{\hfill $\square$}

\begin{document}

\title{Frobenius Manifolds and Central Invariants for the Drinfeld - Sokolov Bihamiltonian Structures}
\author{Boris Dubrovin${}^*$, Si-Qi Liu${}^{**}$, Youjin Zhang${}^{**}$\\
{\small * SISSA, Via Beirut 2-4, 34014 Trieste, Italy}\\
{\small and Steklov Math. Institute, Moscow}\\
{\small ** Department of Mathematical Sciences,
Tsinghua University}\\
{\small Beijing 100084, P. R. China}}
\date{}\maketitle
\begin{abstract}
The Drinfeld - Sokolov construction associates a hierarchy of
bihamiltonian integrable systems with every untwisted affine Lie
algebra. We compute the complete set of invariants of the related
bihamiltonian structures with respect to the  group of
Miura type transformations.
\end{abstract}

\section{Introduction}
The problem of classification of integrable systems of evolutionary PDEs
\begin{eqnarray}
&&
u^i_t = K^i(u; u_x, u_{xx}, \dots), \quad i=1, \dots, n
\nn\\
&&
u=(u^1, \dots, u^n)\in M^n
\nn
\end{eqnarray}
was studied by many mathematicians in the last 40 years with the
help of various techniques; in such a general setup it remains
essentially open, although there are already strong results for many
particular subclasses of equations (see, for example,
\cite{MSY, mikh-shabat, kruskal, zakharov, DZ2} and references
therein).

Before starting the classification work one has to adopt a
definition of complete integrability. For Hamiltonian PDEs
$$
K^i(u; u_x, u_{xx}, \dots) =\{ u^i(x), H\}, \quad i=1, \dots, n
$$
with a suitable class of the Poisson brackets $\{~\,, ~\}$ and the Hamiltonians $H$,
 one can define integrability, similarly to the finite dimensional case, by assuming existence of a complete family of commuting Hamiltonians (we do not explain here the notion of completeness, see e.g. in \cite{DZ2}). More specific is the class of {\it bihamiltonian} evolutionary PDEs admitting two different
 Hamiltonian descriptions
$$
K^i(u; u_x, u_{xx}, \dots) =\{ u^i(x), H_1\}_1=\{ u^i(x), H_2\}_2
$$
with respect to a {\it compatible pair} of Poisson brackets (see below). Under certain genericity assumptions existence of a bihamiltonian representation ensures complete integrability
(see details in \cite{DZ2, DLZ}). Thus, the problem of classification of integrable PDEs reduces to the problem of classification of bihamiltonian structures of a suitable class.
 Even in this bihamiltonian framework the classification problem is still far from being resolved.

In \cite{DZ2, LZ, DLZ} we proposed a kind of a perturbative approach to the classification problem considering the subclass of bihamiltonian PDEs admitting a (formal) expansion with respect to a small parameter $\epsilon$
\begin{eqnarray}\label{eps-s}
&&
u^i_t =A^i_j(u) u^j_x +\epsilon \left[ B^i_j(u) u^j_{xx} + C^i_{jk} (u)u^j_x u^k_x\right]
\nn\\
&&\qquad
+\epsilon^2 \left[ D^i_j(u) u^j_{xxx} +E^i_{j\,k}(u) u^j_x u^k_{xx}
+F^i_{jkl}(u) u^j_x u^k_x u^l_x\right]
+\dots,\\
&&\quad i=1,\dots, n
\nn
\end{eqnarray}
(summation over repeated indices will be assumed). Such systems are to be classified with respect to a certain pronilpotent extension of the group of (local) diffeomorphisms
of the manifold $M^n$ that we called the {\it group of Miura type transformations} (see Section \ref{sec-miura} below). On this way we managed to produce a complete set of invariants of  the bihamiltonian structures satisfying certain semisiplicity assumptions. The first part of these invariants is a differential-geometric object defined on the manifold $M^n$ called {\it flat pencil of metrics}; it describes the bihamiltonian structure of the {\it hydrodynamic limit}
\begin{equation}\label{hyper}
u^i_t =A^i_j(u) u^j_x
\end{equation}
of the system \eqref{eps-s}. The second part comes from the deformation theory of these bihamiltonian
structures of hydrodynamic type; it consists of $n$ functions of one variable called the {\it central invariants}
of the bihamiltonian structure.
The main result of the papers \cite{LZ, DLZ} says that the flat pencil of metrics along with the
collection of central invariants
completely characterizes the equivalence class of a semisimple bihamiltonian structure
with respect  to the group of local Miura type transformations
 (for the precise formulation see the Theorem \ref{thm51} below). In particular, the systems of
 bihamiltonian PDEs with all vanishing central invariants are equivalent to the hydrodynamic limit \eqref{hyper}.

Apart from this trivial case no general results about {\it existence} of bihamiltonian structures and integrable hierachies with a given pair
$$
\mbox{(flat pencil of metrics, collection of central invariants)}
$$
is available. The most studied is the class of the so-called {\it integrable hierarchies of the topological type} motivated by the needs of the theory of Gromov - Witten invariants. For this class the Poisson pencil comes from a semisimple Frobenius structure on the manifold $M^n$; all the central invariants are constants equal to each other. Some partial existence results for integrable hierarchies of the topological type will appear elsewhere \cite{htt}.
So, for the moment we have decided to review the list of known examples of bihamiltonian PDEs of the form \eqref{eps-s} in the framework of our theory of flat pencils and central invariants.

First examples of such analysis have been carried out in \cite{LZ, DLZ}. In the present paper we will consider the flat pencils of metrics and the central invariants for the bihamiltonian hierarchies
constructed by V.Drinfeld and V.Sokolov in \cite{DS}.

The Drinfeld - Sokolov's celebrated paper \cite{DS} gives a very simple construction, in terms of the
Poisson reduction procedure, of a hierarchy of integrable PDEs associated with a Kac - Moody Lie algebra and
a choice of a vertex on the extended Dynkin diagram. In this paper we will only consider the most
well known version of this construction for which the affine Lie algebra is untwisted and the chosen vertex
of the Dynkin diagram is $c_0$ (the one added to the Dynkin diagram of the associated simple Lie algebra).
In this case the hierarchy admits a bihamiltonian structure. The importance of this part of
the Drinfeld - Sokolov construction became clear after the discovery, due to V.\,Fateev and S.\,Lukyanov \cite{fl},
of the connection of the {\it second} Poisson structure for the Drinfeld - Sokolov hierarchy with
the semiclassical limit $W_{\rm cl}(\g)$ of the Zamolodchikov's $W$-algebra \cite{zam}.
Moreover, according to the conjecture of Drinfeld, proved by B.\,Feigin and
E.\,Frenkel (see in \cite{ff, frenkel}) the classical $W$-algebra $W_{\rm cl}(\g)$ arises
naturally on the center of the universal enveloping algebra of the affine algebra $\hat \g'$ of the Langlands
dual Lie algebra $\g'$ at the critical level.

In all these theories the {\it first} Poisson structure of the Drinfeld and Sokolov seems to be something superfluous: in the standard definition the classical $W$-algebra is defined just as  the second Poisson structure of Drinfeld and Sokolov.
However, in the framework of our differential-geometric classification approach a single Poisson bracket has essentially no invariants: after extension to Miura-type transformations with complex coefficients any two local Poisson brackets of our class are equivalent \cite{getzler}; see also \cite{magri, DZ2}.

The main result of this paper is the complete description of the flat pencils of metrics and computation of the central invariants
for the Drinfeld - Sokolov bihamiltonian structures for all untwisted affine Lie algebras.
%\footnote{Our list  $A\, B\, C\, D\, G_2$ coincides with the list of examples of integrable hierarchies considered explicitely in \cite{DS}. A realization of integrable hierarchies associated with the exceptional Lie algebras of other types has been obtained by V.Kac and M.Wakimoto \cite{kw}. They did not consider however the bihamiltonian structure of the exceptional hierarchies.} .
We prove that the flat pencils of metrics are obtained from the Frobenius structures on the orbit spaces of the corresponding Weyl groups
%\footnote{This result has also been independently proved in general by Y.Dinar; the proof will appear elsewhere.},
 constructed by one of the authors in \cite{coxeter} via the theory of flat structures of K.Saito {\sl et al.} \cite{S2, S1}. The central invariants are proved to be all constants; they are identified with $\frac1{48} \times\,$the square lengths, with respect to the normalized invariant bilinear form, of the generators in the Cartan subalgebra. In particular, this proves that the Drinfeld - Sokolov integrable hierarchies for the $A$, $D$ and $E$ series are equivalent, in the sense of Definition \ref{ekvi}, to an integrable hierarchy of the topological type.

The plan of the paper is as follows: we first recall in the next
section the definitions of the bihamiltonian structures, the
associated flat pencils of metrics and central invariants. In
Section \ref{sec21} and Section \ref{sec-3}, we also remind some
preliminaries from Poisson geometry and the Drinfeld - Sokolov
reduction procedure. In Section \ref{principale} we formulate the
Main Theorem about invariants of the Drinfeld - Sokolov
bihamiltonian structures. The proof of this theorem is given in
Section \ref{a-n} for the $A_n$ hierarchies, in Section \ref{bcd-n}
for the $B_n$, $C_n$, $D_n$ hierarchies, and in Section \ref{except}
for the hierarchies associated with the exceptional simple Lie
algebras (some relevant formulae are given in the Appendices). In the final section we give some concluding remarks; we
also give an example of the Drinfeld - Sokolov equation associated
with the {\it twisted} Kac - Moody Lie algebra of the $A_2^{(2)}$
type not admitting a bihamiltonian structure.

\section{Central invariants of semisimple bihamiltonian structures }\label{sec-miura}

We study bihamiltonian structures of the following form
\beq
\begin{split}
&\{u^i(x),u^j(y)\}_a=\{u^i(x),u^j(y)\}_a^{[0]}+\sum_{k\ge 1}\e^k \{u^i(x),u^j(y)\}_a^{[k]}, \\
&\quad \{u^i(x),u^j(y)\}_a^{[k]} =
\sum_{l=0}^{k+1}A^{ij}_{k,l;a}(u;u_x,\cdots, u^{(l)}) \delta^{(k-l+1)}(x-y)
\end{split}
\label{bhs0}
\eeq
where $i,j=1,\cdots,n,\ a=1,2$. Here $u=(u^1, \dots, u^n)\in M$ for
some $n$-dimensional manifold $M$.
The dependent variables $u^1$, \dots, $u^n$ will be considered as local coordinates on $M$.  In this paper the manifold $M$ will be assumed to be diffeomorphic to a ball.

The coefficients $A^{ij}_{k,l;a}$ in \eqref{bhs0} are homogeneous degree $l$ elements of the graded ring ${\cal B}$
of polynomial functions on the jet bundle of $M$
\begin{align*}
{\cal B}= \varinjlim_k {\cal B}_k, &\quad {\cal B}_k={\mathcal C}^\infty (M) [u_x, u_{xx}, \dots, u^{(k)}],
\\
& \deg \p_x^k u^i=k.
\end{align*}
Antisymmetry and Jacobi identity for both brackets as well as the compatibility condition (see below) are understood as identities for formal power series in $\e$.

The leading terms of the Poisson brackets form a bihamiltonian
structure of hydrodynamic type. The coefficients of this term will be redenoted as follows
\eqa
&& \{u^i(x),u^j(y)\}_a^{[0]}=g^{ij}_a(u(x))\delta'(x-y)+{\Gamma^{ij}_k}_a(u(x)) u^k_x\delta(x-y),\label{ldt} \\
&&\quad\ a=1,2.
\nn
\eeqa
Here for any $\lm\in\mathbb{R}$ the symmetric matrix $(g^{ij}_2(u) -\lambda g^{ij}_1(u))$ is assumed to be nondegenerate
for generic $u\in M$.
The Poisson bracket $\{~,~\}_a^{[0]}$ is called {\it the dispersionless limit} of the bracket $\{~\,, ~\}_a$ for every $a=1,\, 2$.

For every $a=1, 2$ the map
$$
{\cal B}\times {\cal B} \to {\cal B}[[\epsilon]]
$$
given by the formula
\begin{eqnarray}\label{skob}
&&
(P,Q) \mapsto \frac{\delta P}{\delta u^i(x)} \Pi^{ij}_a \frac{\delta Q}{\delta u^j(x)} ,
\\
&&
\nn\\
&&
\begin{split}
&P=P(u; u_x, \dots, u^{(p)}), \quad Q=Q(u; u_x, \dots, u^{(q)})\in {\cal B}
\\
&
\Pi^{ij}_a= g^{ij}_a(u)\partial_x   +{\Gamma^{ij}_k}_a (u)  u^k_x
+\sum_{k\ge 1}\e^k \sum_{l=0}^{k+1}A^{ij}_{k,l;a}(u;u_x,\cdots, u^{(l)}) \partial_x^{k-l+1}
\end{split}\nn
\end{eqnarray}
defines a Lie algebra structure on the quotient ring
\begin{equation}\label{quot}
\bar{\cal B}:={\cal B}[[\epsilon]]/{\rm Im}\, \partial_x
\end{equation}
where
$$
\partial_x =\sum_k u^{i, k+1}\frac{\partial}{\partial u^{i,k}}, \quad u^{i,k}:= \frac{\partial^k u^i}{\partial x^k} .
$$
In the formula (\ref{skob}) summation over repeated indice $i$, $j$ is assumed,
$$
\frac{\delta}{\delta u^i(x)} =\frac{\partial}{\partial u^i}-\partial_x \frac{\partial}{\partial u^i_x}+\partial_x^2 \frac{\partial}{\partial u^i_{xx}}
-\partial_x ^3\frac{\partial}{\partial u^i_{xxx}}+\dots
$$
is the Euler - Lagrange operator. The class of equivalence in the quotient space \eqref{quot} of
any element $P(u; u_x, \dots; \epsilon)\in {\cal B}$  will be denoted by
$$
\int P(u; u_x, \dots; \epsilon)\, dx\in\bar {\cal B}
$$
and called a {\it local functional}. Observe that, if $P$ and $Q$ are two homogeneous differential polynomials of the degrees $p$ and $q$ respectively then their bracket (\ref{skob}) will be a homogeneous element of the ring ${\cal B}[[\epsilon]]$ of formal power series in $\epsilon$ of the degree $p+q+1$  if the degree
$$
\deg \epsilon=-1
$$
is assigned to the indeterminate $\epsilon$. So, for an arbitrary local functional of the degree zero
$$
H=\int \sum_{k\geq 0} \epsilon^k  P_k(u; u_x, u_{xx}, \dots, u^{(k)}) dx, \quad \deg P_k (u; u_x, u_{xx}, \dots, u^{(k)})=k
$$
the Hamiltonian vector field
$$
u^i_t = \{ u^i(x), H\}=\Pi^{ij} \frac{\delta P}{\delta u^j(x)}
$$
is a system of evolutionary PDEs of the form \eqref{eps-s} for any of the two Poisson structures $\Pi^{ij} =\Pi^{ij}_1$ or $\Pi^{ij}=\Pi^{ij}_2$.

By the definition of a bihamiltonian structure, any linear combination with constant coefficients of the two Poisson brackets must be again a Poisson bracket on $\bar {\cal B}$ (the so-called {\it compatibility} condition). Due to this property an infinite hierarchy of pairwise commuting systems of PDEs of the form \eqref{eps-s}
%\begin{equation}\label{pde}
%u^i_t =A^i_j(u) u^j_x +\epsilon \left[ B^i_j(u) u^j_{xx} + C^i_{jk} (u)u^j_x u^k_x\right] +\dots, \quad i=1, \dots, n
%\end{equation}
can be associated with the bihamiltonian structure (see details in \cite{DZ2}).

In the  {dispersionless limit} $\epsilon\to 0$ the equations \eqref{eps-s} become
a system of the first order quasilinear PDEs \eqref{hyper}.
%\begin{equation}\label{pde0}
%u^i_t =A^i_j(u) u^j_x.
%\end{equation}
The leading term \eqref{ldt} gives a bihamiltonian structure of \eqref{hyper}.
%For this reason we will also call \eqref{ldt} {\it the dispersionless limit} of the bihamiltonian structure \eqref{bhs0}.

The bihamiltonian structures \eqref{bhs0} will be considered up to invertible linear transformations with constant coefficients
\begin{eqnarray}\label{cambio}
&&
\{~\,, ~\}_1 \mapsto \kappa_{11}\{~\,, ~\}_1 + \kappa_{12} \{~\,, ~\}_2
\nn\\
&&
\{~\,, ~\}_2 \mapsto \kappa_{21}\{~\,, ~\}_1 + \kappa_{22} \{~\,, ~\}_2
\\
&&
\kappa_{11} \kappa_{22}-\kappa_{12}\kappa_{21}\neq 0.
\nn
\end{eqnarray}
The dependence of the associated integrable hierarchy on the changes \eqref{cambio} is nontrivial; it simplifies if one allows only triangular transformations
\begin{eqnarray}\label{cambio1}
&&
\{~\,, ~\}_1 \mapsto \kappa_{11}\{~\,, ~\}_1 \nn\\
&&
\{~\,, ~\}_2 \mapsto \kappa_{21}\{~\,, ~\}_1 + \kappa_{22} \{~\,, ~\}_2
\\
&&
\kappa_{11} \kappa_{22}\neq 0.
\nn
\end{eqnarray}
\begin{defn} A compatible pair of Poisson brackets \eqref{bhs0} considered modulo triangular transformations \eqref{cambio1}
is called a {\rm Poisson pencil}.
\end{defn}

The antisymmetry of the Poisson brackets \eqref{bhs0} gives a system of linear differential constraints for the coefficients. They can be written in a compact form
$$
\Pi^{ji}_a =- \left(\Pi^{ij}_a\right)^\dagger, \quad a=1, \, 2.
$$
Here the (formally) adjoint  to a scalar differential operator
$$
L=\sum_k A_k(x) \partial_x^k
$$
is defined by
\begin{equation}\label{adj}
L^\dagger =\sum_k (-\partial_x)^k A_k(x).
\end{equation}
The validity of the Jacobi identity for the pencil of Poisson brackets imposes a system of highly nontrivial nonlinear differential equations
for the coefficients. We address the classification problem of bihamiltonian structures \eqref{bhs0} under an additional assumption of {\it semisimplicity}.

\begin{defn}
A Poisson pencil \eqref{bhs0} is called {\em semisimple} if the roots $\lm^1(u)$, \dots, $\lm^n(u)$
of the characteristic equation $\det(g^{ij}_2(u)-\lm \, g^{ij}_1(u))=0$ form a system of local coordinates near a generic point $u\in M$. They are called the
{\rm canonical coordinates} of the pencil.
\end{defn}
In the canonical coordinates of a semisimple bihamiltonian structure, the leading terms \eqref{ldt}
diagonalize \cite{FEV, DLZ}:
\begin{equation}
g^{ij}_1(\lm)=f^i(\lm)\delta_{ij},\quad g^{ij}_2(\lm)=\lm^i f^i(\lm)\delta_{ij},\quad i,j=1,\cdots,n
\end{equation}
for some functions $f^1(\lambda)$,\dots,  $f^n(\lambda)$, $\lambda=(\lambda^1, \dots, \lambda^n)\in M$.
\begin{defn}
A {\rm Miura-type transformation} is a change of variables
of
the form
\begin{equation}\label{miura}
u^i\mapsto {\tilde u}^i(u; u_x, u_{xx}, \dots; \epsilon)=F^i_0(u)+\sum_{k\ge 1} \e^k F_k^i(u;u_x,\cdots,u^{(k)})
\end{equation}
where $F^i_k\in {\cal B}$ with $\deg F^i_k=k$, and the map $u \mapsto F^i_0(u)$ is a diffeomorphism of $M$.
\end{defn}

All Miura-type transformations form a group ${\cal G}(M)$. It acts by automorphisms on the graded ring ${\cal B}[[\epsilon]]$. This action commutes with the action of the operator of total $x$-derivative $\partial_x$. Therefore the action of the group ${\cal G}(M)$ on the Poisson brackets of the form \eqref{bhs0} is defined. The explicit formula
\beq
\tilde\Pi^{kl}_a = {L^k_i}\, \Pi^{ij}_a {L^l_j}^\dagger, \quad a=1, \, 2
\nn
\eeq
involves the  operator  of linearization of \eqref{miura}
$$
L^i_j=\sum_m \frac{\partial \tilde u^i(u; u_x, u_{xx}, \dots;\epsilon)}{\partial u^{j,m}} \partial_x^m
$$
and the adjoint operator ${L^i_j}^\dagger$ (see \eqref{adj}).

\begin{defn}\label{ekvi} Two bihamiltonian structures of the form \eqref{bhs0} are called {\rm equivalent} if one can be transformed to another by a combination of a Miura-type transformation and a linear change \eqref{cambio}.
\end{defn}

At the leading order $\epsilon=0$ one obtains the tensor law, with respect to the coordinate change $u^i\mapsto F_0^i(u)$,  for the $(2,0)$ symmetric tensors $g^{ij}_a(u)$ and the associated (contravariant) Levi-Civita connections ${\Gamma^{ij}_k}_a(u)$.

Recall \cite{getzler, magri} that the signature of the metric $g^{ij}(u)$ is the only local (i.e., $M = B^n$ = a small ball in $\mathbb{R}^n$) invariant of a {\it single} Poisson bracket with respect to the group ${\cal G}(B^n)$. The theory of invariants of bihamiltonian structure is more rich.

In order to avoid inessential complications with signs, let us consider the complex situation assuming the manifold $M$ to be complex analytic and all coefficients of the Poisson brackets and of the Miura-type transformations to be complex analytic functions in $u$. Then the complete set of local invariants of semisimple bihamiltonian structures of the form
\eqref{bhs0} consists of
\medskip

$\bullet$ flat pencil of metrics on $M$;
\medskip

$\bullet$ collection of $n$ functions of one variable called {\it central invariants}.
\medskip

Flat pencil of metrics on $M$ is, roughly speaking, a pair of (contravariant) metrics $g^{ij}_1(u)$, $g^{ij}_2(u)$ such that,
at any point $u\in M$ their arbitrary linear combination
$$
a_1 g^{ij}_1(u)+a_2 g^{ij}_2(u)
$$
has zero curvature, and the contravariant Christoffel coefficients for the above metric also  have the form of the same linear combination
$$
a_1{\Gamma^{ij}_k}_1 + a_2{\Gamma^{ij}_k}_2
$$
(see details in \cite{tani}).

In particular, a flat pencil of metrics arises on an arbitrary Frobenius manifold according to the following
construction \cite{coxeter, D1}. Recall that an arbitrary Frobenius manifold is equipped with
a flat metric
$\la ~\,, ~\ra$, a product of tangent vectors $(a,b) \mapsto a\cdot b$,
and an Euler vector field $E$. We put
\beq\label{fmt}
(~\,, ~)_1:=\la ~\,, ~\ra
\eeq
and define the second metric\footnote{It also appeared in \cite{arnold} under the guise of the
operation of convolution of invariants of reflection groups.} on the cotangent bundle from the equation
\begin{equation}\label{fro}
(\omega_1, \omega_2)_2 =i_{E}\, \omega_1\cdot \omega_2
\end{equation}
that must be valid for an arbitrary pair of 1-forms on the Frobenius manifold.
In this formula the identification of tangent and cotangent spaces at every point is done by means of the
{\it first} metric $(~\,, ~)_1$. By means of this identification one defines
the product of 1-forms $\omega_1\cdot \omega_2$ via the product of tangent vectors.

Similarly to Definition \ref{ekvi} we give

\begin{defn}\label{ekvi1} Two flat pencils are called (locally) {\rm equivalent} if one can be transformed to another by a combination of a (local) diffeomorphism and a linear change
\eqref{cambio}.
\end{defn}

The differential geometry problem of local classification of flat pencils reduces to an integrable system of differential equations (see \cite{DLZ} and references therein).

One thus arrives at the problem of local classification of semisimple Poisson pencils of the
form \eqref{bhs0}, \eqref{ldt} with a {\it given} flat pencil of metrics (i.e., with the given leading term \eqref{ldt}).
The theory of {\it central invariants} gives a parametrization of the infinitesimal deformation
space of the bihamiltonian structure \eqref{ldt}. We will not recall here the underlined cohomological
theory \cite{LZ, DLZ}; we only give the computational formulae for the central invariants.

Denote  $P^{ij}_a(\lm)$ (resp. $Q^{ij}_a(\lm)$) the components of the tensor $A^{ij}_{1,0;a}(u)$
(resp. $A^{ij}_{2,0;a}(u)$) in the canonical coordinates. Here $i,j=1,\cdots,n,\ a=1,2$. The
$i$-th ($i=1,\cdots,n$) central invariant of the semisimple bihamiltonian structure \eqref{bhs0} is defined by
\begin{equation}\label{fc}
c_i(\lm)=\frac1{3 (f^i(\lm))^2} \left[Q^{ii}_2(\lm)- \lm^i Q^{ii}_1(\lm)+
\sum_{k\ne i}\frac{(P^{ki}_2(\lm)-\lm^i P^{ki}_1(\lm))^2}{f^k(\lm)(\lm^k-\lm^i)}\right].
\end{equation}

\begin{theorem}[\cite{DLZ}]\label{thm51}
i) Each function $c_i(\lm)$ defined in \eqref{fc} depends only on $\lm^i$, $i=1, \dots, n$.
ii) Two semisimple bihamiltonian structures of the form \eqref{bhs0} with the same leading terms $\{~\,, ~\}_a^{[0]}$,
$a=1,2$ are equivalent  {\em if{}f} they have the same set of central invariants
$c_i(\lm^i)$,  $i=1,\cdots,n$.
\end{theorem}

Note that linear transformations \eqref{cambio} yield fractional linear transformations of the canonical coordinates
$$
\lambda^i \mapsto \frac{\kappa_{21}+\lambda^i \kappa_{22}}{\kappa_{11}+\lambda^i \kappa_{12}}, \quad i=1, \dots, n.
$$
The transformation law of central invariants is given by
\begin{equation}\label{law}
c_i \mapsto \Delta^{-1} (\kappa_{11} +\kappa_{12}\lambda^i)\, c_i, \quad i=1, \dots, n
\end{equation}
where $\Delta=\kappa_{11}\kappa_{22}-\kappa_{12}\kappa_{21}$. For the unimodular transformations, $\Delta=1$,  the half-differentials
$$
\Omega_i=c_i(\lambda^i) (d\lambda^i)^{1/2}
$$
remain invariant, while for the simultaneous rescalings
$$
\{~\,, ~\}_a \mapsto \kappa\, \{~\,, ~\}_a, \quad a=1, \, 2
$$
one has
$$
c_i \mapsto \kappa^{-1} c_i , \quad i=1, \dots, n
.
$$
Observe that the central invariants do not change when rescaling only the {\it first} Poisson bracket without changing the second one. Because of this the central invariants of a Poisson pencil are well defined up to a common constant factor.

\section{Preliminaries from Poisson geometry}\label{sec21}

Before explaining the Drinfeld - Sokolov procedure let us first recall the clasical construction of the linear Poisson bracket on the dual space $\g^*$ to a finite dimensional Lie algebra $\g$ (the so-called {\it Lie - Poisson bracket}). It is uniquely defined by the following requirement: given two linear functions $a$, $b$ on $\g^*$,
$a, \, b\in \g$, their Poisson bracket coincides with the commutator in $\g$:
\begin{equation}\label{lie-pois}
\{ a, b\} =[a,b].
\end{equation}
Choosing a basis in the Lie algebra
\begin{equation}\label{lie}
\g={\rm span}(e_1, \dots, e_N), \quad [e_i, e_j]=\sum_{k=1}^N c_{ij}^k e_k
\end{equation}
one obtains the Lie - Poisson bracket in the associated dual system of coordinates $(x_1, \dots, x_N)$ on $\g^*$ written in the following form
\begin{equation}\label{lie-pois1}
\{ x_i, x_j\} =\sum_{k=1}^N c_{ij}^k x_k, \quad i, \, j=1, \dots, N.
\end{equation}
The Jacobi identity for the linear Poisson bracket (\ref{lie-pois1}) is equivalent to the Jacobi identity for the Lie algebra (\ref{lie}). The Poisson bivector \eqref{lie-pois1} will be denoted
\beq\label{pig}
\pi_\g\in \Lambda^2 T_x\g^*.
\eeq

Linear Hamiltonians
$$
H_a(x) =\la a, x\ra, \quad a\in \g, \quad x\in \g^*
$$
generate the coadjoint action of the Lie group $G$ associated with $\g$:
\begin{equation}\label{coad}
\dot x =\{ x, H_a\} \quad \Leftrightarrow \la b, x(t) \ra =\la e^{- t \ad a} b, x(0)\ra\quad \mbox{for any}\quad b\in \g.
\end{equation}

A simple generalization is given by linear inhomogeneous Poisson bracket
\begin{equation}\label{lie-pois2}
\{ x_i, x_j\} =\sum_{k=1}^N c_{ij}^k x_k+c_{ij}^0.
\end{equation}
It can be interpreted as the Lie - Poisson bracket on the one-dimensional central extension
$$
0\to \mathbb{C}\kk \to \tilde \g \to \g\to 0
$$
of the Lie algebra by means of the 2-cocycle
\begin{equation}
c^0(e_i, e_j) =c_{ij}^0, \quad c^0([a,b],c) +c^0([c,a],b)+c^0([b,c],a)=0.
\end{equation}

\bigskip
%\bigskip

Let us now recall the setting of the Marsden - Weinstein Hamiltonian reduction procedure \cite{mw, mr}. Given a Poisson manifold ${\mathcal M}$, a family of Hamiltonians
$$
H_1(x), \dots, H_N(x) \in {\mathcal C}^\infty ({\mathcal M})
$$
forming a $N$-dimensional Lie subalgebra $\g$ in ${\mathcal C}^\infty ({\mathcal M})$
$$
\{ H_i, H_j\}=\sum_{k=1}^N c_{ij}^k H_k(x), \quad c_{ij}^k = \mbox{const}
$$
generates a Poisson action on ${\mathcal M}$ of the connected and simply connected Lie group $G$ associated with $\g$,
assuming that any nontrivial linear combination of the generators $H_1$, \dots, $H_N$ is not a Casimir of the Poisson bracket on ${\cal M}$.
The vector valued function
$$
{\mathcal P}(x) =(H_1(x), \dots, H_N(x))\in \g^*
$$
is called the {\it moment map} for the Poisson action. The diagram
$$
\begin{CD}
{\mathcal M}  @>{\cal P}>> \g^*\\
@V{g}VV @VV{{\rm Ad}^* g}V\\
{\mathcal M}  @>>{\cal P}> \g^*
\end{CD}
$$
is commutative for any $g\in G$.

Given a Hamiltonian $H\in {\cal C}^\infty ({\cal M})$ invariant with respect to the action of the group $G$
$$
\{ H, H_i\}=0, \quad i=1, \dots, N
$$
the  goal of the reduction procedure is to reduce the order of the Hamiltonian system
\beq\label{ham11}
\dot x =\{ x, H\}
\eeq
i.e., to find a Poisson manifold $({\cal M}^{\rm red}, \{~\,, ~\}_{\rm red})$ of a lower dimension and a
Hamiltonian $H_{\rm red}\in {\cal C}^\infty({\cal M}^{\rm red})$  such that problem of integration of the
Hamiltonian system \eqref{ham11} is reduced to the one for
$$
\dot y=\{y, H_{\rm red}\}_{\rm red}, \quad y\in {\cal M}^{\rm red}.
$$

The construction of the reduced space can be given as follows.

Consider a smooth common level surface of the Hamiltonians
$$
{\mathcal M}_h:=\{ x\in {\mathcal M}\, | \,H_1(x)=h_1, \dots, H_N(x)=h_N\} ={\mathcal P}^{-1}(h)\nn\\
$$
where
$$
h=(h_1, \dots, h_N)\in \g^*
$$
is a regular value of the moment map.
Denote $G_h\subset G$ the stabilizer of $h$ with respect to the coadjoint action of $G$ on $\g^*$. The Lie algebra $\g_h$ of the stabilizer is the kernel of the map
\beq\label{stab}
\pi_\g: \g\simeq T_h^* \g^* \to T_h \g^* \simeq \g^*
\eeq
where $\pi_\g$ is the
Poisson bivector \eqref{pig} on $\g^*$.
The group
 $G_h$ acts freely on ${\cal M}_h$.
Assume  this action to be free also on some neighborhood of ${\cal M}_h\subset {\cal M}$ and that the orbit space ${\mathcal M}_h/G_h$ has a structure of a smooth manifold. Define
%Then one obtains the {\it reduced Poisson manifold}
$$
{\mathcal M}^{\rm red}_h:={\mathcal M}_h/G_h.
$$
%The Poisson algebra of functions on ${\mathcal M}^{\rm red}_h$ can be realized as the image of the restriction homomorphism
%$$
%\left[ {\cal C}^\infty \left( {\cal M}\right)\right]^{G_h}\to {\cal C}^\infty \left( {\mathcal M}^{\rm red}_h\right)
%$$
%of the algebra of $G_h$-invariant functions defined on some %neighborhood of the submanifold ${\cal M}_h\subset {\cal M}$.

We will give a construction of the reduced Poisson bracket on ${\mathcal M}^{\rm red}_h$ for the simplest case $G_h=G$. In this particular case the Poisson brackets of the generators all vanish on ${\cal M}_h$:
$$
\{ H_i, H_j\}|_{{\cal M}_h}=0, \quad i, j=1, \dots, N.
$$
Functions on ${\mathcal M}^{\rm red}_h$ can be identified with
$G$-invariant functions on ${\cal M}_h$.
For any two $G$-invariant functions $\alpha$, $\beta$ on ${\cal M}_h$ denote
$\hat \alpha$, $\hat\beta$ arbitrary extensions of these two functions on a neighborhood of ${\cal M}_h$.

\begin{defn} Under the above assumptions the Poisson bracket on the reduced
space ${\cal M}_h^{\rm red}={\cal M}_h /G$ defined by the formula
\beq\label{dir31}
\{ \alpha, \beta\}_{\rm red}:= \{ \hat\alpha, \hat\beta\}|_{{\cal M}_h}
\eeq
%where $\hat\alpha$, $\hat\beta$ are arbitrary extensions to a neighborhood of ${\cal M}_h\subset {\cal M}$ of the $G_h$-invariant functions $\alpha, \, \beta\in {\cal C}^\infty({\cal M}_h)$
is called the {\rm reduced Poisson bracket}.
\end{defn}

It is easy to see that the rhs of \eqref{dir31} is a $G$-invariant function on ${\cal M}_h$. Moreover, the definition does not depend on the extensions $\hat\alpha$, $\hat\beta$ of the $G$-invariant functions $\alpha$, $\beta$.

\section{The Drinfeld - Sokolov reduction}\label{sec-3}

In this section we will briefly outline the main steps of the Drinfeld - Sokolov reduction for the case of untwisted affine Lie algebras. Proofs of all the statements of this section can be found in \cite{DS}.

Let $\g$ be a simple Lie algebra over $\mathbb{C}$, $G$ the associated connected and simply connected Lie group. Fix an
invariant bilinear form $\la ~\,, ~ \ra_\g$ on $\g$.
The central extension
$$
0\to \mathbb{C}\kk\to\hg \to L(\g)\to 0
$$
of the loop algebra $L(\g ):=C^\infty(S^1, \g)$ is defined as the direct sum of vector
spaces $\hg=L(\g )\oplus \mathbb{C}\kk$ equipped with the following Lie bracket
\[[q(x)+a \kk, p(x)+b \kk]=[q(x), p(x)]+ \omega(q, p)\kk.\]
Here the $2$-cocycle $\omega$ is defined by
\begin{equation}\label{cocycle}
\omega(q, p)=-\int_{S^1}\la q(x),p'(x)\ra_\g\, dx.
\end{equation}
The integral over the circle is normalized in such a way that
$$
\int_{S^1} \, dx=1.
$$

Let ${\mathcal M}\subset \hg^*$ be the subspace of linear functionals taking value $\e$ at the central element k. The space ${\mathcal M}$ can be naturally identified
with the space of first order linear differential operators
\begin{equation}\label{g-dual}
{\mathcal M}=\left\{\e\, \frac{d}{dx} + q(x)\, | \, q(x) \in L(\g)\right\}
\end{equation}
in such a way that the coadjoint action of the loop group $\hat G=L(G)$ restricted
 onto ${\mathcal M}$ is given by the gauge transformations
\eqa\label{gauge}
\begin{split}
&\e\,\frac{d}{dx} + q(x) \mapsto \Ad_{g(x)} \left(\e\,\frac{d}{dx} + q(x)\right)
\\
&
\quad\quad\quad\qquad~ =\e\,\frac{d}{dx} +\Ad_{g(x)} q(x) +\Omega_{g(x)} \left(\e\,\frac{dg}{dx}\right)
\end{split}
\eeqa
where the $\g$-valued 1-form $\Omega_g:T_gG\to T_eG=\g$
is defined by
$$
\Omega_g(X) =- dR_{g^{-1}} X
$$
and
$$
R_h:G\to G
$$
is the right shift by $h\in G$.
It is a Poisson action with respect to the standard linear Lie - Poisson bracket on the dual space $\hg^*$ to the Lie algebra $\hg$. Recall that the restriction of the Poisson bracket on the subspace ${\mathcal M}\subset \hg^*$ in our realization (\ref{g-dual}) is uniquely determined by the following condition: the  Poisson bracket of two linear functionals
\beq\label{linear1}
\begin{split}
&H_{a(x)}[q]=\int_{S^1} \la a(x), q(x) \ra_\g\, dx, \quad H_{b(x)}[q]=\int_{S^1} \la b(x), q(x) \ra_\g\, dx
\\
& a(x), ~ b(x) \in L(\g)
\end{split}
\eeq
coincides with the Lie bracket in $\hg$:
\begin{equation}\label{pb-dual}
\{ H_{a(x)}, H_{b(x)}\} = H_{c(x)} +\omega(a,b), \quad c(x) = [a(x), b(x)].
\end{equation}
Observe that this functional can be also written in the following elegant form
\begin{eqnarray}\label{pb-dual1}
&&
\{ H_{a(x)}, H_{b(x)}\} =\frac1{\e}\int_{S^1} \la a(x), [b(x), \e\,\frac{d}{dx} +q(x)]\ra_\g\, dx
\nn\\
&&
\\
&&=-\frac1{\e}\int_{S^1}\la a(x), \e\,b_x(x)+\ad_{q(x)} b(x)\ra\, dx.
\nn
\end{eqnarray}
In these (and also subsequent) formulae we denote $H[q]$ the value of a functional $H$ on the operator
$$
\e\,\frac{d}{dx} + q(x) \in {\mathcal M}
$$
for brevity.

It is a standard fact from the theory of Lie - Poisson brackets that the linear Hamiltonians (\ref{linear1})
generate the coadjoint action (\ref{gauge}). The Drinfeld - Sokolov construction can be interpreted as the
Hamiltonian reduction procedure applied to a certain subgroup of the gauge group $\hat G$.

\begin{rem} Given a basis $I^1$, \dots, $I^N$ in $\g$ such that
\eqa
&&
[I^i, I^j]=c^{ij}_k I^k
\nn\\
&&
\la I^i, I^j\ra_\g = g^{ij}
\nn
\eeqa
one obtains a system of coordinates
\beq\label{uu}
u^i=\la I^i, \xi\ra, \quad \xi\in\g^*, \quad i=1, \dots, N
\eeq
on the dual space $\g^*$. The Poisson bracket \eqref{pb-dual1} can be written in the form
\beq\label{pb-dual10}
\{ u^i(x), u^j(y)\}=\frac1{\e} \, c^{ij}_ku^k(x) \delta(x-y) - g^{ij}\, \delta'(x-y).
\eeq
This form of the Poisson bracket is similar to \eqref{bhs0} but the $\e$-expansion begins with terms of order $\epsilon^{-1}$. These terms will disappear after the reduction.
\end{rem}

We need some preliminaries from the simple Lie algebras theory in order to develop a suitable infinite dimensional analogue of the above construction of Marsden - Weinstein reduction.

Denote $n$ the rank of the simple Lie algebra $\g$ over $\mathbb{C}$. Choose a Cartan subalgebra $\h\subset \g$ and denote
$X_i, H_i, Y_i$ $(i=1,\cdots,n)$ a set of Weyl generators of $\g$ associated with $\h=\mbox{span}(H_1, \dots, H_n)$.
The generators of the Cartan subalgebra can be identified with the basis of simple coroots
\begin{equation}\label{coroot}
H_i=\alpha_i^\vee\in\h, \quad i=1, \dots, n
\end{equation}
associated with a given basis of simple roots $\alpha_1$, \dots, $\alpha_n\in\h^*$.
Recall the commutation relations between the generators:
\beq\label{commute}
\begin{split}
&[H_i, H_j]=0,
\\
&[H_i, X_j]=A_{ij} X_j, \quad [H_i, Y_j]=-A_{ij} Y_j,
\\
&[X_i, Y_j]=\delta_{ij} H_i.
\end{split}
\eeq
Here $\left(A_{ij}\right)$ is the Cartan matrix of the Lie algebra $\g$,
\beq\label{cartan}
A_{ij}=\la \alpha_j, \alpha_i^\vee\ra.
\eeq
The full set of defining relations in the simple Lie algebra $\g$ is obtained by adding to (\ref{commute}) the Serre relations
$$
(\ad X_i)^{1-A_{ij}}X_j =0, \quad (\ad Y_i)^{1-A_{ij}}Y_j=0, \quad
i\ne j.
$$

The choice of Cartan subalgebra defines the {\it principal gradation} on $\g$
\beq\label{ts-1}
\g=\bigoplus_{1-h\le j\le h-1}\, \g^j
\eeq
such that $X_i\in\g^1,\ Y^i\in\g^{-1}, \ H_i\in\g^0$. Here $h$ is the Coxeter number of $\g$.

Let $\n=\n^+,\, \n^-$ be the nilpotent
subalgebras generated by $\{X_i\},\, \{Y_i\}$ respectively. Introduce also the Borel subalgebras $\fb=\fb^+=\n^+\oplus\h,\
\fb^-=\n^-\oplus\h$.

Let $N\subset G$ be the subgroup of the Lie group $G$ associated with the Lie subalgebra $\n\subset\g$.
We will apply the Hamiltonian reduction procedure to the coadjoint action (\ref{gauge}) of the loop
group $\hat N=L(N)$ on the subspace ${\mathcal M}\subset \hg^*$.

As we already know the coadjoint action (\ref{gauge}) of the subgroup
$L(N)$ is generated by the linear Hamiltonians of the form
$$
H_{v(x)}[q]=\int_{S^1} \la v(x), q(x)\ra_\g\, dx, \quad v(x) \in \hat\n=L(\n).
$$
Therefore the moment map
\begin{equation}\label{moment1}
\mathcal{P}: {\mathcal M}\to \hat\n^*
\end{equation}
associated with the coadjoint action of $\hat N$ is given by the same formula considered as a linear functional on $\hat\n$
\begin{equation}\label{moment2}
\mathcal{P}(q(x)) (v(x)) = \int_{S^1} \la v(x), q(x)\ra_\g\, dx, \quad v(x) \in \hat\n.
\end{equation}
As
$$
\g=\fb\oplus \n^- %\quad\mbox{and}\quad \n^\bot=\fb,
$$
%where $\n^\bot$ is
and the orthogonal complement of $\n$ w.r.t. the invariant bilinear form $\la\ ,\ \ra_\g$ coincides with $\fb$,
one can identify the dual space $\n^*$ with the quotient
\begin{equation}\label{nil-dual}
\n^*=\g/\fb\simeq \n^-.
\end{equation}
Thus the moment map (\ref{moment2}) can be identified with the orthogonal projection of the $\g$-valued function $q(x)$ onto the ``lower triangular part"
\beq\label{moment3}
\begin{split}
&
\mathcal{P}(q(x)) = \pi^-(q(x))
\\
&\mbox{where} \quad \pi^-:\g\to \n^- \quad \mbox{is the natural projection}.
\end{split}
\eeq

Let us now choose a particular value of the moment map.
Let
\begin{equation}\label{choice}
I=\sum_{i=1}^n Y_i \in \n^-
\end{equation}
be a  {\it principal nilpotent element} (see \cite{kostant}) of the Lie algebra $\g$.
Denote
\begin{equation}\label{orbit}
{\mathcal M}^I:=\mathcal{P}^{-1}(I)=I+\hb
\end{equation}
the level surface of the moment map
considering $I$ as a constant map $S^1 \to \n^-$.

{}From the commutation relations (\ref{commute}) it follows that the element $I\in\n^*$ is
invariant with respect to the coadjoint action of $\n$,
$$
[I, \n ]\subset \fb.
$$
Therefore the level surface ${\mathcal M}^I$ is invariant with respect to the gauge action of the nilpotent group
$N$. By definition the functionals on the quotient ${\mathcal M}^I/\hat N$ are the gauge invariant functionals on ${\mathcal M}^I$. We will now construct a ``system of coordinates" on the quotient space.

According to the theory of simple Lie algebras \cite{kostant}, the map
$$
\ad_I:\n \to \fb
$$
is injective.
We fix a subspace $V$ of $\fb$ such that
\begin{equation}\label{sub-v}
\fb=V\oplus[I,\n],
\end{equation}
so $\dim V=\dim \fb-\dim \n=n$.

\begin{prop} The Hamiltonian action of the loop group $\hat N$ on ${\mathcal M}^I$ is free, namely, each orbit contains a unique operator of the form
$$
\e\,\frac{d}{dx}+q^{\rm can}(x) \quad \mbox{with}\quad q^{\rm can}(x)\in V.
$$
\end{prop}

According to this result of \cite{DS} the reduced Poisson manifold can be identified
with the space of operators written in the canonical form
\beq\label{zh-2}
{\mathcal M}^I/\hat N \simeq \left\{\e\, \frac{d}{dx}+q^{\rm can}(x) \ |\ q^{\rm can}(x)\in V\right\}
\eeq
for the given choice of the subspace $V\subset\fb$ of the form (\ref{sub-v}). Let us now construct a bihamiltonian structure on this reduced manifold.

Let us first do the following trivial observation: given an element $\alpha\in\g$, the formula
\begin{eqnarray}\label{pencil}
&&
\{ H_{a(x)}, H_{b(x)}\}_\lambda =\frac1{\e}\int_{S^1} \la a(x), [b(x), \e\,\frac{d}{dx} +q(x)-\lambda\,\alpha]\ra_\g\, dx
\nn\\
&&
\\
&&
= -\frac1{\e}\int_{S^1}\la a(x), \e\, b_x(x)+\ad_{q(x)} b(x)\ra\, dx+\lambda\,\frac1{\e}\int_{S^1} \la a(x), \ad_\alpha b(x)\ra_\g\, dx
\nn
\end{eqnarray}
(cf. (\ref{pb-dual1})) defines a Poisson bracket on ${\mathcal M}$ for an arbitrary $\lambda$. Indeed, the translation
$$
q(x) \mapsto q(x) -\lambda\,\alpha
$$
for any $\lambda$ is a Poisson map for a linear Poisson bracket. We obtain thus a Poisson pencil on ${\mathcal M}$.

Let us now choose $\alpha$ to be a generator of the (one-dimensional) centre
of the nilpotent subalgebra,
\begin{equation}\label{centr}
\alpha\in\n, \quad [\alpha,\n]=0.
\end{equation}

The functionals on the reduced space ${\mathcal M}^I/\hat N$ can be realized as functionals
on ${\cal M}^I$ invariant with respect to the gauge action of $\hat N$.
Let us call them simply gauge invariant functionals for brevity.

\begin{prop} Given two gauge invariant functionals $\phi[q]$, $\psi[q]$ on ${\mathcal M}^I$, then for any
of their extensions $\hat{\phi}[q]$, $\hat{\psi}[q]$ to $\mathcal{M}$, the functional obtained by restricting
the
Poisson bracket
\beq
\{\hat{\phi}, \hat{\psi}\}_\lm
\eeq
to $\mathcal{M}^I$ is again a gauge invariant functional on ${\mathcal M}^I$.
\end{prop}

According to this result,
the projection of the Poisson pencil from ${\mathcal M}$ to the
reduced space ${\mathcal M}^I/\hat N$ is again a Poisson pencil.
In principle this completes the Drinfeld - Sokolov construction, although
the explicit realization of the bihamiltonian structure on the reduced space strongly depends
on the choice of the subspace $V$ in $(\ref{sub-v})$. Changing the subspace yields a Miura-type
transformation of the resulting bihamiltonian structure. The resulting bihamiltonian structure
\beq\label{dsbh}
\{~\,, ~\}_\lm=\{~\,, ~\}_2-\lm\,\{~\,, ~\}_1
\eeq
is called the Drinfeld - Sokolov bihamiltonian structure associated to the simple Lie algebra $\g$.
The commuting Hamiltonians of the
associated integrable hierarchy can be constructed as (formal) spectral invariants of the differential operator
\beq\label{zh-1}
\e\,\frac{d}{dx} + q^{\rm can}(x) +I - \lambda \alpha.
\eeq

In the subsequent sections, we will recall the explicit representations, following \cite{DS},
of the reduced space and also of the bihamiltonian structures associated to the simple Lie algebras of type
$A$-$B$-$C$-$D$ in
terms of pseudo-differential operators.

%For practical purposes it is convenient to chose a graded basis in $\n_{\rm dual}$ using decomposition
%\beq\label{ndualb1}
%\n_{\rm dual} =\h\oplus \n_{\rm dual}^{-j}, \quad  \n_{\rm dual}^{-j}\subset \g^{-j}
%\eeq
%where
%\eqa\label{ndualb2}
%&&
%\n_{\rm dual}^{-j}=\g^{-j}, \quad \mbox{if}\quad j\quad\mbox{is not an exponent of}\quad \g
%\nn\\
%&&
%\\
%&&
%\dim \g^{-m_k}/ \n_{\rm dual}^{-m_k} = \mbox{multiplicity of the exponent}\quad m_k
%\nn
%\eeqa
%(we use that, according to \cite{kostant} the centralizer ${\rm Ker}\, \ad_I$ is a commutative subalgebra in $\g$ having generators only in the degrees $-m_1$, \dots, $-m_n$, and the number of generators in the degree $-m_k$ is equal to the multiplicity of exponent $m_k$).

\medskip

At the end of this section we also mention an alternative approach to the Drinfeld - Sokolov reduction,
due to P.\,Casati and M.\,Pedroni \cite{cp}. They start from the bihamiltonian structure
\begin{eqnarray}
&&
\{ H_{a(x)}, H_{b(x)}\}_1 =- \int_{S^1} \la a(x), \ad_\alpha b(x)\ra_\g\, dx
\nn\\
&&
\nn\\
&&
\{ H_{a(x)}, H_{b(x)}\}_2=-\int_{S^1}\la a(x), \e\,b_x(x)+\ad_{q(x)} b(x)\ra\, dx
\nn
\end{eqnarray}
on ${\cal M}$ with $\alpha$ chosen as in \eqref{centr}. Then they apply the procedure of bihamiltonian reduction,
inspired by the more general Marsden - Ratiu reduction algorithm \cite{mr} that, in this case, consists of the following main steps:
\medskip

$\bullet$ find all Casimirs of the first Poisson bracket
$$
\{~, H_{b(x)}\}_1=0.
$$
They have the form $H_{b(x)}$ with $[\alpha, b(x)]\equiv 0$;
\medskip

$\bullet$ choose a suitable common level surface $S$ of the Casimirs;
\medskip

$\bullet$ consider the Lie subalgebra of those Casimirs that the Hamiltonian flows $\{~, H_{b(x)}\}_2$ are tangent to $S$. The quotient of $S$ over the action of these flows coincides with the Drinfeld - Sokolov reduced space. The generating functions of the commuting Hamiltonians are the Casimirs of the Poisson pencil
$$
\{~\,, ~\}_2 -\lambda \{~\,, ~\}_1
$$
restricted onto the reduced space. In a more recent paper \cite{ortenzi} this bihamiltonian approach has been also applied to the $G_2$ hierarchy.

\section{Formulation of Main Results}\label{principale}

As the first result of the present paper, we will identify the
dispersionless limit of the Drinfeld - Sokolov bihamiltonian structures with the canonical bihamiltonian structures
defined on the jet spaces of the Frobenius manifolds -- the
orbit spaces of the Weyl groups. To this end, we need first to establish an isomorphism
between the reduced manifolds ${\cal{M}}^I/{\hat N}$ that
underline the Drinfeld - Sokolov bihamiltonian structures and
the loop spaces of the orbit spaces of the Weyl groups.

As in section 4, let $X_i, Y_i, H_i,\
i=1,\dots,n$ be a set of Weyl generators of the simple Lie algebra
$\g$.

We specify the choice of the
complement of the subspace $[I,\n]$ of $\fb$ that appears in \eqref{sub-v}
so that
\beq
V=\oplus_{j=0}^{h-1}\, V_j,
\eeq
where the subspaces $V_j$ satisfy
\beq
V_j\in \fb_j=\fb\cap \g^j,\quad \fb_j=V_j\oplus [I,\fb_{j+1}].
\eeq
Note that $V_j$ is not a null space if and only if $j$ is one of the
exponents
$$
1=m_1\le m_2\le\dots\le m_n=h-1
$$
of the simple Lie algebra $\g$. For all simple Lie algebras except
the ones of $D_n$ type with even $n$ the exponents have multiplicity one, i.e.
$\dim V_{m_i}=1$ and the exponents are distinct.
For the $D_n$ (with even $n$) case, the exponents $m_i$ for $i\ne \frac{n}{2}, \frac{n}{2}+1$
have multiplicity one, $m_{\frac{n}2}=m_{\frac{n}2+1}=n-1$ and
$\dim V_{n-1}=2$.

To choose a system of local coordinates of the reduced manifold
${\cal{M}}^I/{\hat N}$ of \eqref{zh-2}, we fix a
canonical form
$$
\e\,\frac{d}{d x}+q^{\rm{can}}+I\in{\cal{M}}^I/{\hat N}
$$
of the linear operator $\e\,\frac{d}{d x}+q+I$ under the gauge action of
$\hat N$ such that
\beq\label{q-can}
q^{\rm{can}}=\sum_{i=1}^n
u^i\,\gamma_i\in V.
\eeq
Here for the exponent $m_i$ with multiplicity one, $\gamma_i$ is a basis of the
one-dimensional subspace $V_{m_i}$; for the $D_n$ case with even $n$, $\gamma_{\frac{n}2},
\gamma_{\frac{n}2+1}$ is a basis of the 2-dimensional subspace $V_{n-1}$. Then $u^1,\dots,
u^n$ form a coordinate system on the space $V\subset \mathfrak n$.

\begin{rem} The subspace $V_{h-1}=\fb_{h-1}$ is determined uniquely since
$$
\fb_j=0\quad \mbox{\rm for}\quad j\geq h.
$$
Recall \cite{bour1} that $\fb_{h-1}$ coincides with the (one-dimensional) centre of $\n$. We will choose the basic vector
$\gamma_n\in V_{h-1}$ as follows:
\beq\label{gamman}
\gamma_n=\alpha
\eeq
where the generator $\alpha$ of the centre of $\n$ has been chosen in \eqref{centr} (see also \eqref{zh-1}).
\end{rem}

According to the results of Section 4 there exists a gauge transformation reducing the linear operator $\e\,\frac{d}{dx} + q +I$ to the canonical form,
\beq\label{g1}
S^{-1}(x) \left( \e\,\frac{d}{dx} + q +I\right)\, S(x) = \e\,\frac{d}{d x}+q^{\rm{can}}+I
\eeq
where the function $S(x)$ takes values in the nilpotent group $N$. The canonical form $q^{\rm can}$ and the
reducing gauge transformation $S(x)$ are determined uniquely from the following recursion
procedure\footnote{Strictly speaking, the form we write \eqref{g1} and the recursion relation \eqref{recur1} uses a
matrix realization of the Lie algebra. See \cite{DS} for the formulation of the recursion procedure independent of the matrix realization.}
\beq\label{recur1}
[I, S_{i+1}] -q_i^{\rm can}= \sum_{j=1}^i S_j q_{i-j}^{\rm can} -q_i -\sum_{j=1}^i q_{i-j} S_j -\e\,\frac{d\,S_i}{dx}, \quad i\geq 0.
\eeq
Here we use decomposition
\beq\label{recur2}
S=1+S_1 +S_2 +\dots \in \hat{N}
\eeq
induced by the principal gradation of $\g$ since the exponential map
$$
\n \to N
$$
is a polynomial isomorphism. As it was proved in \cite{DS}, the reducing transformation and the canonical form
are uniquely determined from the recursion relation. Moreover, they are differential polynomials in $q$.
In particular, the defined above coordinates $u^1$, \dots, $u^n$ of $q^{\rm can}$  are certain differential polynomials
\beq\label{recur3}
u^i=u^i(q; q_x, \dots, q^{(h-1)}), \quad i=1, \dots, n.
\eeq

We will now use these differential polynomials for defining a polynomial isomorphism of affine algebraic varieties
\beq\label{affi1}
\h/W \to V.
\eeq
where $W=W_\g$ is the Weyl group of the simple Lie algebra $\g$.

Restricting the differential polynomials $u^i(q; q_x, q_{xx},\dots)$ to the Cartan
subalgebra
$$
q=\xi=\sum_{i=1}^n \xi^i\,\alpha_i^{\vee} \in C^{\infty}(S^1,\h)
$$
we obtain differential polynomials
\beq\label{miura-1}
 u^1(\xi;\xi_x,\xi_{xx},\dots),\dots,u^n(\xi;\xi_x,\xi_{xx},\dots).
\eeq
Define polynomial functions on $\h$ by
\beq\label{miura-2}
y^i(\xi)=u^i(\xi;0, 0,\dots)\in \mathbb{C}[\h^*].
\eeq

\begin{lem}
The functions  $y^i(\xi)$ are $W$-invariant homogeneous polynomials of degree $m_i+1$. Moreover, they generate
the ring of $W$-invariant polynomials $\mathbb{C}[\h^*]^{W}$.
\end{lem}
\begin{proof} The restriction
$$
F(q; q_x, q_{xx}, \dots) \mapsto F(q; 0, 0, \dots)=:f(q)
$$
of any gauge invariant polynomial function on the differential operators
of the form
$$
\e\,\frac{d}{dx} + q
$$
yields a  polynomial function on $\g$ invariant wrt adjoint action of the Lie group $G$. Further restriction onto the Cartan subalgebra establishes an isomorphism
$$
S(\g)^G \to S(\h)^W = \mathbb C[\h^*]^W
$$
of the ring of ${\rm Ad}$-invariant polynomial functions on $\g$ and the ring of $W$-invariant polynomial functions on $\h$,
according to Chevalley theorem \cite{bour1}. Furthermore, the homomorphism
$$
S(\g)^G \to S(\fb)^N
$$
defined by the formula
$$
f \mapsto f(I + q), \quad q\in \fb
$$
is an isomorphism (see \cite{kostant-2}, Theorem 1.3). Finally, according to Theorem 1.2 of \cite{kostant-2} the adjoint action of the nilpotent group  establishes an isomorphism of affine varieties
$$
N\times (I+V) \to I+\fb.
$$
Combining these statements we prove that the polynomials
$y^1(\xi)$, \dots, $y^n(\xi)$ generate the ring $\mathbb C[\h^*]^W$.

%We first prove that $y^i(\xi)$ are $W$-invariant polynomials. %For this we only
%need to prove the invariance of these functions with respect %to the generating
%reflections $\sigma_i\in W$ corresponding to the simple roots %$\al_i$
%\beq
%\sigma_i (\xi)=\xi-\la \al_i,\xi\ra\,\al_i^\vee,\quad i=1,\dots,n.
%\eeq
%Here we identify $\h$ with $\h^*$ by using the fixed invariant %bilinear
%form $\la\, ,\,\ra_{\g}$.
%It is easy to verify that
%\beq
%e^{\ad_X} \left(I+\xi\right)=I+\sigma_i (\xi),
%\eeq
%where
%\beq
%X=-\la\al_i,\xi\ra\,\,X_i\in\n\,.
%\eeq
%Since the functions $u^i(\xi; \xi_x, \dots)$ are gauge invariant %differential polynomials,
%by the above construction we know that $y^i(\xi)$ are
%invariant
%under the transformation
%\beq
%\xi\mapsto \tilde\xi
%\eeq
%with $\tilde\xi$ defined by the relation $e^{\ad_X} \left(I+\xi%\right)=I+\tilde\xi$.
%From this it follows that $y^i(\xi)$ are invariant with respect to the reflection $\sigma_i$,
%and consequently, they are $W$-invariant.

Now let us prove that $\deg y^i(\xi)=m_i+1$. From the above definition, we know that
these functions are determined by the following equation obtained from \eqref{g1} by eliminating $d/dx$
\beq\label{ts-2}
e^{\ad_s} (q+I) =q^{\rm{can}}+I,
\eeq
where $s\in\n, q\in \fb, q^{\rm{can}}\in V$ have the decomposition
\beq
s=\sum_{k=1}^{h-1} s_i,\ q=\sum_{k=0}^{h-1} q_k,\ q^{\rm{can}}=\sum_{i=1}^{h-1} q^{\rm{can}}_i
\eeq
with $s_k, q_k\in \fb_k, q^{\rm{can}}_i\in V_i$. Comparing the degree 0 parts of the left and right hand sides of \eqref{ts-2},
we arrive at
\beq
\ad_I s_1=q_0.
\eeq
Since the map $\ad_I:\, \fb_1\to \fb_0$ is an isomorphism, we have a unique $s_1$ satisfying
the above equation. Restricting to $q_0=\xi$ we see that $s_1$ depends linearly on $\xi$. Continuing this procedure by comparing
the degree 1, degree 2 etc. parts, at the i-th step we arrive at the equation
of the form
\beq\label{kost1}
\ad_I s_i+q^{\rm{can}}_{i-1} =F_i
\eeq
where $F_i\in\fb_{i-1}$ is a homogeneous polynomial  in $\xi$ of degree $i$.
If $i-1$ is not an exponent, then the above equation has a unique solution with
$q^{\rm{can}}_{i-1}=0$ since the map
\beq\label{kost}
{\rm ad}_I : \fb_i \to \fb_{i-1}
\eeq
is an isomorphism \cite{kostant}.
So $s_i$ will  be a homogeneous polynomial in $\xi$ of degree $i$. In the case when $i-1=m_k$ is an exponent the map \eqref{kost} is only injective. So the solution $s_i\in \fb_i$, $q_{i-1}^{\rm can}\in V_{i-1}$  of the above equation \eqref{kost1} exists and is determined uniquely.  The degree of homogeneous polynomials $s_i(\xi)$ and $q_{i-1}^{\rm can}(\xi)$ is equal to
$$
\deg s_i(\xi)=\deg q_{i-1}^{\rm can} (\xi)= \deg F_i(\xi)=i = m_k+1.
$$
Thus
the function $y^{k}(\xi)$ (or $y^{k}(\xi), y^{k+1}(\xi)$) when $m_k$ has multiplicity one
(resp. has multiplicity two)  is a homogeneous polynomial of degree $m_k+1$. In this way we prove
that $\deg y^i(\xi)=m_i+1$ for any $i=1, \dots, n$. \end{proof}

{}
We obtained an isomorphism of rings
$$
\mathbb C[V^*] \to \mathbb C[\h^*]^W.
$$
Dualizing we obtain the isomorphism \eqref{affi1} of affine algebraic varieties. This induces the isomorphism
\beq\label{iso-main}
\left\{ \begin{array}{l}\mbox{gauge invariant differential}\\ \mbox{polynomials}~ f(q; q_x, q_{xx}, \dots)\\
\mbox{on the space of differential}\\
\mbox{operators} ~ \e\,\frac{d}{dx} + q +I, ~ q(x)\in\fb\end{array}
\right\}\to \left\{ \begin{array}{l}\mbox{differential polynomials}\\ \mbox{on the affine algebraic}\\ \mbox{variety}\quad \h /W\end{array}
\right\}
\eeq
Recall \cite{coxeter} that the orbit space $M_\g=\h/W$ carries a natural structure
of a polynomial Frobenius manifold. According to \eqref{iso-main} the Hamiltonians of
Drinfeld - Sokolov hierarchy can be realized as polynomial functions, considered modulo
total $x$-derivatives,  on the jet space of the Frobenius manifold. We want to compute the
Drinfeld - Sokolov bihamiltonian structure in terms of $M_\g$.

\begin{theorem}\label{thm-4-1}
Under the isomorphism \eqref{iso-main}, the Drinfeld - Sokolov bihamiltonian structure
associated to an untwisted affine Lie algebra $\hat \g$ is realized
as a bihamiltonian structure on the jet space of $M_\g$. Its dispersionless
limit coincides with the bihamiltonian structure of hydrodynamic type naturally
defined on the jet space of the Frobenius manifold by its flat pencil of metrics.
\end{theorem}
\begin{proof}
Let us first remind the construction of the flat pencil of metrics on the orbit space
$M_\g$.
Actually, the construction works uniformly for the orbit space of an arbitrary finite Coxeter
group $W$ (in our case $W=W_\g$). For the chosen basis of simple roots $\alpha_1$, \dots, $\alpha_n\in \h^*$ denote
$$
G_{ab} =\la \alpha_a^\vee, \alpha_b^\vee\ra_\g, \quad a, \, b=1, \dots, n
$$
the Gram matrix of the invariant bilinear form. Let
\beq\label{def-gram}
\left(G^{ab}\right)=\left( G_{ab}\right)^{-1}
\eeq
be the inverse matrix. It gives a (constant) bilinear form on the
cotangent bundle $T^* \h$. The projection of the bilinear form
onto the quotient $\h/W$ defines a bilinear form on $T^*M_\g$
non-degenerate outside the locus $\Delta\subset M_\g$ of singular
orbits (the so-called {\it discriminant} of the Coxeter group $W$).
In order to represent this form in the coordinates let us choose
the above constructed system of
$W$-invariant homogeneous polynomials $y^1(\xi)$, \dots, $y^n(\xi)$
generating the ring $\mathbb{C}[\h^*]^W$. Here $\xi=\xi^a \alpha_a\in
\h$. The polynomial function
$$
G^{ab} \frac{\p y^i(\xi)}{\p \xi^a}\, \frac{\p y^j(\xi)}{\p \xi^b}
$$
is $W$-invariant for every $i$, $j=1, \dots, n$ and, thus, is a
polynomial in $y^1$, \dots, $y^n$. Denote $g^{ij}_2(y)$ these polynomials,
\begin{equation}\label{metr2}
g_2^{ij}\left(y(\xi)\right) = G^{ab} \frac{\p y^i(\xi)}{\p \xi^a}\, \frac{\p y^j(\xi)}{\p \xi^b}.
\end{equation}
This gives the Gram matrix of the {\it second} metric on $T^*M_\g$ in the coordinates $y^1$, \dots, $y^n$.
The associated contravariant Christoffel coefficients are polynomials ${\Gamma^{ij}_k}_2(y)$ defined from the equations
\begin{equation}\label{christ2}
{\Gamma^{ij}_k}_2(y) dy^k = \frac{\p y^i}{\p \xi^a}\,G^{ab}
\frac{\p^2 y^j}{\p \xi^b \p \xi^c} d\xi^c.
\end{equation}
To define the first metric, following \cite{S1, S2}, let us assume
that the invariant polynomial $y^1(\xi)$ has the maximal degree
$$
\deg y^1(\xi)=h.
$$
Here $h$ is the Coxeter number of the Lie algebra $\g$.
Put
\begin{equation}\label{metr1}
g^{ij}_1(y):=\frac{\p g^{ij}_2(y)}{\p y^1}, \quad {\Gamma^{ij}_k}_1(y): = \frac{\p {\Gamma^{ij}_k}_2(y)}{\p y^1}.
\end{equation}
This is the first metric and the associated contravariant Christoffels of the flat pencil of
metrics \eqref{fmt}, \eqref{fro} for the Frobenius structure on $M_\g$.
The second metric of the pencil
depends only on the normalization of
the invariant bilinear form.
The first metric depends on the choice of the invariant polynomial $y^1(x)$ of the maximal degree.
Changing this polynomial yields a rescaling of the first metric; the Frobenius structure will also be
rescaled. This rescaling, however, does not change the central invariants (see the end of Section \ref{sec-miura}).

Let us also remind the algorithm of \cite{coxeter} of reconstruction of the Frobenius structure on the orbit space\footnote{This construction was extended in \cite{affine, DZZ} to the orbit spaces of certain extensions of affine Weyl groups,
and in \cite{bertola} to the orbit spaces of some Jacobi groups. More recently I.Satake \cite{satake} extended this
construction to the orbit spaces of the reflection groups for elliptic root systems for the so-called case of codimension one.}.
Let $v^1(\xi)$, \dots, $v^n(\xi)$ be a system of {\it flat generators} of the ring of $W$-invariant polynomials in the sense of \cite{S1, S2}. Geometrically they give a system of flat coordinates
for the first metric:
$$
\eta^{ij}:=(dv^i, dv^j)_1 =\mbox{const}.
$$
Put
$$
g^{ij}(v):= (dv^i, dv^j)_2.
$$
Then there exists an element $F(v)$ of the degree $2h+2$ in the ring of $W$-invariant polynomials such that
\begin{equation}\label{old}
\eta^{ik}\eta^{jl} \frac{\p^2 F(v)}{\p v^k \p v^l} =\frac{h}{\deg v^i + \deg v^j -2}\, g^{ij}(v).
\end{equation}
The third derivatives
$$
c_{ij}^k(v) := \eta^{kl}\frac{\p^3 F(v)}{\p v^l \p v^i \p v^j}
$$
are the structure constants of the multiplication on the tangent space $T_vM$.

Define a Poisson bracket for two functionals $\varphi$, $\psi$ on $C^{\infty}(S^1,\h)$ by the formula
\begin{equation}
\{\varphi,\psi\}[\xi]=\int_{S^1} \la\frac{d}{d x}
\rm{grad}_{\xi(x)}\varphi,\rm{grad}_{\xi(x)} \psi\ra_{\g}\, dx
\end{equation}
In terms of the coordinates $\xi^1(x),\dots, \xi^n(x)$, we have
\begin{equation}\label{constant-PB}
\{\xi^i(x),\xi^n(y)\}=-G^{ij}\, \delta'(x-y),\quad i,j=1,\dots,n,
\end{equation}
where $(G^{ij})$ is defined in (\ref{def-gram}).
Then as it is shown in \cite{DS}, the Miura map
$$
\mu:\ (\xi^1,\dots,\xi^n)\mapsto \left(u^1(\xi;\xi_x,\xi_{xx},\dots),\dots,u^n(\xi;\xi_x,\xi_{xx},\dots)\right)
$$
is a Poisson
map between $C^{\infty}(S^1,\h)$ and ${\cal{M}}^I/{\hat N}$ if the
latter is endowed with the second Poisson bracket of the
Drinfeld - Sokolov bihamiltonian structure \eqref{dsbh}.

{}From the above argument and (\ref{constant-PB}), we see that the
second metric (\ref{metr2}) defined on the orbit space of $W_g$
coincides, up to a minus sign, with the metric defined on
$(I+\fb)/N$ by the leading terms of the second Poisson bracket of the
Drinfeld - Sokolov bihamiltonian structure associated with the untwisted affine Lie algebra
$\hat\g$.

The definition of the first Drinfeld - Sokolov Poisson bracket depends on the choice of the base element $\alpha$ of
the one-dimensional center of the nilpotent subalgebra $\n$ of $\g$, see \eqref{pencil}, \eqref{centr}.
We note that $\g^{m_n}=\g^{h-1}$ is just the center of $\n$, so we can take $\gamma_n=\alpha$ in (\ref{q-can}).
Then in terms of the local coordinates $u^1(x),\dots, u^n(x)$ the first Drinfeld - Sokolov Poisson bracket
is obtained from the second one by the shifting
$$
u^n(x)\mapsto u^n(x)-\lm,\quad  \p^k_x u^n(x)\mapsto \p^k_x u^n(x),\quad k\ge 1,
$$
and
$$
\{~\,, ~\}_2\mapsto \{~\,, ~\}_2-\lm \{~\,, ~\}_1.
$$

Thus from the above results it follows the validity of Theorem \ref{mainconj}.
\end{proof}

{{\begin{rem} Relationship of the generalized Drinfeld - Sokolov hierarchies with {\rm algebraic} Frobenius manifolds
is currently under investigation; first results have been obtained in \cite{pavlyk, dinar}.
\end{rem}
}}

%%%%%%%%%%%%%%%%%%%%%%%%%end of change %%%%%%%%%%%%%%%%%%%%%%%%

\begin{theorem}\label{mainconj}
The suitably ordered central invariants of the Drinfeld - Sokolov
bihamiltonian structure for an untwisted affine Lie algebra $\hat
\g$ are given by the formula
\begin{equation}\label{otvet1}
c_i =\frac1{48} \la \alpha_i^\vee, \alpha_i^\vee\ra_\g, \quad i=1,
\dots, n,
\end{equation}
where $\alpha_i^\vee\in \h$ are the coroots of the simple Lie algebra $\g$.
\end{theorem}

In the formula \eqref{otvet1} we use the same invariant bilinear form
as the one used in the definition of the Kac - Moody Lie algebra in Section \ref{sec-3} .

Let us now comment the statement of Theorem \ref{mainconj} regarding the central invariants of the
Drinfeld - Sokolov bihamiltonian structures.
Let us fix on $\g$
the so-called {\it normalized} invariant bilinear form (see \cite{kac}, \S 6.2 and Exercise 6.2)
\begin{equation}\label{normalized}
\la a, b\ra_\g:= \frac1{2 h^\vee} \tr (\ad a \cdot \ad b).
\end{equation}
Here $h^\vee$ is the dual Coxeter number.
With the help of the table in \S\,6.7 of \cite{kac}  one obtains the following
values of central invariants, according to Theorem \ref{mainconj}:
$$
\begin{array}{lccccc}
\g & & c_1 & \dots & c_{n-1} & c_n \\
 & & & & & \\
 & & & & & \\
A_n & & \frac1{24} & \dots & \frac1{24} & \frac1{24} \\
 & & & & & \\
 B_n &  &\frac1{24} & \dots & \frac1{24} & \frac1{12} \\
 & & & & & \\
 C_n & & \frac1{12} & \dots & \frac1{12} & \frac1{24} \\
 & &  & & & \\
 D_n & & \frac1{24} & \dots & \frac1{24} & \frac1{24} \\
 & & & & & \\
E_n, ~n=6, \, 7,\, 8 &  & \frac1{24} & \dots & \frac1{24} & \frac1{24} \\
 & & & & & \\
F_n, ~ n=4 & & \frac1{24} & \frac1{24} & \frac1{12} & \frac1{12} \\
 & & & & & \\
 G_n, ~ n=2 & & \frac1{8} & & & \frac1{24}
\end{array}
$$

%\begin{theorem}\label{thm-4-2}
%The conjecture valids for the
%cases $A_n, B_n, C_n, D_n, G_2$.
%\end{theorem}

The ``breaking of symmetry" between the central invariants
for the non-simply laced Lie algebras has the following ``experimental" explanation. Recall that the central invariants
\eqref{fc} are in one-to-one correspondence with the canonical coordinates on the Frobenius manifold, i.e., with the roots $\lambda_1, \dots, \lambda_n$ of the characteristic equation
\beq\label{ura1}
\det \left( g_2^{ij}(u) -\lambda\, g_1^{ij}(u)\right)=0.
\eeq
It turns out that the characteristic polynomial factorizes in the product of two factors of the degrees $p$ and $q$, $p+q=n$, where $p$ is the
number of long simple roots and $q$ is the number of short simple roots.
Such a splitting defines a partition of the set of central invariants in two subsets; the central invariants inside each of the subsets have the same value. For simply laced root systems the characteristic polynomial is irreducible. Recall that
the map associating with the point $u$ the collection of the coefficients of the characteristic
polynomial \eqref{ura1} for the case of simply laced root systems coincides with
the {\it Lyashko - Looijenga map} \cite{looijenga, lyashko}, see also \cite{hertling}.

The values of the central invariants associated to non-simply
laced Lie algebras can be obtained by means of the following ``folding prescriptions":
\begin{align*}
&B_n: \left(\frac1{24},\cdots, \frac1{24},
\frac1{12}\right)=\left(\frac1{24},\cdots, \frac1{24},
\frac1{24}+\frac1{24}\right),\\
&C_n: \left(\frac1{12},\cdots, \frac1{12},
\frac1{24}\right)=\left(\frac1{24}+\frac1{24},\cdots,
\frac1{24}+\frac1{24},
\frac1{24}\right),\\
&F_4: \left(\frac1{12}, \frac1{12}, \frac1{24},
\frac1{24}\right)=\left(\frac1{24}+\frac1{24},
\frac1{24}+\frac1{24},
\frac1{24}, \frac1{24}\right),\\
&G_2: \left(\frac1{8},
\frac1{24}\right)=\left(\frac1{24}+\frac1{12}, \frac1{24}\right)=
\left(\frac1{24}+\frac1{24}+\frac1{24}, \frac1{24}\right).
\end{align*}
These prescriptions correspond to the ``folding of Dynkin diagrams'' procedure known in the theory of
simple Lie algebras and singularity theory
\[D_{n+1} \to B_n,\ A_{2n+1} \to C_n,\ E_6 \to F_4,\ D_4 \to B_3 \to G_2.\]
Note that the relationships between the Frobenius manifolds associated with simply laced and non-simply laced Coxeter groups established by the folding procedure has been clarified in \cite{folding}.

Let us take $B_3\to G_2$ as an example to illustrate this
relation. Their Dynkin diagrams are
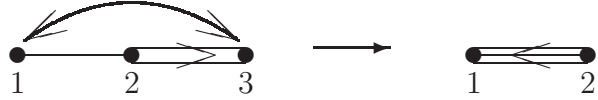
\begin{center}
\setlength{\unitlength}{1mm}
\begin{picture}(80,15)
\put(1,5){\circle*{2}} \put(16,5){\circle*{2}}
\put(31,5){\circle*{2}} \put(0,0){$1$} \put(15,0){$2$}
\put(30,0){$3$} \put(1,5){\line(1,0){15}} \put(16,6){\line(1,0){15}}
\put(16,4){\line(1,0){15}} \put(27,5){\line(-3,1){5}}
\put(27,5){\line(-3,-1){5}} \thicklines
\qbezier(2,7)(16,17)(30,7)\thinlines \put(2,7){\line(1,1){4}}
\put(2,7){\line(4,1){5.488}} \put(30,7){\line(-1,1){4}}
\put(30,7){\line(-4,1){5.488}} \thicklines
\put(40,6){\vector(1,0){10}} \thinlines \put(61,5){\circle*{2}}
\put(76,5){\circle*{2}} \put(60,0){$1$} \put(75,0){$2$}
\put(61,4){\line(1,0){15}} \put(61,5){\line(1,0){15}}
\put(61,6){\line(1,0){15}} \put(66,5){\line(3,1){5}}
\put(66,5){\line(3,-1){5}}
\end{picture}
\end{center}
The folding relation means that the simple Lie algebra of type $B_3$
contains a subalgebra ${\tilde\g}$ of type $G_2$.

Let $X_i, \alpha^{\vee}_i, Y_i\ (i=1,2,3)$ be the Weyl
generators for the simple Lie algebra of type $B_3$ corresponding
to the above Dynkin diagram, then the Lie subalgebra generated by
\[\tilde{X}_1=X_1+X_3, \tilde{\alpha}^{\vee}_1=H_1+H_3,\ \tilde{Y}_1=Y_1+Y_3,\
\tilde{X}_2=X_2,\ \tilde{\alpha}^{\vee}_2=H_2,\ \tilde{Y}_2=Y_2\] is
a simple Lie algebra of type $G_2$, and $\tilde{X}_i,
\tilde{\alpha}^{\vee}_i, \tilde{Y}_i\,(i=1,2)$ form a set of Weyl
generators of this subalgebra.

Note that in this case these two Lie algebras have the same
normalized invariant bilinear forms, so we can compute the central invariants
of associated to the simple Lie algebra of type $G_2$ from that of type
$B_3$ as follows:
\begin{align*}
\tilde{c}_1&=\frac{\la\tilde{\alpha}^{\vee}_1,\tilde{\alpha}^{\vee}_1\ra_{\tilde{\g}}}{48}=
\frac{\la \alpha^{\vee}_1+\alpha^{\vee}_3,\alpha^{\vee}_1+\alpha^{\vee}_3\ra_{\g}}{48}\\
&=\frac{\la \alpha^{\vee}_1,\alpha^{\vee}_1\ra_{\g}}{48}+\frac{\la \alpha^{\vee}_3,\alpha^{\vee}_3\ra_{\g}}{48}=c_1+c_3,\\
\tilde{c}_2&=\frac{\la\tilde{\alpha}^{\vee}_2,\tilde{\alpha}^{\vee}_2\ra_{\tilde{\g}}}{48}=
\frac{\la \alpha^{\vee}_2,\alpha^{\vee}_2\ra_{\g}}{48}=c_2.
\end{align*}
Here we used the fact $\la \alpha^{\vee}_1,
\alpha^{\vee}_3\ra_{\g}=0$.

In general, when we fold two vertices $i, j$ in a Dynkin
diagram, they must be non-connected. So we have $\la
\alpha^{\vee}_i, \alpha^{\vee}_j\ra_{\g}=0$. Then the central
invariant corresponding to the folded vertex reads
\[\tilde{c}=\frac{\la\tilde{\alpha}^{\vee}_i,\tilde{\alpha}^{\vee}_j\ra_{\tilde{\g}}}{48}
=\frac{\la \alpha^{\vee}_i,\alpha^{\vee}_i\ra_{\g}}{48}+\frac{\la
\alpha^{\vee}_j,\alpha^{\vee}_j\ra_{\g}}{48} =c_i+c_j.\]

On the other hand, the folding of Dynkin diagrams also
establishes relations of the bihamiltonian structures associated to
the relevant simple Lie algebras through Dirac reductions. The above
mentioned relation between the central invariants and the folding of
the Dynkin diagrams provide clues to understand connections of the
central invariants of two bihamiltonian structures, with one
bihamiltonian structure obtained from the other by Dirac reduction.
We will study this aspect in detail in subsequent papers.

\medskip

The proof of the theorem \ref{mainconj} will be given in Section \ref{a-n} for the $A_n$ series,
in Section \ref{bcd-n} for the $B_n$, $C_n$, $D_n$ series and in Section \ref{except} for the exceptional cases.

\section{The $A_n$ case}\label{a-n}
We first recall the Drinfeld - Sokolov bihamiltonian structure related to the
simple Lie algebra $\g$ of $A_n$
type. This Lie algebra has the matrix realization $sl(n+1,\mathbb{C})$. We
denote by $e_{ij}$ the matrix with $1$ at the $(i,j)$-th entry and
$0$ elsewhere. The Weyl generators of $\g$ are chosen as
\begin{equation}
X_i=e_{i,i+1},\ Y_i=e_{i+1,i},\ H_i=e_{i,i}-e_{i+1,i+1},\quad
i=1,\dots,n.
\end{equation}
We use here the invariant bilinear form
\begin{equation}\label{zh-9}
\la a, b\ra_\g=\tr(a\,b),
\end{equation}
which coincides with the normalized invariant bilinear form (\ref{normalized}) on $\g$.
The nilpotent subalgebra $\n$, the Borel subalgebra $\fb$ and the group $N$ are realized as
\begin{align*}
\n&=\{(a_{ij})\in \mathrm{Mat}(n+1,\mathbb{C})\left|\, a_{ij}=0, \mbox{ for } i \ge j\right.\},\\
\fb&=\{(a_{ij})\in \mathrm{Mat}(n+1,\mathbb{C})\left|\, a_{ij}=0, \mbox{ for } i > j\right.\},
\\
N&=\{(s_{ij})\in \mathrm{Mat}(n+1,\mathbb{C})\left|\, s_{ij}=0~\mbox{for}~ i>j, ~ s_{ii}=1\right.\}.
\end{align*}

The element $I\in\g$ that is introduced in (\ref{choice})
now has the expression $\sum_{i=1}^n e_{i+1,i}$. We choose the base element $\alpha\in\g$
of the center of $\n$, see (\ref{centr}), as
$$
\alpha=-e_{1,n+1}\in \n.
$$
Let $q$ be an element in $\hb$,
\[q=\sum_{i=1}^n\sum_{j=i}^{n+1}q_{ij}(x)\,e_{ij}-\sum_{i=1}^n q_{ii}(x)\, e_{n+1,n+1}. \]
We can choose the coordinate $q^{\can}$ on the orbit space (\ref{zh-2}) as \cite{DS}
\[q^{\can}=-(u_1(x)e_{1,n+1}+u_2(x)e_{2, n+1}+\cdots+ u_n(x)e_{n, n+1}),\]
where $u_k(x)$ are certain differential polynomials of $q_{ij}$.
Here and henceforth we use lower indices for the variable $u$ instead of upper ones as in
\eqref{q-can} for the convenience of presentation of relevant formulae.
Then the gauge invariant functionals take the following form
\begin{equation}
F=\int_{S^1} f(x, u(x), u_x(x), \cdots) dx. \label{funct}
\end{equation}

The space of the gauge invariant functionals can be
described in the following way \cite{DS}. Consider the operator
\begin{equation}\label{cal-L}
{\cal L}=\e\,\frac{d}{d x}+q+I
\end{equation}
as a $(n+1)\times (n+1)$ matrix with entries of differential operators. Let us  represent it in
the form
\begin{equation}
{\cal L}=\begin{pmatrix}\alpha & \beta\\ A&
\gamma\end{pmatrix}.
\end{equation}
Here $A$ is a $n\times n$ matrix. We can associate to
it a scalar differential operator
\begin{equation}\label{delta-op}
\Delta({\cal {L}}):=\beta-\alpha A^{-1} \gamma.
\end{equation}
Define
\begin{equation}
L=-\Delta({\cal L})^\dag,
\end{equation}
where the conjugation of a differential operator is defined as in
(\ref{adj}). It can be written in the form
\begin{equation}\label{lax-a}
L=D^{n+1}+u_n(x)D^{n-1}+\cdots+u_2(x) D+u_1(x), \ D=\e\,\frac{d}{dx}.
\end{equation}

Gauge invariant functionals on ${\cal M}$ will be identified with functionals on the space of Lax operators \eqref{lax-a}.
The variational derivative of a gauge invariant functional $F$ w.r.t $L$ is defined as
the following pseudo-differential operator
\[\frac{\dl F}{\dl L}=\sum_{i=1}^{n}D^{-i}\frac{\dl F}{\dl u_i}.\]
It is easy to verify the following identity
\begin{equation}
\delta F=\int \sum_{i=1}^n \frac{\delta F}{\delta u_i(x)} \delta u_i\, dx=
\Tr\left(\frac{\delta F}{\delta L} \delta L\right) \label{var-der}
\end{equation}
where the linear functional $\Tr$ on pseudo-differential operators is defined by
$$
\Tr A =\int \res A\, dx \in \bar {\cal B}
$$
and  the residue of a pseudo-differential operator has the definition
\[\res\,(\sum_{i\le m}f_iD^i)=f_{-1}.\]
Recall that, due to the important property of the residue
\begin{equation}\label{resid}
\res (BA) =\res(AB) + \mbox{total } x\mbox{-derivative},
\end{equation}
the formula
\[\Tr(AB)=\int_{S^1} \res(AB)\, dx\in \bar {\cal B} \]
defines an invariant symmetric inner product between two pseudo-differential operators.

In terms of the gauge invariant functionals $F, G$, the Drinfeld - Sokolov bihamiltonian structure
can be written as
\begin{align}
&\{F, G\}_\lm=\{F, G\}_2-\lm\{F, G\}_1 \label{bhs-a} \\
&=\frac1{\e}\Tr\left((LY)_+LX-XL(YL)_++\frac1{n+1}X[L,g_Y]\right)-\lm\,\frac1{\e}\Tr\left([Y,X]L\right),\nn
\end{align}
where $X=\frac{\dl F}{\dl L}, Y=\frac{\dl G}{\dl L}$, and the positive part of a pseudo-differential operator $Z=\sum z_i D^i$
is defined by
$$
Z_+=\sum_{i\ge0}z_i D^i.
$$
The function $g_Y$ is defined by
$$
g_Y=D^{-1}(\res[L,Y]).
$$
Due to \eqref{resid}, $g_Y$ is a differential polynomial of the coefficients of the operators $L, Y$.

In the computation of Poisson brackets of our type it suffices to deal with the linear functionals
\begin{equation}
\ell_X=\int \sum_{i=1}^n a_i(x) u_i(x) dx,\ \ell_Y=\int \sum_{i=1}^n b_i(x) u_i(x) dx. \label{linear}
\end{equation}
Then the operators $X=\delta \ell_X/\delta L, \, Y=\delta \ell_Y/\delta L$ read
\begin{equation}\label{oper-xy}
X=\sum_{i=1}^n D^{-i}a_i(x), \ Y=\sum_{i=1}^n D^{-i}b_i(x).
\end{equation}

%To compute the pencil of flat metrics and the central invariants of the bihamiltonian structure \eqref{bhs-a}
%we must expand it with respect to the powers of $\epsilon$. To this end introduce a slow variable $\tilde{x}=\e\,x$,
%and denote $\frac{d}{d\tilde{x}}$ by $\p_x$. Then we have $D=\e\,\p_x$ and
%$[D, \tilde{x}]=\e$.
For a pseudo-differential operator $Z=\sum_{i \le m} z_i(x) D^i$, define its symbol as
\[\hat{Z}(x,p)=\sum_{i \le m}z_i(x) p^i.\]
The symbol of the composition of two pseudo-differential operators can be computed by the following
well known formula\footnote{Warning: we use here the symbol $\star$ that usually arises in the quantization
of Poisson brackets. However our ``star product" is different from the standard one.}
\begin{eqnarray}
&&
\widehat{Z_1Z_2}(x,p)=\hat{Z}_1(x,p)\star\hat{Z}_2(x,p):= e^{\epsilon\frac{\p^2}{\p p \p x'} }\hat Z_1(x,p) \hat Z_2(x', p')|_{x'=x, p'=p}
\nn\\
&&
\nn\\
&&
=
\sum_{k=0}^\infty\frac{\e^k}{k!}\p_p^k \hat{Z}_1(x,p)\,\p_x^k \hat{Z}_2(x,p).
\label{star}
\end{eqnarray}
Taking the commutator in the leading term one obtains the Poisson bracket
on the $(x,p)$-plane as follows
\begin{eqnarray}\label{2pb}
&&
f(x,p)\star g(x,p) - g(x,p) \star f(x,p) =\epsilon \,\{ f, g\}  +O(\epsilon^2),
\nn\\
&&
\\
&&
 \{ f, g\}:= \frac{\p f}{\p p} \frac{\p g}{\p x} - \frac{\p g}{\p p} \frac{\p f}{\p x}.
\nn
\end{eqnarray}
In the sequel we will often omit writing explicitly the $x$-dependence of the symbol.

The symbol of the positive part of a pseudo-differential operator can be computed by Cauchy integral formula
\begin{equation}
\widehat{Z_+}(p)=\left(\hat{Z}(p)\right)_+=\oint\frac{dq}{2 \pi i}\frac{\hat{Z}(q)}{q-p}.\label{cauchy}
\end{equation}
where the integration is taken along the circle of radius $|q|>|p|$.

Let
$$
\lm(x,p)=p^{n+1}+u_n(x)p^{n-1}+\cdots+u_2(x)p+u_1(x)=\hat L
$$
be the symbol of the Lax operator \eqref{lax-a}.

\begin{theorem}\label{theor-a} (i) The dispersionless limit of the $A_n$ Drinfeld - Sokolov bihamiltonian structure is given by the following formulae
\beq\label{pen1-a}
\begin{split}
&\{ \lambda(x,p), \lambda(y,q)\}_1 = \frac{\lambda'(p)-\lambda'(q)}{p-q}\, \delta'(x-y)\\
&
\quad  + \left[
\frac{\lambda_x(p)-\lambda_x(q)}{(p-q)^2}
-\frac{\lambda'_x(q)}{p-q}\right]\,\delta(x-y),
\end{split}
\eeq
\beq\label{pen2-a}
\begin{split}
&\{ \lambda(x,p), \lambda(y,q)\}_2\\
=&\left(\frac{\lambda'(p)\lambda(q)-\lambda'(q)\lambda(p)}{p-q}+\frac1{n+1} \lambda'(p)\lambda'(q)\right)\, \delta'(x-y)
\\
&+\left[
\frac{\lambda_x(p)\lambda(q)-\lambda_x(q)\lambda(p)}{(p-q)^2}
+\frac{\lambda_x(q)\lambda'(p)-\lambda'_x(q) \lambda(p)}{p-q}\right.\\
&\quad \ \ \left.+\frac1{n+1} \lambda'(p)
\lambda'_x(q)\right]\, \delta(x-y).
\end{split}
\eeq
(ii) The central invariants of the bihamiltonian structure are equal to
$$
c_1=c_2=\dots=c_n=\frac1{24}.
$$
\end{theorem}

Before proceeding to the proof let us explain the notations in the formulae \eqref{pen1-a} - \eqref{pen2-a}. In the left hand sides we simply write the generating polynomials for the matrices $\{ u_i(x), u_j(y)\}_{1,\, 2}$ of Poisson brackets, i.e.,
$$
\{ \lambda(x,p), \lambda(y,q)\}_{1, \, 2}
=\sum_{i,j=1}^n \{ u_i(x), u_j(y)\}_{1,\, 2} p^{i-1} q^{j-1}.
$$
In the right hand sides we denote $\lambda(p)\equiv \lambda(x,p)$,
$$
\lambda'(p) =\frac{\partial}{\partial p} \lambda(x,p),  \quad \lambda_x(p)=\partial_x \lambda(x,p).
$$
Same for the terms depending on $q$, i.e. $\lambda(q)\equiv \lambda(x,q)$, $\lambda'(q) =\frac{\p}{\p q} \lambda(x,q)$ etc.
Observe that the sign of the second metric (the coefficients of $\delta'(x-y)$ of (\ref{pen2-a}))
is opposite to the one given in Proposition 2.4.2 of \cite{S2}.

\smallskip

\begin{proof} Let us introduce the symbols
\begin{equation}\label{sym-fg}
f(p)=\sum_{i=1}^{n}\frac{a_i(x)}{p^{i}},\quad g(p)=\sum_{i=1}^{n}\frac{b_i(x)}{p^{i}}.
\end{equation}
They are related to the symbols of the operators \eqref{oper-xy} via
\begin{equation}
 \hat{X}(p)=\sum_{k=0}^\infty\frac{\e^k}{k!}\p_p^k\p_x^kf(p),\
\hat{Y}(p)=\sum_{k=0}^\infty\frac{\e^k}{k!}\p_p^k\p_x^kg(p). \label{smbl-a}
\end{equation}
We begin with the calculation of the leading term of the first Poisson bracket. Due to \eqref{2pb} one obtains
$$
\{ \ell_X, \ell_Y\}_1 =\int \res \left( \{ g(x,p), f(x,p)\}\,\lambda(x,p)\right) dx +O(\epsilon).
$$
Here $\res$ of a symbol is just the coefficient of $p^{-1}$. Integrating by parts one rewrites
$$
\int \res \left( \{ g, f\}\,\lambda\right) dx= \int \res \left(f \, \{ \lambda, g\}\right) dx.
$$
As the series $f$ contains only negative powers of $p$, one can replace the series $  \{ \lambda, g\}$ by its positive part
$$
 \{ \lambda, g\}_+ =\oint \frac{dq}{2\pi i} \frac{\lambda'(q) g_x(q) -\lambda_x(q) g'(q)}{q-p}.
$$
Integrating by parts in $q$ and inserting two zero terms
$$
-\oint \frac{dq}{2\pi i} \frac{\lambda'(p)}{q-p}\, g_x(q)=0, \quad
\oint \frac{dq}{2\pi i} \frac{\lambda_x(p)}{(q-p)^2}\, g(q)=0
$$
one obtains the following expression for the leading term of the first Poisson bracket
\begin{align}
\{ \ell_X, \ell_Y\}_1 & = \int dx \oint \frac{dp}{2 \pi i}
\oint \frac{dq}{2 \pi i} \left[ f(p) \frac{\lambda'(p) -\lambda'(q)}{p-q} g_x(q) \right.
%\nn\\
%&
\nn\\
&\left.
+\left( \frac{\lambda_x(p) -\lambda_x(q)}{(p-q)^2} -\frac{\lambda'_x(q)}{p-q}\right)\, f(p) g(q)\right] + O(\epsilon).
\nn
\end{align}
This gives the formula \eqref{pen1-a}. Note that the rational functions
$$
\frac{\lambda'(p) -\lambda'(q)}{p-q}
$$
and
$$
\frac{\lambda_x(p) -\lambda_x(q)}{(p-q)^2} -\frac{\lambda'_x(q)}{p-q}
$$
have no singularity on the diagonal, so the order of the loop integrals is inessential.

A similar computation proves also the formula \eqref{pen2-a}.

Let us proceed to computing the higher order corrections.
Note that what we want to compute is just four tensors $P^{ij}_a(u), Q^{ij}_a(u)$ ($a=1,2$)
independent of the jet coordinates (see (\ref{fc})). So through the
computation we can omit all the derivatives of $u_i$ w.r.t. $x$, i.e. we can treat $u_i$ as constants.
By using this assumption, one can obtain
\begin{equation}
g_Y=\oint\frac{dq}{2 \pi i}\sum_{k=1}^\infty\frac{\e^{k-1}}{k!}\p_p^k\hat{L}(q)\p_x^{k-1}\hat{Y}(q). \label{gy}
\end{equation}

By substituting the formulae \eqref{star}, \eqref{cauchy}, \eqref{smbl-a}, \eqref{gy} into the formula
\eqref{bhs-a}, we can obtain
\begin{align*}
\{\ell_X, \ell_Y\}_a=&\int dx \oint\frac{dp}{2 \pi i} \oint\frac{dq}{2 \pi i} \\
&\sum_{k, i, s, j, t\ge0} \p_p^i\p_x^s f(p) \tilde{A}_{a, k, i, s, j, t}(p,q,x)
\e^{k+s+t-1} \p_q^j\p_x^t g(q), \ a=1,2.
\end{align*}
After few integration by parts, the above equation  reduces to the following one
\begin{equation}\label{lxly}
\{\ell_X, \ell_Y\}_a=\int dx \oint\frac{dp}{2 \pi i} \oint\frac{dq}{2 \pi i}
\sum_{k,s\ge0} f(p) A_{a,k,s}(p,q,x) \e^{k+s-1}\p^s_x g(q).
\end{equation}
We already know the coefficients
$$
A_{1,0,1}=\frac{\lambda'(p) -\lambda'(q)}{p-q}
$$
and
$$
A_{2,0,1} =\frac{\lambda'(p)\lambda(q)-\lambda'(q)\lambda(p)}{p-q}+\frac1{n+1} \lambda'(p)\lambda'(q).
$$
The subsequent coefficients $A_{a,0,2}$,  $A_{a,0,3}$\,$(a=1,2)$  read
\begin{align}
A_{1,0,2}=&\frac{\lm'(q)-\lm'(p)}{(q-p)^2}-\frac{\lm''(q)+\lm''(p)}{2(q-p)}, \nn\\
A_{1,0,3}=&\frac{\lm'(q)-\lm'(p)}{(q-p)^3}-\frac{\lm''(q)+\lm''(p)}{2(q-p)^2}+\frac{\lm'''(q)-\lm'''(p)}{6(q-p)}, \nn\\
A_{2,0,2}=&\frac{\lm'(q)\lm(p)-\lm(q)\lm'(p)}{(q-p)^2}-\frac{\lm''(q)\lm(p)-2\lm'(q)\lm'(p)+\lm(q)\lm''(p)}{2(q-p)} \nn\\
&-\frac{\lm''(q)\lm'(p)-\lm'(q)\lm''(p)}{2(n+1)}, \nn\\
A_{2,0,3}=&\frac{\lm'(q)\lm(p)-\lm(q)\lm'(p)}{(q-p)^3}-\frac{\lm''(q)\lm(p)-2\lm'(q)\lm'(p)+\lm(q)\lm''(p)}{2(q-p)^2} \nn\\
&+\frac{\lm'''(q)\lm(p)-3\lm''(q)\lm'(p)+3\lm'(q)\lm''(p)-\lm(q)\lm'''(p)}{6(q-p)} \nn\\
&+\frac{2\lm'''(q)\lm'(p)-3\lm''(q)\lm''(p)+2\lm'(q)\lm'''(p)}{12(n+1)}. \label{a123}
\end{align}

Now we introduce two complex numbers $P, Q$ such that $|P|<|p|, |Q|<|q|$, and define the functions $f(p), g(p)$ as
\[f(p)=\frac1{p-P}\dl(x-y)=\sum_{i=1}^\infty\frac{P^{i-1}}{p^i}\dl(x-y),\ g(p)=\frac1{q-Q}\dl(x-z).\]
Here, unlike the form given in (\ref{sym-fg}), we allow the symbols $f(p), g(p)$ to contain terms of
the form $\frac1{p^i}$ with $i>n$. However, it is easy to see that these additional
terms do not affect the Poisson bracket \eqref{lxly}.

It follows then that
\begin{align*}
\ell_X&=\lm(y,P)-P^{n+1}=u_n(y)P^{n-1}+\cdots+u_2(y)P+u_1(y), \\
\ell_Y&=\lm(z,Q)-Q^{n+1}=u_n(z)Q^{n-1}+\cdots+u_2(z)Q+u_1(z),
\end{align*}
and the formula \eqref{lxly} reads
\begin{align}
\{\lm(y,P), \lm(z,Q)\}_a&=\sum_{k,s\ge0}\e^{k+s-1}\delta^{(s)}(y-z)
\left[\oint\frac{dp}{2\pi i}\oint\frac{dq}{2\pi i} \frac{A_{a,k,s}(p,q,y)}{(p-P)(q-Q)}\right] \nn\\
&=\sum_{k,s\ge0}\e^{k+s-1}A_{a,k,s}(P,Q,y)\delta^{(s)}(y-z). \label{lmlm}
\end{align}

Let $r_1, \cdots, r_n$ be the critical points of the polynomial $\lambda(p)$, i.e., the roots of $\lm'(r)=0$. Assuming them to be pairwise distinct, we have
\[A_{1,0,1}(r_i, r_j, x)=\dl_{ij}\lm''(x,r_i),\ A_{2,0,1}(r_i, r_j, x)=\dl_{ij}\lm(x,r_i)\lm''(x,r_i).\]
This shows that the critical {\it values} $\lm_i=\lm(r_i)$ are the canonical coordinates of the bihamiltonian structure \eqref{lmlm}.
Then the quantities in the formula \eqref{fc} read
\begin{align*}
&f^i=\lm''(r_i), \\
&Q^{ii}_1=\frac1{12}\lm^{(4)}(r_i), \
Q^{ii}_2=\frac1{12}\lm(r_i)\lm^{(4)}(r_i)+\frac{n}{n+1}\frac{\lm''(r_i)^2}4, \\
&P^{ki}_1=\frac{\lm''(r_k)+\lm''(r_i)}{2(r_k-r_i)},\
P^{ki}_2=\frac{\lm''(r_k)\lm(r_i)+\lm(r_k)\lm''(r_i)}{2(r_k-r_i)}.
\end{align*}
Thus the central invariants read
\begin{align*}
c_i=&\frac1{3\lm''(r_i)^2}\left(\frac{n}{n+1}\frac{\lm''(r_i)^2}4+
\sum_{k \ne i}\frac{(\lm(r_k)-\lm(r_i))\lm''(r_i)^2}{4\lm''(r_k)(r_k-r_i)^2}\right)\\
=&\frac1{12}\left(\frac{n}{n+1}+\sum_{k \ne i}\frac{(\lm(r_k)-\lm(r_i))}{\lm''(r_k)(r_k-r_i)^2}\right)
=\frac1{12}\left(\frac{n}{n+1}+\frac{1-n}{2(n+1)}\right)\\
=&\frac1{24}.
\end{align*}
Here the third equality is obtained by applying the residue theorem to the meromorphic function
$$
m(q)=\frac{\lm(q)-\lm(r_i)}{\lm'(q)(q-r_i)^2}.
$$
The Theorem is proved.
\end{proof}

\section{The $B_n$, $C_n$ and $D_n$ cases}\label{bcd-n}

The simple algebras of type $B_n$, $C_n$ and $D_n$ can be realized as matrix Lie algebras $o(2n+1)$, $sp(2n)$ and $o(2n)$.
The details of these realizations are omitted here, see Appendix 1 of \cite{DS}. Note that the Weyl generators
$X_i, Y_i, H_i$ we choose here correspond respectively to $Y_i, X_i, -H_i$ of \cite{DS}.
We begin with the following scalar differential operators satisfying certain symmetry/antisymmetry conditions:
\begin{align}
B_n:&\quad L=D^{2n+1}+\sum_{i=1}^n u_i(x)\,D^{2i-1}+\sum_{i=1}^n v_i(x)\,D^{2i-2}, \quad L+L^\dag=0\label{b-lax}\\
C_n:&\quad L=D^{2n}+\sum_{i=1}^n u_i(x)\,D^{2i-2}+\sum_{i=2}^{n} v_i(x)\,D^{2i-3}, \quad\quad L=L^\dag\label{c-lax}\\
D_n:&\quad L=D^{2n-1}+\sum_{i=2}^{n} u_i(x)\,D^{2i-3}+\sum_{i=2}^{n} v_i(x)\,D^{2i-4}+\rho(x)D^{-1}\rho(x),
\label{d-lax}\\  &\quad\quad\quad\quad\quad\quad\quad\quad\quad\quad\quad\quad\quad\quad\quad\quad\quad
\quad\quad\quad\quad\quad ~ L+L^\dag=0.
\nn
\end{align}
Here  $L^\dag$ is the adjoint operator \eqref{adj}, the coefficients $v_i(x)$ are
linear combinations of derivatives of $u_i(x)$ uniquely determined by the symmetry/antisymmetry conditions.
We assume $u_1(x)=\rho^2(x)$ for the $D_n$ case.

As for the $A_n$ case, the above scalar (pseudo) differential operators can also be derived from the differential operator
${\cal L}$ of the form (\ref{cal-L}). In the present cases, the matrices $q$ are upper triangular ones  belonging to ${\frak{o}}(2 n+1),\,
{\frak{sp}}(2n)$ and ${\frak{o}}(2n)$ respectively. The matrices $I$ are given respectively by
\begin{equation}
I=\sum_{i=1}^n \left(e_{i+1,i}+e_{2n+2-i,2n+1-i}\right),\quad
I=\sum_{i=1}^{n-1} \left(e_{i+1,i}+e_{2n+1-i,2n-i}\right)+e_{n+1,n}\nn
\end{equation}
and
\begin{equation}
I=\sum_{i=1}^{n-1} \left(e_{i+1,i}+e_{2n+1-i,2n-i}\right)+\frac12 (e_{n+1,n-1}+e_{n+2,n}).\nn
\end{equation}
The scalar differential operators $L$ are given by
$
-\Delta({\cal L})^\dagger
$,
where the operator $\Delta$ is defined as in (\ref{delta-op}).

The variational derivative of a functional of $L$ w.r.t. $L$ is now defined as
\begin{equation}
\frac{\delta F}{\delta L}=\frac12\sum_{i=1}^n \left(D^{-2i+\nu}\frac{\delta F}{\delta u_i(x)}
+\frac{\delta F}{\delta u_i(x)}D^{-2i+\nu}\right), \label{vder}
\end{equation}
where $\nu=0,1,2$ for the $B_n$, $C_n$ and $D_n$ cases respectively.
This definition ensures the validity of \eqref{var-der}.

In order to have a uniform expression of the Drinfeld - -Sokolov second hamiltonian structures
for the three types of simple Lie algebras, we fix in this section the invariant bilinear form on $\g$
by
\beq\label{zh-7}
\la a,b\ra_{\g}=\tr(a\,b).
\eeq
Let us note that the normalized invariant bilinear form defined in (\ref{normalized}) for the
simple Lie algebras of type $B_n, C_n, D_n$ have the
expressions
\beq\label{zh-8}
\frac12\,\tr(a\,b),\quad \tr(a\,b), \quad \frac12\,\tr(a\,b)
\eeq
respectively.
With the above fixed invariant bilinear form, the
second hamiltonian structures
for the three types of simple Lie algebras
have a uniform expression
\begin{equation}
\{F,G\}_2=\frac1{\e} \Tr\left[(L Y)_+ L X-X L (Y L)_+\right], \label{bhs}
\end{equation}
while the first ones are defined as the Lie derivatives of the second ones
along the coordinate $u_i$, where $i=1$ for $B_n$, $C_n$ and $i=2$ for $D_n$,
\[\{F,G\}_2(u_i,\cdots)-\lm \{F,G\}_1(u_i,\cdots)=\{F,G\}_2(u_i-\lm,\cdots).\]
Explicitly,
\begin{align}
B_n:&\quad \{F,G\}_1=\frac1{\e} \Tr L\left(Y D X-X D Y\right), \label{bhs-1-b}\\
C_n:&\quad \{F,G\}_1=\frac1{\e} \Tr L\left(Y X-X Y\right), \label{bhs-1-c}\\
D_n:&\quad \{F,G\}_1=\frac1{\e} \Tr L\left(X_+ D Y_+-Y_+ D X_++Y_- D X_--X_- D Y_-\right). \label{bhs-1-d}
\end{align}

Let us now describe the main result of this section. Let
\begin{eqnarray}\label{symb-b}
&&
\lambda_B(p) = p^{2n+1} +\sum_{i=1}^n u_i(x) p^{2i-1}
\\
&&\label{symb-c}
\lambda_C(p) =p^{2n} + \sum_{i=1}^n u_i(x) p^{2i-2}
\\
&&\label{symb-d}
\lambda_D(p) =p^{2n-1} +\sum_{i=2}^n u_i(x) p^{2i-3} +\frac{u_1(x)}{p}
\end{eqnarray}
be the $\epsilon=0$ limits of the symbols of the Lax operators \eqref{b-lax} - \eqref{d-lax}. Introduce
\beq\label{pol-lam}
\begin{split}
&\Lambda_B(P)=\Lambda_C(P)=P^n + u_n(x) P^{n-1} +\dots +u_1(x)\\
&
\Lambda_D(P) = P^{n-1} +u_n(x) P^{n-1} + \dots +u_{n-1}(x) +\frac{u_1(x)}{P}.
\end{split}
\eeq
by the following subtitution:
\begin{eqnarray}\label{subs-pol}
&&
\lambda_B(p) = p \, \Lambda_B(p^2)
\nn\\
&&
\lambda_C(p) = \Lambda_C(p^2)
\\
&&
\lambda_D(p) =p\, \Lambda_D(p^2).
\nn
\end{eqnarray}

\begin{theorem}\label{theor-bcd}
(i) The dispersionless limits of the Drinfeld - Sokolov bihamiltonian structures associated to
the simple Lie algebras of type $B_n$, $C_n$, and $D_n$ have the following uniform
expression
\begin{align}\label{pen1-bcd}
&\{\LM(x,P),\LM(y,Q)\}_1=2\frac{P\LM'(P)-Q\LM'(Q)}{P-Q}\dl'(x-y)\nn\\
&\qquad+\left[\frac{P+Q}{(P-Q)^2}(\LM_x(P)-\LM_x(Q))-2\frac{Q\LM'_x(Q)}{P-Q}\right]\dl(x-y),\\
&
\nn\\
&\{\LM(x,P),\LM(y,Q)\}_2=2\frac{P\LM'(P)\LM(Q)-Q\LM'(Q)\LM(P)}{P-Q}\dl'(x-y)\nn\\
&\qquad+\left[\frac{P+Q}{(P-Q)^2}(\LM_x(P)\LM(Q)-\LM_x(Q)\LM(P))\right.\nn\\
&\qquad\qquad\left.+2\frac{P\LM'(P)\LM_x(Q)-Q\LM'_x(Q)\LM(P)}{P-Q}\right]
\dl(x-y),\label{pen2-bcd}
\end{align}
where $\LM(x,P)=\LM_B$,  $\LM_C$, or $\LM_D$ respectively.

(ii) The central invariants of the Drinfeld - Sokolov bihamiltonian structures read
\begin{align}
B_n:&\quad  c_1=\cdots=c_{n-1}=\frac1{12},\ c_n=\frac16\\
C_n:&\quad  c_1=\cdots=c_{n-1}=\frac1{12},\ c_n=\frac1{24}\\
D_n:&\quad c_1=c_2=\cdots=c_n=\frac1{12}.
\end{align}
\end{theorem}

Note that the rescaling
$$
\la\,\ ,\ \ra_{\g}\mapsto \kappa\, \la\,\ ,\ \ra_{\g}
$$
of the invariant bilinear form on $\g$ yields the rescaling of the central invariants (\ref{fc})
 of the
related Drinfeld - Sokolov bihamiltonian structure
$$
c_i\mapsto \kappa \, c_i,\quad i=1,\dots,n.
$$
So from the definition of the normalized bilinear
form (\ref{normalized}) and (\ref{zh-9}), (\ref{zh-7}), (\ref{zh-8}) and Theorem \ref{theor-a},
Theorem \ref{theor-bcd} it follows
the validity of the Theorem \ref{mainconj} for the cases $A_n, B_n, C_n, D_n$.
%Similarly,the validity of the Theorem \ref{thm-4-2} for the $G_2$ case follows from (\ref{zh-10}) of thenext section.

Before proceeding to the proof of the Theorem let us explain the rule of labeling of the central
invariants for the $B_n$ and $C_n$ cases. The reader may remember that the labeling of the central
invariants is in one-to-one correspondence with labeling of the canonical coordinates. It will be shown
below that the canonical coordinates for the bihamiltonian structure
\eqref{pen1-bcd}, \eqref{pen2-bcd} are defined as follows:
\beq\label{zh-3}
\begin{split}
&\lambda_i =\Lambda(r_i^2), \quad i=1, \dots, n\\
&
\frac{d}{dp} \Lambda(p^2) |_{p=r_i}=0.
\end{split}
\eeq
For $B_n$ and $C_n$ cases $r_n=0$ is always a critical point of $\Lambda(p^2)$. The associated critical value $\lambda_n=\Lambda(0)$ ``breaks the symmetry" between the canonical coordinates; the corresponding central invariant $c_n$ differs from others.

\begin{proof}
The derivation of the dispersionless Poisson structures \eqref{pen1-bcd}, \eqref{pen2-bcd} follows the
lines of the proof of Theorem \ref{theor-a}. We will omit this part of the proof, and proceed directly
to computation of the central invariants.

Some part of the computation can be done uniformly for all the three types of Lie algebras. To this end we introduce the symbol
\begin{equation}\label{symb-nu}
\lm(p)=p^{2n+1-\nu}+\sum_{i=1}^n u_i(x) p^{2i-1-\nu}
\end{equation}
and also
\begin{equation}
f(p)=\sum_{i\ge1} \frac{a_i(x)}{p^{2i-\nu}},\ g(p)=\sum_{i\ge1} \frac{b_i(x)}{p^{2i-\nu}}
\label{symb-ab1}
\end{equation}
(we plan to still use the linear functionals \eqref{linear}). Recall that $\nu=0,1,2$ for $B_n$, $C_n$ and $D_n$ respectively.
The symbols of the pseudo-differential operators $X$ and $Y$ read
\begin{gather}
\hat{X}(p)=f(p)+\frac12\sum_{k\ge1}\frac{\e^k}{k!}\p^k_p\p^k_x f(p),\
\hat{Y}(p)=g(p)+\frac12\sum_{k\ge1}\frac{\e^k}{k!}\p^k_p\p^k_x g(p).\label{sym-XY}
\end{gather}
We omit the derivatives of $u_i$ w.r.t. $x$ in $\hat{L}(p)$ just like in the previous section.

By using the same method used in the proof of Theorem \ref{theor-a}, we obtain the coefficients
$A_{2,0,1}$, $A_{2,0,2}$ and $A_{2,0,3}$ in the expansion \eqref{lxly}
\begin{align}
A_{2,0,1}=&\frac{\lm'(q)\lm(p)-\lm'(p)\lm(q)}{q-p},\ A_{2,0,2}=0, \nn\\
A_{2,0,3}=&\frac{\lm'(q)\lm(p)-\lm'(p)\lm(q)}{2(q-p)^3}-
\frac{\lm''(q)\lm(p)-2\lm'(q)\lm'(p)+\lm''(p)\lm(q)}{4(q-p)^2} \nn\\
&+\frac{\lm'(q)\lm''(p)-\lm'(p)\lm''(q)}{4(q-p)}+\frac{\lm'''(q)\lm(p)-\lm'''(p)\lm(q)}{6(q-p)}. \label{a13}
\end{align}

Now let $P$, $Q$ be two complex numbers such that $|P|<|p|^2$ and $|Q|<|q|^2$.
Define the functions $a_i(x)$, $b_i(x)$ as in \eqref{symb-ab1}
from the following expanions
\[f(p)=\frac{p^{\nu}}{p^2-P}\delta(x-y)=\sum_{k=1}^\infty \frac{P^{k-1}}{p^{2k-\nu}}\delta(x-y),\
g(q)=\frac{q^{\nu}}{q^2-Q}\delta(x-z).\]
Then $\ell_X=\LM(y,P)-P^n$, $\ell_Y=\LM(z,Q)-Q^n$, where
\begin{equation}
\LM(y,P)=P^{n}+u_n(y)P^{n-1}+\cdots+u_1(y). \label{Lambda}
\end{equation}
The second Poisson bracket between the linear functionals now read
\[\{\LM(y,P), \LM(z,Q)\}_2=\sum_{k,s\ge0}\e^{k+s-1}\delta^{(s)}(y-z)
\left[\oint\frac{dp}{2\pi i}\oint\frac{dp}{2\pi i} \frac{(p\,q)^{\nu} A_{2,k,s}(p,q,y)}{(p^2-P)(q^2-Q)}\right].\]
Denote by $R_{2,1}$ the coefficient of $\e^0\,\delta'(y-z)$. It is easy to obtain
\begin{equation}
R_{2,1}=2\frac{P\,\LM'(P)\,\LM(Q)-Q\,\LM'(Q)\,\LM(P)}{P-Q}. \label{r21}
\end{equation}
Here $\LM(P)=\LM(y,P)$, $\LM(Q)=\LM(y,Q)$, and the primes stand for differentiations w.r.t. $P$ or $Q$.
Then by definition one can obtain the coefficient of $\e^0\,\delta'(y-z)$ in $\{\LM(y,P), \LM(z,Q)\}_1$
denoted by $R_{1,1}$
\begin{align}
B_n, C_n:&\quad R_{1,1}=2\frac{P\,\LM'(P)-Q\,\LM'(Q)}{P-Q}, \label{r11-bc}\\
D_n:&\quad R_{1,1}=2\frac{P\,Q(\LM'(P)-\LM'(Q))+P \LM(Q)-Q\LM(P)}{P-Q}. \label{r11-d}
\end{align}

Denote the coefficients of $\e^2\,\delta'''(y-z)$ in $\{\LM(y,P), \LM(z,Q)\}_\alpha$ by $R_{\alpha,3}$.
After a lengthy computation, we obtain
\begin{align}
&B_n: \quad R_{2,3}=\nn\\
&\frac{(P+Q)^2(\LM'(P)\,\LM(Q)-\LM'(Q)\,\LM(P))}{(P-Q)^3}+
4\frac{P^2\,\LM'''(P)\,\LM(Q)-Q^2\,\LM'''(Q)\,\LM(P)}{3(P-Q)} \nn\\
&+2\frac{P\,Q(\LM'(P)\,\LM''(Q)-\LM'(Q)\,\LM''(P))}{P-Q}+2\frac{P\,\LM''(P)\,\LM(Q)-Q\,\LM''(Q)\,\LM(P)}{P-Q} \nn\\
&+3\LM'(P)\,\LM'(Q)-2\frac{P\,Q(\LM''(P)\,\LM(Q)-2\LM'(P)\,\LM'(Q)+\LM(P)\,\LM''(Q))}{(P-Q)^2}. \label{zh-5}\\
&R_{1,3}=\frac{(P+Q)^2(\LM'(P)-\LM'(Q))}{(P-Q)^3}+4\frac{P^2\,\LM'''(P)-Q^2\,\LM'''(Q)}{3(P-Q)} \nn\\
&+2\frac{P\,\LM''(P)-Q\,\LM''(Q)}{P-Q}-2\frac{P\,Q(\LM''(P)+\LM''(Q))}{(P-Q)^2}.\label{zh-6}
\end{align}

\begin{align}
&C_n: \quad R_{2,3}=\frac{(P^2+6\,P\,Q+Q^2)(\LM'(P)\,\LM(Q)-\LM'(Q)\,\LM(P))}{2(P-Q)^3} \nn\\
&+4\frac{P^2\,\LM'''(P)\,\LM(Q)-Q^2\,\LM'''(Q)\,\LM(P)}{3(P-Q)}+2\frac{P\,Q(\LM'(P)\,\LM''(Q)-\LM'(Q)\,\LM''(P))}{P-Q} \nn\\
&+\frac{P\,\LM''(P)\,\LM(Q)-Q\,\LM''(Q)\,\LM(P)}{P-Q}+\LM'(P)\,\LM'(Q) \nn\\
&-2\frac{P\,Q(\LM''(P)\,\LM(Q)-2\LM'(P)\,\LM'(Q)+\LM(P)\,\LM''(Q))}{(P-Q)^2}. \\
&R_{1,3}=\frac{(P^2+6\,P\,Q+Q^2)(\LM'(P)-\LM'(Q))}{2(P-Q)^3}+4\frac{P^2\,\LM'''(P)-Q^2\,\LM'''(Q)}{3(P-Q)} \nn\\
&+\frac{P\,\LM''(P)-Q\,\LM''(Q)}{P-Q}-2\frac{P\,Q(\LM''(P)+\LM''(Q))}{(P-Q)^2}.
\end{align}

\begin{align}
&D_n: \quad R_{2,3}=4\frac{P\,Q(\LM'(P)\,\LM(Q)-\LM'(Q)\,\LM(P))}{(P-Q)^3}+ \nn\\
&+4\frac{P^2\,\LM'''(P)\,\LM(Q)-Q^2\,\LM'''(Q)\,\LM(P)}{3(P-Q)}+2\frac{P\,Q(\LM'(P)\,\LM''(Q)-\LM'(Q)\,\LM''(P))}{P-Q} \nn\\
&-\LM'(P)\,\LM'(Q)-2\frac{P\,Q(\LM''(P)\,\LM(Q)-2\LM'(P)\,\LM'(Q)+\LM(P)\,\LM''(Q))}{(P-Q)^2} \nn\\
&+\frac{P^2\,\LM'(P)\,\LM(Q)-Q^2\,\LM'(Q)\,\LM(P)}{P\,Q(P-Q)}-\LM(0)\frac{P\,\LM'(P)+Q\,\LM'(Q)}{P\,Q}. \\
&R_{1,3}=\frac{4 P\,Q(P\,\LM'''(P)-Q\,\LM'''(Q))}{3(P-Q)}-2\frac{P\,Q(P\,\LM''(P)+Q\,\LM''(Q))}{(P-Q)^2}
\nn\\
&+\frac{4\,P\,Q(P^2\,\tLM'(P)-Q^2\,\tLM'(Q))}{(P-Q)^3}+\frac{P\,Q(\tLM'(P)-\tLM'(Q))}{P-Q}-\frac{\LM(0)(P+Q)}{P\,Q},
\end{align}
where
\beq\label{zh-4}
\tLM(P)=\LM(P)/P.
\eeq

Now we begin to compute the central invariants for the $B_n, C_n$ cases.
The formulae \eqref{r21} \eqref{r11-bc} show that in these two cases we have the same
dispersionless limit, so the corresponding Drinfeld - Sokolov bihamiltonian structures
have the same canonical coordinates. Let $r_1, \cdots, r_n$ be defined
as in (\ref{zh-3}). Then we have $\lm_n=u_1$ and
$\lm_1, \cdots, \lm_{n-1}$ are the critical values of $\LM(P)$.
{}From the formulae \eqref{r21} and \eqref{r11-bc}, one can see that $\lambda_1$, \dots, $\lambda_n$
can serve as the canonical coordinates of the Drinfeld - Sokolov bihamiltonian structures of $B_n$ and $C_n$ type.
Following the notations in \eqref{fc}, we have
\[B_n:\quad  f^i=2\,r_i\LM''(r_i),\quad f^n=2\LM'(0);\]
\[Q^{ii}_1=3 \LM''(r_i)+\frac{14}3 r_i \LM'''(r_i)+r_i^2 \LM''''(r_i),\quad Q^{nn}_1=3\LM''(0).\]
\[ Q^{ii}_2=r_i^2 \LM''(r_i)^2+\LM(r_i) Q^{ii}_1,\quad
Q^{nn}_2=2\LM'(0)^2+3\LM(0)\,\LM''(0);\]
\[ c_i=\frac{Q^{ii}_2-\LM(r_i)\,Q^{ii}_1}{3(f^i)^2}=\frac1{12},\quad
c_n=\frac{Q^{11}_2-\LM(0)\,Q^{11}_1}{3(f^1)^2}=\frac16.\]

\medskip

\[C_n: f^i=2\,r_i\LM''(r_i),\quad f^n=2\LM'(0)\]
\[ Q^{ii}_1=3 \LM''(r_i)+\frac{11}3 r_i \LM'''(r_i)+r_i^2 \LM''''(r_i),\quad
Q^{nn}_1=\frac32\LM''(0);\]
\[Q^{ii}_2=r_i^2 \LM''(r_i)^2+\LM(r_i) Q^{ii}_1,\quad
Q^{nn}_2=\frac12\LM'(0)^2+\frac32\LM(0)\,\LM''(0);\]
\[c_i=\frac{Q^{ii}_2-\LM(r_i)\,Q^{ii}_1}{3(f^i)^2}=\frac1{12},
\quad c_n=\frac{Q^{11}_2-\LM(0)\,Q^{11}_1}{3(f^1)^2}=\frac1{24}.\]
Here $i=1,\dots,n$.

To compute the central invariants for he $D_n$ case, we first rewrite the two Poisson brackets in terms of the symbol
 $\tLM$ defined in (\ref{zh-4}).
Let $\tilde{R}_{\alpha,k}$ be obtained from $R_{\alpha,k}$ of (\ref{zh-5}), (\ref{zh-6})
with $\LM$ replaced by $\tLM$.
Denote by $S_{\alpha,k}$ the coefficients of $\e^{k-1} \delta^{(k)}(y-z)$ in $\{\tLM(P,y), \tLM(Q,z)\}_\alpha$.
Then we have $S_{\alpha,1}=\tilde{R}_{\alpha,1}$, and
\[S_{2,3}=\tilde{R}_{2,3}-\LM(0)\frac{P\,\LM'(P)+Q\,\LM'(Q)}{P^2\,Q^2},\
S_{1,3}=\tilde{R}_{2,3}-\frac{\LM(0)(P+Q)}{P^2\,Q^2}.\]
Let $\lm_1,\cdots,\lm_n$ be defined by (\ref{zh-3}),
they are the critical values of the rational function $\tLM(P)$ and can serve as the
canonical coordinates of the Drinfeld - Sokolov bihamiltonian structure in the $D_n$ case.
So we have
\[D_n:\quad f^i=2\,r_i\tLM''(r_i),\
Q^{ii}_1=3 \tLM''(r_i)+\frac{14}3 r_i \tLM'''(r_i)+r_i^2 \tLM''''(r_i)-\frac{2\,\LM(0)}{r_i^3},\]
\[Q^{ii}_2=\tLM(r_i)Q^{ii}_1+r_i^2 \tLM''(r_i)^2,\ c_i=\frac{Q^{ii}_2-\tLM(r_i)\,Q^{ii}_1}{3(f^i)^2}=\frac1{12}.\]
The Theorem is proved.
\end{proof}

\section{The exceptional root systems}\label{except}

The hierarchies associated with the exceptional root systems have been systematically treated by V.Kac and M.Wakimoto in \cite{kw}.  They did not consider however the bihamiltonian structure of the exceptional hierarchies. In our computations we
will use
the approach of \cite{balog} based on the Dirac reduction procedure \cite{dirac}.
Let us consider the Poisson bracket $\pi_\g(I)$  on $\g^*$ evaluated at the point $I$ as a skew symmetric bilinear form
on
$$
\g \simeq T^*_I \g^*
$$
(cf. \eqref{stab}). The stabilizer ${\rm Ker}\, {\rm ad}_I$ of $I$ coincides with the kernel of this bilinear form. The quotient
$$
\g/{\rm Ker}\, {\rm ad}_I
$$
acquires a symplectic structure induced by $\pi_\g(I)$.
The projection
$$
\n \hookrightarrow \g \to \g/{\rm Ker}\, {\rm ad}_I
$$
realizes the
nilpotent subalgebra $\n$ as a Lagrangian subspace in the quotient. Let
\beq\label{ndual1}
\n_{\rm dual}\subset \h\oplus \n^-
\eeq
be a pullback of a complementary Lagrangian subspace of the image of $\n$ such that
\beq\label{ndual2}
\g={\rm Ker}\, {\rm ad}_I\oplus \n \oplus \n_{\rm dual}.
\eeq
A choice of $\n_{\rm dual}$ specifies the transversal subspace $V\subset \fb$ of \eqref{sub-v} by the equation
\beq\label{ndual3}
\la b, q^{\rm can}\ra_\g =0 \quad \forall \,b\in \n_{\rm dual}, \quad q^{\rm can}\in V.
\eeq
One can unify constraints \eqref{orbit} and \eqref{ndual3} by considering a system of equations for $q\in\g$:
\eqa\label{ndual4}
&&
\la a, q\ra_\g = \la a, I\ra_\g\quad \forall\, a\in \n
\nn\\
&&
\\
&&
\la b, q\ra_\g =0 \quad \forall\, b\in \n_{\rm dual}.
\nn
\eeqa
The solution
$$
q=I+q^{\rm can}
$$
determines the transversal slice $V$. The reduced Poisson bracket on\\ $q^{\rm can}$-valued loops can be obtained as follows. Let us choose a basis
$$
f_1, \dots, f_{2 m} \in \n\oplus\n_{\rm dual}, \quad 2 m=2\dim\n=\dim \g-n.
$$
Introduce two $2m  \times 2m$ matrices
\eqa\label{ndual5}
&&
P=(P_{ab}), \quad P_{ab}=-\la I+q^{\rm can}, [f_a, f_b]\ra_\g,
\nn\\
&&
\\
&&
Q=(Q_{ab}), \quad Q_{ab}=\la f_a, f_b\ra_\g.
\nn
\eeqa
By construction of $\n_{\rm dual}$ the matrix
$$
P|_{q^{\rm can}=0}=\pi_\g(I)|_{\n\oplus \n_{\rm dual}}
$$
does not degenerate. Consider matrix differential operator
\beq\label{ndual6}
M:= P +Q\,\e\, \p_x
\eeq
with coefficients depending on $q^{\rm can}$ (via $P$).
Note that the matrix of pairwise Poisson brackets of the constraints \eqref{ndual4}  is equal to
$$
\{ \la f_a, q(x)\ra_\g, \la f_b, q(y)\ra_\g\}=-\frac1{\e}\,M_{ab}\,\delta(x-y).
$$
The following statement was proved in \cite{feher}.

\begin{lem} The inverse $M^{-1}$ to \eqref{ndual6} is  a matrix valued differential operator of finite order with coefficients depending polynomially on $q^{\can}$, $q^{\can}_x$, \dots.
\end{lem}

Let
\beq\label{gamai1}
\gamma^1, \dots, \gamma^n\in {\rm Ker}\, {\rm ad}\, I
\eeq
be a basis in the centralizer of $I$. Recall \cite{kostant} that this centralizer is a commutative subalgebra in $\n^-$ having generators only in the degrees $-m_1$, \dots, $-m_n$; the number of generators in the degree $-m_k$ is equal to the multiplicity of the exponent $m_k$.
The linear functions
of $q^{\rm can}\in V$ given by
\beq\label{gamai2}
u^i =\la \gamma^i, q^{\rm can}\ra_\g, \quad i=1, \dots, n
\eeq
define a system of coordinates on $V$.
Denote $\gamma_1$, \dots, $\gamma_n$ the dual basis in $V$,
\beq\label{basisg}
\la \gamma^i, \gamma_j\ra_\g =\delta^i_j, \quad \la f_a, \gamma_i\ra_\g=0, \quad i, j =1, \dots, n, \quad a=1, \dots, 2m
\eeq
so
\beq\label{basisg1}
q^{\can}=\sum_{i=1}^n u^i \gamma_i.
\eeq
Introduce the $n\times 2m$ matrix differential operator
\beq\label{ndual7}
N=(N^i_{ a})=(R^i_a + S^i_a \epsilon \p_x), \quad N^i_{a}=\e\,\la \gamma^i, f_a\ra_\g\, \p_x - \la q^{\rm can}, [\gamma^i, f_a]\ra_\g.
\eeq
Denote $N^\dagger$ the matrix of (formally) adjoint differential operators,
\beq\label{ndual8}
\left(N^\dagger\right)^a_{i} ={N^i_{a}}^\dagger, \quad i=1, \dots, n, \quad a=1, \dots, 2m.
\eeq
%Put also
%\beq\label{ndual10}
%L=(L^{ij})= -\la I+q^{\rm can}, [\gamma^i, \gamma^j]\ra_\g.
%\eeq
Then the matrix of the second reduced Poisson bracket is given by the formula
\beq\label{ndual9}
 \{ u^i(x), u^j(y)\}^{\rm red}_2 =- \frac1{\e}\left(N\, M^{-1} N^\dagger\right)^{ij}\, \delta(x-y).
\eeq
The first reduced bracket is given by a similar formula
\beq\label{ndual11}
\begin{split}
&\{ u^i(x), u^j(y)\}_1^{\rm red}\\
=&\frac1{\e}\left(
N\, M^{-1} \tilde{M}\, M^{-1} N^\dagger+\tilde{N}\, M^{-1} \, N^\dagger
+ N\, M^{-1} \, \tilde{N}^\dagger\right)^{ij}\, \delta(x-y)
\end{split}
\eeq
where
the $n\times 2m$ and $2m\times 2m$ matrices  $\tilde{N}^i_a$ and $\tilde{M}_{ab}$ respectively are defined as follows:
\beq\label{perva}
%&&
%{\rm l}^{ij} =-\la \alpha, [\gamma^i, \gamma^j]\ra_\g
%\nn\\
\tilde{N}^i_a = \la \alpha, [\gamma^i, f_a]\ra_\g,\ \tilde{M}_{ab} = \la \alpha, [f_a, f_b]\ra_\g,
\eeq
where $\alpha\in\n$ is the generator of the center of $\n$ chosen above (see \eqref{centr}). We will see below that the terms of order $\e^{-1}$ disappear from \eqref{ndual9}, \eqref{ndual11}.

Let us now explain how we compute the Frobenius structure and the central invariants using the formula \eqref{ndual9}. For the second metric $g_2^{ij}$ one obtains
\beq\label{dir1}
\left(g_2^{ij}(q^{\can}) \right)= R\, P^{-1} Q\,P^{-1} R^T-S\,P^{-1}R^T+R\,P^{-1} S^T
\eeq
where $R^T, S^T$ denotes their transposed matrices.
The matrices $\left(A_{1,0; 2}^{ij}\right)$ and $\left(A_{2,0; 2}^{ij}\right)$ have the following form:
\eqa
&&
\left(A_{1,0; 2}^{ij}(q^{\can})\right) =-R\, P^{-1}Q\,P^{-1} Q\, P^{-1} R^T-R\,P^{-1}Q\,P^{-1}S^T
\nn\\
&& +S\,P^{-1}Q\,P^{-1}R^T+S\,P^{-1} S^T,\label{dir2}
\eeqa
\eqa
&&
\left(A_{2,0; 2}^{ij}(q^{\can})\right) =R\,P^{-1}Q\,P^{-1}Q\,P^{-1}Q\,P^{-1}R^T-R\,P^{-1}Q\,P^{-1}Q\,P^{-1}S^T
\nn\\
&&+
S\,P^{-1}Q\,P^{-1}Q\,P^{-1}R^T+S\,P^{-1}Q\,P^{-1}S^T,\label{dir3}
\eeqa
where the matrices $R=(R^i_a)$, $S=(S^i_a)$ are defined in \eqref{ndual7}. Doing the shift
\beq\label{ctn}
q^{\can}\mapsto q^{\can}+\lambda\,\alpha, \quad \alpha\in \{{\rm the\ center\ of}\ \n\}
\eeq
one obtains in \eqref{dir1} - \eqref{dir3} linear functions in $\lambda$. The coefficients of $\lambda$ of these functions give
the matrices $g_1^{ij}$, $A_{1,0; 1}^{ij}(q^{\can})$ and $A_{2,0; 1}^{ij}(q^{\can})$ respectively.

The dual bases $\gamma^i\in {\rm Ker} \ad I$ and $\gamma_i\in V$ can be chosen as follows.  According to \cite{kostant} the triple
\beq\label{sl1}
I_-:=I, \quad \rho=\sum_{i=1}^n\omega_i, \quad I_+ =\sum_{i=1}^n {a_i} \, X_i
\eeq
defines an embedding of the $sl_2$ Lie algebra into $\g$,
\beq\label{sl0}
[I_+, I_-]=2\rho, \quad [\rho, I_\pm ] =\pm I_\pm.
\eeq
Here $\omega_1, \dots, \omega_n\in \h$ are the fundamental weights, i.e.
the basis dual to the basis of simple roots, and
the integer coefficients $a_1, \dots, a_n$ are defined from the decomposition
\beq\label{sl2}
2\rho =\sum_{i=1}^n a_i H_i.
\eeq
We put
\beq\label{sl3}
V:={\rm Ker}\ad I_+.
\eeq
We choose
\eqa\label{sl4}
&&
\gamma_i\in {\rm Ker} \ad {I_+}\cap \g^{m_i}, \quad i=1, \dots, n
\nn\\
&&
\\
&&
\gamma^i\in {\rm Ker} \ad {I_-}\cap \g^{-m_i}, \quad i=1, \dots, n
\nn
\eeqa
For all exceptional Lie algebras the vectors $\gamma_i$ and $\gamma^i$ are determined uniquely up to normalization. We can normalize them in such a way that
$$
\la\gamma^i, \gamma_j\ra_\g =\delta^i_j.
$$

Denote
\beq\label{dir4}
p(z; q^{\can})=\det \left[ g^{ij}_2 (q^{\can})-z\, g^{ij}_1 (q^{\can})\right]
\eeq
the characteristic polynomial of the pair of quadratic forms
$g_2^{ij}$, $g_1^{ij}$. The roots $z^1(q^{\can})$, \dots, $z^n(q^{\can})$ will be used as the canonical coordinates of the pair of metrics: in these coordinates both metrics become diagonal:
\beq\label{dir5}
\begin{split}
&\sum_{k, l=1}^n \left(\frac{\p p(z; q^{\can})}{\p u^k}\right)_{z=z^i}
\left(\frac{\p p(z; q^{\can})}{\p u^l}\right)_{z=z^j} g^{kl}_1(q^{\can})=0, \quad i\neq j,\\
&
\sum_{k, l=1}^n \left(\frac{\p p(z; q^{\can})}{\p u^k}\right)_{z=z^i}
\left(\frac{\p p(z; q^{\can})}{\p u^l}\right)_{z=z^j} g^{kl}_2(q^{\can})=0, \quad i\neq j.
\end{split}
\eeq
Here we used the implicit function theorem formula
\beq
\frac{\p z^i(q^{\can})}{\p u^k}=-\left(\frac1{p'(z; q^{\can})}\frac{\p p(z; q^{\can})}{\p u^k}\right)_{z=z^i},
\ p'(z; q^{\can}) =\frac{\p p(z; q^{\can})}{\p z}.\nn
\eeq
For the central invariants one obtains the following expressions:
\eqa\label{otvet}
&&
c_i=\frac13\left[ p'(z^i; q^{\can})\right]^2
\\
&&
\nn\\
&&\left.\times\frac{\sum_{k, l=1}^n \left(\frac{\p p(z; q^{\can})}{\p u^k}\right)
\left(\frac{\p p(z; q^{\can})}{\p u^l}\right)
\left(A^{kl}_{2,0;2}(q^{\can}) -z \, A^{kl}_{2,0;1}(q^{\can})\right)}
{\left[\sum_{k, l=1}^n \left(\frac{\p p(z; q^{\can})}{\p u^k}\right)
\left(\frac{\p p(z; q^{\can})}{\p u^l}\right) g^{kl}_1(q^{\can})\right]^2}\right|_{z=z^i}.
\nn
\eeqa

We will now consider the $G_2$, $F_4$ and $E_6 - E_8$ cases. The computer supported calculations for these Lie
algebras of types $F$ and $E$ use an explicit realization of the Chevalley basis \cite{bour1}.
The Chevalley bases we describe in this paper are generated by the computer
algebra system GAP \cite{GAP}. For convenience of the reader we give an explicit matrix realization of the bases choosing a faithful representation of the Lie algebra of the minimal dimension. The source program is as follows:

\begin{verbatim}
    PARAMETERS:
            X: The type of the Lie algebras
            r: The rank of the Lie algebras
            w: The fundamental weight with the minimum dimension
            f: The name of the file that stores the result

    PROGRAM:
            L:=SimpleLieAlgebra("X",r,Rationals);;
            V:=HighestWeightModule(L,w);;
            ll:=Basis(L);;
            vv:=Basis(V);;
            mm:=List(ll, x->MatrixOfAction(vv,x));;
            PrintTo("f",mm);;
\end{verbatim}
The parameter \verb|w| takes the following values:
\begin{align*}
F_4:& \quad [1,0,0,0]\\
E_6:& \quad [0,0,0,0,0,1]\\
E_7:& \quad [0,0,0,0,0,0,1]\\
E_8:& \quad [0,0,0,0,0,0,0,1].
\end{align*}
For the Lie algebras of type E, the results of the above program are
excatly what we need, while for that of type F, we have to modify the signs
of some base elements in order to satisfy our normalizing rules.

\begin{lem}
For the exceptional Lie simple Lie algebras of type $G_2$, $F_4$, $E_6$, $E_7$, $E_8$, the central
invariants of the corresponding Drinfeld - Sokolov bihamiltonian structures coincide with the
values listed in the table that is given at the end of Section \ref{principale}.
\end{lem}
\begin{proof} The lemma can be proved by a straightforward computation by using the
formula \eqref{fc} for the Lie algebras of $G_2$ and $F_4$ types.
For the E type case we can use the
formula \eqref{otvet} to compute the central invariants,
however the computations become very involved; so we use a different method based on a comparison of the
Drinfeld - Sokolov bihamiltonian structure with the one obtained in \cite{DZ1} (see below).

Since the central invariants do not depend on the choice
of $\al$ in \eqref{ctn}, in what follows we will fix $\al=\gamma_n$.

We first illustrate the procedure by considering the $G_2$
case\footnote{Explicit formulae for the $G_2$ bihamiltonian structure were obtained in the original paper \cite{DS}. In \cite{feher} they have been rederived using the Dirac reduction procedures.} in detail.

Let $X_i, H_i, Y_i\ (i=1,2)$ be a set of Weyl generators of the simple Lie algebra $\g$ of $G_2$ type,
whose Dynkin diagram is labelled as follow
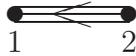
\begin{center}
\setlength{\unitlength}{1mm}
\begin{picture}(16,7)
\put(1,5){\circle*{2}} \put(16,5){\circle*{2}} \put(0,0){$1$} \put(15,0){$2$}
\put(1,4){\line(1,0){15}} \put(1,5){\line(1,0){15}}
\put(1,6){\line(1,0){15}} \put(6,5){\line(3,1){5}}
\put(6,5){\line(3,-1){5}}
\end{picture}
\end{center}
We define a Chevalley basis of $\g$
\begin{align*}
&X_3=-[X_1, X_2],\ Y_3=[Y_1, Y_2], \\
&X_4=-[X_1, X_3]/2,\ Y_4=[Y_1, Y_3]/2, \\
&X_5=-[X_1, X_4]/3,\ Y_5=[Y_1, Y_4]/3, \\
&X_6=-[X_2, X_5],\ Y_6=[Y_2, Y_5].
\end{align*}
The normalized invariant bilinear form is given by
\begin{equation}\label{bf-G2}
\begin{split}
&\la X_1, Y_1\ra_{\g}=\la X_3, Y_3\ra_{\g}=\la X_4, Y_4\ra_{\g}=3,\\
&\la X_2, Y_2\ra_{\g}=\la X_5, Y_5\ra_{\g}=\la X_6, Y_6\ra_{\g}=1,\\
&\la H_1, H_1\ra_{\g}=6,\ \la H_1, H_2\ra_{\g}=-3,\ \la H_2, H_2\ra_{\g}=2.
\end{split}
\end{equation}

The elements $\rho, I_+$ read
\[\rho=3H_1+5H_2,\ I_+=6X_1+10X_2.\]
We choose a basis of $\Ker\ad I_+$
\[\gamma_1=\frac35 X_1+X_2,\ \gamma_2=X_6.\]
Then we can obtain the result of the Dirac reduction:
\begin{align*}
&\left(g_2^{ij}\right)=\left(\begin{array}{ll} -\frac{5 u_1}{7} & -\frac{15 u_2}{7} \\
-\frac{15 u_2}{7} & -\frac{768 u_1^5}{875}-\frac{1144}{525} u_2 u_1^2\end{array}\right),\
\left(g_1^{ij}\right)=\left(\begin{array}{ll} 0 & -\frac{15}{7} \\
-\frac{15}{7} & -\frac{1144 u_1^2}{525}\end{array}\right),\\
&\left(A_{2,0,2}^{ij}\right)=\left(\begin{array}{ll} \frac{25}{14} & 0 \\
 0 & \frac{42152 u_1^4}{13125}+\frac{62 u_2 u_1}{21}\end{array} \right),\
\left(A_{2,0,1}^{ij}\right)=\left(\begin{array}{ll} 0 & 0 \\ 0 & \frac{62 u_1}{21}\end{array}\right),
\end{align*}
and $A_{1,0,2}^{ij}=A_{1,0,1}^{ij}=0.$

If we introduce the flat coordinates
\[t_1=u_2-\frac{572 u_1^3}{3375},\ t_2=-\frac{7 u_1}{15},\]
the above metrics are just the flat pencil defined by the following Frobenius manifold
\[F=\frac{1}{2}t_1^2t_2+\frac{24}{35}t_2^7,\ E=t_1\frac{\p}{\p t_1}+\frac13t_2\frac{\p}{\p t_2}.\]

In the flat coordinates, we have
\begin{align*}
&\left(g_2^{ij}\right)=\left(\begin{array}{ll} 48 t_2^5 & t_1 \\ t_1 & \frac{t_2}{3}\end{array}\right),\
\left(g_1^{ij}\right)=\left(\begin{array}{ll} 0 & 1 \\ 1 & 0\end{array}\right),\\
&\left(A_{2,0,2}^{ij}\right)=\left(\begin{array}{ll} -\frac{310 t_2}{49} & 0 \\ 0 & 0\end{array}\right),\
\left(A_{2,0,1}^{ij}\right)=\left(\begin{array}{ll} 88 t_2^4-\frac{310 t_1 t_2}{49} & \frac{286 t_2^2}{147}
\\ \frac{286 t_2^2}{147} & \frac{7}{18}\end{array}\right)
\end{align*}

The canonical coordinates are
\[\lambda_1=t_1+4t_2^3,\ \lambda_2=t_1-4t_2^3,\]
from which we can compute the quantities appeared in the formula \eqref{fc}
\begin{align*}
&f^1=24 t_2^2,\ f^2=-24 t_2^2,\\
&Q_2^{11}=\frac{9344 t_2^4}{49}-\frac{310 t_1 t_2}{49},\ Q_2^{22}=\frac{4768 t_2^4}{49}
-\frac{310 t_1 t_2}{49},\\
&Q_1^{11}=-\frac{310 t_2}{49},\ Q_1^{22}=-\frac{310 t_2}{49}.
\end{align*}
So the central invariants are given by
\[c_1=\frac{Q_2^{11}-\lambda_1Q_1^{11}}{3(f^1)^2}=\frac1{8},\ c_2=\frac{Q_2^{22}-\lambda_2Q_1^{22}}{3(f^2)^2}=\frac1{24}.\]

\noindent{\bf The $F_4$ case}

The root system of type $F_4$
contains $24$ positive roots, it's not convenient to define
the Chevalley basis explicitly, so we use an alternative way below to describe this basis.

The simple Lie algebra of type $F_4$ has a $26$-dimensional matrix realization, whose Weyl generators are
\begin{align*}
X_1=&e_{1,2}+e_{6,8}+e_{7,10}+e_{9,12}+2e_{11,13}+e_{11,14}\\
&\quad+e_{13,16}+e_{15,18}+e_{17,20}+e_{19,21}+e_{25,26},\\
X_2=&e_{4,5}+e_{6,7}+e_{8,10}-e_{17,19}-e_{20,21}-e_{22,23},\\
X_3=&e_{2,3}-e_{4,6}-e_{5,7}-e_{9,11}-e_{12,13}-2e_{12,14}\\
&\quad-e_{14,15}-e_{16,18}+e_{20,22}+e_{21,23}-e_{24,25},\\
X_4=&e_{3,4}-e_{7,9}-e_{10,12}+e_{15,17}+e_{18,20}+e_{23,24},
\end{align*}
\begin{align*}
Y_1=&e_{2,1}+e_{8,6}+e_{10,7}+e_{12,9}+e_{13,11}+2e_{16,13}\\
&\quad+e_{16,14}+e_{18,15}+e_{20,17}+e_{21,19}+e_{26,25},\\
Y_2=&e_{5,4}+e_{7,6}+e_{10,8}-e_{19,17}-e_{21,20}-e_{23,22},\\
Y_3=&e_{3,2}-e_{6,4}-e_{7,5}-e_{11,9}-e_{14,12}-e_{15,13}\\
&\quad-2e_{15,14}-e_{18,16}+e_{22,20}+e_{23,21}-e_{25,24},\\
Y_4=&e_{4,3}-e_{9,7}-e_{12,10}+e_{17,15}+e_{20,18}+e_{24,23}.
\end{align*}
These generators correspond the following Dynkin diagram
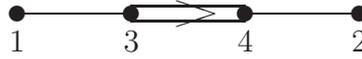
\begin{center}
\setlength{\unitlength}{1mm}
\begin{picture}(47,7)
\put(1,5){\circle*{2}} \put(16,5){\circle*{2}} \put(31,5){\circle*{2}} \put(46,5){\circle*{2}}
\put(0,0){$1$} \put(15,0){$3$} \put(30,0){$4$} \put(45,0){$2$}
\put(1,5){\line(1,0){15}} \put(16,6){\line(1,0){15}} \put(16,4){\line(1,0){15}} \put(31,5){\line(1,0){15}}
\put(27,5){\line(-3,1){5}}
\put(27,5){\line(-3,-1){5}}
\end{picture}
\end{center}
The normalized Killing form can be computed by the following formula
\[\la A, B \ra_{\g}=\frac16\tr(A\,B).\]

Let $\alpha_i$ be the simple root corresponding to $X_i$, $i=1, \cdots, 4$.
For any positive root $\beta \in \Phi^+$ of the form
\[\beta=\sum_{i=1}^4n_i\,\alpha_i, \mbox{ where } n_i\ge0,\ i=1, \cdots, 4,\]
we define $X_{\beta}=X_{n_1, \cdots, n_4}$ (resp. $Y_{\beta}=Y_{n_1, \cdots, n_4}$) to be the matrix
in the root space $\g_\beta$ (resp. $\g_{-\beta}$) such that the first nonzero element of the first
nonzero row (resp. column) is equal to $1$. Since $\dim\g_{\pm\beta}=1$, $X_\beta, Y_\beta$ are fixed in this
way uniquely. By a straightforward calculation, one can show that
$$
\{H_i, X_\beta, Y_\beta\, |\, i=1,\dots, 4,\, \beta\in\Phi^+\}
$$
form a Chevalley basis, and the element $\rho$ is given by
\[\rho=\frac12\sum_{\beta\in\Phi^+}[X_\beta, Y_\beta].\]

The element $I_+$ now reads
\[I_+=16\,X_1+22\,X_2+30\,X_3+42\,X_4.\]
We fix a basis $\{\gamma_i\}_{i=1}^4$ of $V=\Ker\ad {I_+}$ as follows:
\begin{align*}
\gamma_1=&X_{0001}+\frac{5}{7}X_{0010}+\frac{11}{21}X_{0100}+\frac{8X_{1000}}{21},\ \gamma_4=X_{2243},\\
\gamma_2=&X_{0122}-\frac{8}{21}X_{1121}+\frac{128}{231}X_{2021},\ \gamma_3=X_{2122}+\frac{15}{8}X_{1132}.
\end{align*}

By using the formulae given at the beginning of the present section, we can compute the reduced Poisson brackets w.r.t.
the above basis. To present the result, we introduce the following flat coordinates
\begin{align*}
t_1=&u_4-\frac{762841u_1^6}{49009212}-\frac{129973u_2u_1^3}{259308}-\frac{2783u_3u_1^2}{3528}-\frac{56741u_2^2}{142296},\\
t_2=&\frac{1781u_1^4}{64827}+\frac{34u_2u_1}{231}+u_3,\ t_3=\frac{4199u_1^3}{63504}+\frac{4199u_2}{3696},\
t_4=-\frac{13u_1}{42}.
\end{align*}
Then the two metrics given by the coefficients of the leading terms of the
reduced Poisson brackets correspond to the flat pencil of metric of the Frobenius manifolds with potential
\begin{align*}
F=&\frac{1}{2}t_1^2t_4+t_1t_2t_3+\frac{20736t_4^{13}}{143}+\frac{82944t_3^2t_4^7}{2527}+\frac{1083}{20}t_2^2t_4^5\\
&+\frac{288}{19}t_2t_3^2t_4^3+\frac{27648t_3^4t_4}{130321}+\frac{6859t_2^3t_4}{1152},
\end{align*}
its Euler vector field is
\[E=\sum_{i=1}^4 E^i\frac{\p}{\p t_1}=t_1 \frac{\p}{\p t_1}+\frac{2t_2}{3}\frac{\p}{\p t_2}
+\frac{t_3}{2}\frac{\p}{\p t_3}+\frac{t_4}{6}\frac{\p}{\p t_4}.\]
In the coordinates $t_i$, the first metric $g_1^{ij}$ given by the coefficients of the leading terms of
the first Poisson structure has the standard expression \cite{tani}
\beq\label{metric-g1}
(g^{ij}_1)=(\eta_{ij})^{-1},\ \eta_{ij}=\p_{t_i}\p_{t_j} Lie_e F,
\eeq
where the unity vector field $e$ is given by
\beq
e=\frac{\p}{\p t_1}.
\eeq
The second metric $g^{ij}_2$ satisfies the formula
\beq\label{metric-g2}
g^{ij}_2(t)=\sum_{m=1}^n E^m c^{ij}_m(t), \mbox{ with }c^{ij}_m(t)=g^{ik}_1 g^{jl}_1
\p_{t_m}\p_{t_k}\p_{t_l} F.
\eeq
Here $n=4$.

The coefficients $A^{ij}_{1,0;a}\,(a=1,2)$ of the reduced Poisson brackets are equal to zero.
The coefficients $A^{ij}_{2,0;2}$ read
{\scriptsize
\begin{align*}
A^{11}_{2,0,2}=&238464\, t_4^{10}-\frac{79854336\, t_3\, t_4^7}{4693}+\frac{362769128\, t_2\, t_4^6}{37349}
+\frac{82248768000\, t_3^2\, t_4^4}{13482989}\\
&+\frac{65740256\, t_1\, t_4^4}{371293}-\frac{286440}{247}\, t_2\, t_3\, t_4^3+\frac{6443534125\, t_2^2\, t_4^2}{2689128}
-\frac{53236224\, t_3^3\, t_4}{1694173}\\
&-\frac{4015872\, t_1\, t_3\, t_4}{54587}+\frac{1656}{19}\, t_2\, t_3^2+\frac{443\, t_1\, t_2}{26},\\
A^{12}_{2,0,2}=&-\frac{15818112\, t_4^8}{4693}+\frac{42634554624\, t_3\, t_4^5}{13482989}-\frac{6453151372\, t_2\, t_4^4}{21163701}
-\frac{51777792\, t_3^2\, t_4^2}{1694173}\\
&-\frac{28255104\, t_1\, t_4^2}{709631}+\frac{7349328\, t_2\, t_3\, t_4}{54587}+\frac{153\, t_2^2}{13},\\
A^{13}_{2,0,2}=&\frac{204693422\, t_4^7}{37349}-\frac{5205718984\, t_3\, t_4^4}{7054567}+\frac{9722937545\, t_2\, t_4^3}{5378256}
+\frac{3133152\, t_3^2\, t_4}{54587}\\&+\frac{79\, t_1\, t_4}{4}-\frac{3611\, t_2\, t_3}{312},\\
A^{14}_{2,0,2}=&\frac{16435064\, t_4^5}{1113879}+\frac{14507020\, t_3\, t_4^2}{709631}-\frac{2783\, t_2\, t_4}{312},\\
A^{22}_{2,0,2}=&\frac{13824\, t_4^6}{19}-\frac{3170304\, t_3\, t_4^3}{89167}+\frac{4883336\, t_2\, t_4^2}{125229}
+\frac{13824\, t_3^2}{6859}+\frac{2400\, t_1}{4693},\\
A^{23}_{2,0,2}=&-\frac{2508\, t_4^5}{13}+\frac{197596\, t_3\, t_4^2}{2197}+\frac{817\, t_2\, t_4}{312},\\
A^{24}_{2,0,2}=&\frac{39412\, t_4^3}{6591}-\frac{56\, t_3}{247},\
A^{34}_{2,0,2}=-\frac{2261\, t_4^2}{624},\ A^{44}_{2,0,2}=\frac{13}{24},
\end{align*}}
and the coefficients $A^{ij}_{2,0;1}$ is given by
\begin{equation}
A_{2,0;1}^{ij}(t)=\frac{\p}{\p {t}_1} A^{ij}_{2,0;2}(t). \label{a201}
\end{equation}

Now we begin to compute the central invariants. We first find the canonical coordinates from the characteristic equation
$\det(g^{ij}_2-\lambda g^{ij}_1)=0$. The roots can be represented in the form
\begin{align*}
\lambda_{\mu_1,\mu_2}=&\left(t_1+\frac{288}{19}t_3t_4^3\right)+\mu_1\left(\frac{57}{2}t_2t_4^2+\frac{288}{361}t_3^2\right)\\
& \quad +\mu_2\frac{\left(361t_2+\mu_1576t_3t_4+2736t_4^4\right)^{\frac32}}{228\sqrt{57}},
\end{align*}
where $\mu_1, \mu_2=\pm1$. We number them in the way that
\[\lambda_1=\lambda_{++},\ \lambda_2=\lambda_{+-},\ \lambda_3=\lambda_{-+},\ \lambda_4=\lambda_{--}.\]
We then compute the metrics $g_1$, $g_2$ and the functions
$A_{2,0;1}, A_{2,0;1}$ in the canonical coordinates. After a
straightforward computation, we obtain the central invariants from
the formula \eqref{fc}, they read
\[\{c_1, c_2, c_3, c_4\}=\{\frac1{24},\frac1{24},\frac1{12},\frac1{12}\},\]
which proves the lemma for the $F_4$ case.
\vskip 0.5cm

\noindent{\bf The $E_6$ case}

The proof of the lemma for the simple Lie algebras of $E$ types
are similar to that of the $F_4$ case. We take
$E_6$ for example. It has a $27$-dimensional matrix realization, the
Weyl generators are realized as
\begin{align*}
&X_1=e_{6, 7} + e_{8, 9} + e_{10, 11} + e_{12, 14} + e_{15, 17} + e_{26, 27},\\
&X_2=e_{4, 5} + e_{6, 8} + e_{7, 9} - e_{18, 20} - e_{21, 22} - e_{23, 24},\\
&X_3=e_{4, 6} + e_{5, 8} + e_{11, 13} + e_{14, 16} + e_{17, 19} + e_{25, 26},\\
&X_4=e_{3, 4} - e_{8, 10} - e_{9, 11} - e_{16, 18} - e_{19, 21} + e_{24, 25},\\
&X_5=e_{2, 3} - e_{10, 12} - e_{11, 14} - e_{13, 16} + e_{21, 23} + e_{22, 24},\\
&X_6=e_{1, 2} + e_{12, 15} + e_{14, 17} + e_{16, 19} + e_{18, 21} + e_{20, 22},
\end{align*}
and $Y_i={X_i}^T,\ i=1, \cdots, 6$.
The Dynkin diagram for these generators are given by
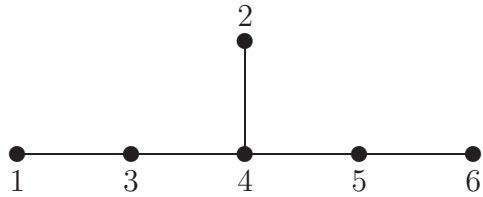
\begin{center}
\setlength{\unitlength}{1mm}
\begin{picture}(62,25)
\put(1,5){\circle*{2}} \put(16,5){\circle*{2}}
\put(31,5){\circle*{2}} \put(46,5){\circle*{2}}
\put(61,5){\circle*{2}} \put(31,20){\circle*{2}}
\put(1,5){\line(1,0){15}} \put(16,5){\line(1,0){15}} \put(31,5){\line(1,0){15}}
\put(46,5){\line(1,0){15}} \put(31,5){\line(0,1){15}}
\put(0,0){$1$} \put(30,22){$2$} \put(15,0){$3$} \put(30,0){$4$} \put(45,0){$5$} \put(60,0){$6$}
\end{picture}
\end{center}
The normalized Killing form is
\[\la A, B \ra_{\g}=\frac16\tr(A\,B).\]

The Chevalley basis is defined in the same way. The element $I_+$ reads
\[I_+=16\,X_1+22\,X_2+30\,X_3+42\,X_4+30\,X_5+16\,X_6.\]
The basis $\{\gamma_i\}_{i=1}^6$ of $V=\Ker\ad {I_+}$ is chosen as
\begin{align*}
&\gamma_1=X_1+\frac{11}{8}X_2+\frac{15}{8}X_3+\frac{21}{8}X_4+\frac{15}{8}X_5+X_6,\\
&\gamma_2=X_{001111}-\frac{11}{15} X_{010111}+X_{101110}+\frac{11}{15}X_{111100},\\
&\gamma_3=X_{011111}-\frac{21}{8} X_{011210}-\frac{16}{11}X_{101111}-X_{111110},\\
&\gamma_4=X_{011221}+\frac{8}{15}X_{111211}+X_{112210},\\
&\gamma_5=X_{111221}+X_{112211},\ \gamma_6=X_{122321}.
\end{align*}

The flat coordinates have the expressions
\begin{align*}
t_1=&\frac{5339887u_1^6}{84934656}+\frac{129973u_3u_1^3}{442368}-\frac{2783u_4u_1^2}{77760}\\
&+\frac{1679u_2^2u_1}{24300}+\frac{56741u_3^2}{1672704}+\frac{u_6}{81},\\
t_2=&\frac{4}{15}u_2u_1^2+\frac{2u_5}{27},\ t_3=\frac{33839u_1^4}{147456}+\frac{2261u_3u_1}{12672}-\frac{38u_4}{405},\\
t_4=&\frac{1547u_1^3}{9216}+\frac{221u_3}{528},\ t_5=\frac{52u_2}{135},\ t_6=\frac{13u_1}{8}.
\end{align*}
The potential of the corresponding Frobenius manifold is given by
\begin{align}
F=&-\frac{3^8}{2}\left(\frac{1}{2}t_1^2t_6+t_1t_2t_5+t_1t_3t_4+\frac{t_6^{13}}{185328}+\frac{1}{576}t_5^2t_6^8
+\frac{1}{252}t_4^2t_6^7\right.\nn\\
&+\frac{1}{60}t_3^2t_6^5+\frac{1}{24}t_4t_5^2t_6^5+\frac{1}{24}t_2^2t_6^4+\frac{1}{24}t_3t_5^2t_6^4
+\frac{1}{24}t_5^4t_6^3+\frac{1}{6}t_3t_4^2t_6^3\nn\\
&+\frac{1}{6}t_2t_4t_5t_6^3+\frac{1}{4}t_4^2t_5^2t_6^2+\frac{1}{2}t_2t_3t_5t_6^2+\frac{1}{12}t_4^4t_6
+\frac{1}{6}t_3^3t_6+\frac{1}{6}t_2t_5^3t_6\nn\\
&\left.+\frac{1}{2}t_3t_4t_5^2t_6+\frac{1}{2}t_2^2t_4t_6+\frac{1}{12}t_4t_5^4+\frac{1}{4}t_3^2t_5^2
+\frac{1}{2}t_2^2t_3+\frac{1}{2}t_2t_4^2t_5\right). \label{f-e6}
\end{align}
Note that the function $-\frac2{3^8} F(t)$ was obtained
as polynomial solutions of the WDVV equations associated to the root systems of type $E_6$ by
Di Francesco {\sl et al} in \cite{zuber}. Polynomial solutions to the WDVV equations associated to the root
systems of type $E_7$ and $E_8$ are also computed in \cite{zuber}.

The Euler vector field and the unity vector field have the forms
\eqa
&&E=\sum_{k=1}^6E^k\frac{\p}{\p t_k}\nn\\
 &&\quad =t_1\frac{\p}{\p t_1}+\frac34t_2\frac{\p}{\p t_2}+\frac23t_3\frac{\p}{\p t_3}
+\frac12t_4\frac{\p}{\p t_4}+\frac5{12}t_5\frac{\p}{\p t_5}+\frac16t_6\frac{\p}{\p t_6},\label{ev}\\
&&e=\frac1{81} \frac{\p}{\p t_1}.\label{uv}
\eeqa

The two flat metrics $g_1$, $g_2$ are expressed by the formulae given in \eqref{metric-g1}, \eqref{metric-g2}.
We will not write down the explicit expression of the functions $A_{2,0;1}^{ij}$, $A_{2,0;2}^{ij}$,
since in this case they are quite long.
As a consequence of this fact, the computation of the central invariants by using the formula \eqref{otvet}
becomes rather tedious. To avoid this complexity, we employ an alternative way to prove the result that
the central invariants of the Drinfeld - Sokolov bihamiltonian structure related to the $E_6$
(also for $E_7, E_8$) type simple Lie algebra are equal to $\frac1{24}$.

Our approach is to establish,
through an appropriate Miura type transformation, a relationship of the present bihamiltonian structure to
the one defined by a semisimple Frobenius manifold via the formulae of Theorem 1 and Theorem 2
of \cite{DZ1}. Then the needed result follows
if we can prove that the central invariants of the bihamiltonian structure given by
Theorem 1 and Theorem 2 of \cite{DZ1} are equal to $\frac1{24}$. This fact can be proved
by using properties of a semisimple Frobenius manifold. In fact, by using the formulae (3.9), (3.14), (3.15), (5.24)
of \cite{DZ1} we can express the functions $f^i$, $Q^{ii}_1$, $Q^{ii}_2$,
$P^{ki}_1, P^{ki}_2$ that appear in \eqref{fc}
 as follows:
\eqa
&&f^i=\frac1{\psi_{i1}^2},\quad P^{ki}_1=P^{ki}_2=0,\nn\\
&& Q_1^{ii}=\frac1{12} \sum_{j=1}^n \left(\frac{\gamma_{ij}}{\psi_{i1}^3 \psi_{j1}}+\frac{\gamma_{ij}\psi_{j1}}
{\psi_{i1}^5}\right),\nn\\
&&Q_2^{ii}=\frac1{24}\left[
\frac1{3\psi_{i1}^4}+2 \sum_{j=1}^n\left(\frac{\lm_i \gamma_{ij}}{\psi_{i1}^3 \psi_{j1}}+\frac{\lm_i
\gamma_{ij}\psi_{j1}}{\psi_{i1}^5}\right)\right].\nn
\eeqa
Here $n$ is the dimension of the semisimple Frobenius manifold, $\lm_1,\dots,\lm_n$ are its canonical coordinates,
the functions $\gamma_{ij}$ are the rotation coefficients of the flat metric of the Frobenius manifold,
and the functions $\psi_{i1}$ are defined by (4.5) of \cite{DZ1}. Note that in \cite{DZ1} we denote the
canonical coordinates by $u_1,\dots, u_n$. By plugging the above expressions into the formula \eqref{fc}
we immediately obtain the result $c_i=\frac1{24}, i=1,\dots, n$.

Now let us assume that the needed Miura type transformation has the form
\[\tilde{t}_i=t_i-\e^2 (\sum_{m} K^i_{m}\,t_{m,xx}+\sum_{k,l}M^{i}_{kl}\, t_{k,x} t_{l,x}),\ i=1,\dots 6,\]
A straightforward computation shows that there is a unique choice of the $(1,1)$ tensor $K^i_j$ with
the following nonzero components:
\begin{align*}
&K_3^1=\frac{92\, t_6}{247},\ K_4^1=\frac{1172287\, t_6^2}{2016846},\ K_5^1=\frac{3197\, t_5}{4056},\ K_6^4=\frac{17 \, t_6}{26},\\
&K_5^2=\frac{460\, t_6}{507},\ K_6^2=\frac{502\, t_5}{507},\ K_4^3=\frac{19}{39},\ K_6^3=\frac{47120\, t_6^2}{59319},\\
&K_6^1=\frac{2054383\, t_6^4}{60149466}+\frac{7521\, t_4\, t_6}{5746}+\frac{115\, t_3}{247}.
\end{align*}
such that in the new coordinates, the coefficients of $\e^2 \delta'''(x-y)$ of our reduced
bihamiltonian structure can be expressed, in terms of the potential
$F({\tilde{t}}\,)=\left.F(t)\right|_{t\mapsto \tilde t}$ given
in \eqref{f-e6}, by the following formulae of Theorem 1 and Theorem 2 of \cite{DZ1}:
\eqa
&&A^{ij}_{2,0;1}(\tilde{t}\,)=\frac{1}{12} \p_{t_k} (g_1^{kl} c^{ij}_l),\\
&&A^{ij}_{2,0;2}(\tilde{t}\,)=\frac{1}{12}\left(\p_{t_k} (g_2^{kl} c^{ij}_l)+\frac12\,c^{kl}_l c^{ij}_k\right),
\eeqa
where
$g^{ij}_1(\tilde{t}\,), g^{ij}_2(\tilde{t}\,), c^{ij}_k(\tilde{t}\,)$ are defined as in \eqref{metric-g1},
\eqref{metric-g2} by using the function $F(t)$ and the vector fields $\eqref{ev}, \eqref{uv}$, and then
replacing $t$ by $\tilde{t}$.

Since in the present case the central invariants are determined by the coefficients of $\e^2 \delta'''(x-y)$
of the bihamiltonian structure, the above Miura type transformation (with arbitrary chosen functions $M^i_{kl}$)
already establishes the fact that all the central invariants of the bihamiltonian structure that we are considering
are equal to $\frac1{24}$.

For the simple Lie algebra of type $E_7$, $E_8$, we give the relevant data
in the Appendix A and B, the notations are
in agreement with that of the above $E_6$ case.
We thus complete the proof of the lemma.
\end{proof}

\section{Conclusion}

In this paper, we compute the central invariants of the bihamiltonian structures of Drinfeld - Sokolov reduction
related to the affine Kac-Moody algebras of type $A_n^{(1)}$, $B_n^{(1)}$, $C_n^{(1)}$, $D_n^{(1)}$,
$G_2^{(1)}$, $F_4^{(1)}$, $E_{6,\, 7, \, 8}^{(1)}$ with the standard gradation which is given by the vertex $c_0$ in the (extended) Dynkin diagram.

For the standard gradations defined by another vertex, Drinfeld and Sokolov didn't give the bihamiltonian structures.
We point out that the generalized KdV equations for other standard gradations do possess bihamiltonian structures of
the form \eqref{bhs}, but in general these bihamiltonian structures have infinite many terms.
This is because all these equations
are related to the generalized mKdV equations through a Miura type transformation, while these transformations are
invertible in the formal power series sense. This fact has an immediate corollary that the central invariants of these
bihamiltonian structures are the same with the ones we have computed.

The generalized KdV equations related to the {\it twisted} affine Lie algebras seem not to possess a bihamiltonian structure.
We give here a counterexample.

Let us consider the generalized KdV equation related to $A_2^{(2)}$ equipped with the
standard gradation defined by the vertex $c_0$. The simplest integrable equation reads \cite{DS}
\begin{equation}\label{a22}
u_t=5\,u^2\,u_x+5\,\epsilon^2 (u_x\,u_{xx}+u\,u_{xxx})+\epsilon^4 u_{xxxxx}.
\end{equation}

\begin{prop} The equation \eqref{a22} possesses only one local Hamiltonian structure found in \cite{DS}
\begin{eqnarray}\label{a22-h}
&&
u_t=\{ u(x), H\}, \quad H=\int ( u^3-3 u_x^2)\, dx
\nn\\
&&
\\
&&
\{ u(x), u(y)\} = 2u(x) \delta'(x-y) +u_x \delta(x-y) +\frac{\epsilon^2}2 \delta'''(x-y).
\nn
\end{eqnarray}
\end{prop}

The proof was obtained in \cite{LWZ} following the scheme of \cite{LZ2}. Let us give here the sketch of the proof.
First, we construct the so-called {\it quasitriviality transformation} $v\mapsto u$,
\begin{align}
&u= v+\frac{\e^2}{4}\p_x^2\left(\log v_1-\log v\right)+\nn\\
&+\e^4
\p_x^2\left(\frac{32\,{{v}_2}}{{{v}}^2}-\frac{27\, {v}_1^2}{{{v}}^3}-\frac{12 \,{v_3}}{{{v}} {{v}_1}}+\frac{5\, {v_4}}{{v}_1^2}
+\frac{11\,{v}_2^2}{{{v}}
{v}_1^2}-\frac{21\, {v_3} \ {{v}_2}}{{v}_1^3}+\frac{16\, {v}_2^3}{{v}_1^4}
\right)\nn\\
&+\e^6\p_x^2\left[\frac{7533\,v_1^4}{4480 v^6}-\frac{43081\,{v_2}\,v_1^2}{13440 {v}^5}+\frac{2063\, {v_3} v_1}{1920 \, {v}^4}
+\frac{2077 \, v_2^2}{3360 {v}^4}-\frac{619\, {v_4}}{2240\,{v}^3}\right.\nn\\
&+\frac{1}{v_1}\left(\frac{239\, {v_2}
{v_3}}{1344 {v}^3}+\frac{41 {v_5}}{672 {v}^2}\right)-\frac{1}{v_1^2}\left(\frac{157 v_2^3}{1920 \ {v}^3}+\frac{1549
{v_4} {v_2}}{6720 {v}^2}+\frac{13 v_3^2}{80 \, {v}^2}+\frac{7
{v_6}}{640 {v}}\right)\nn\\
&+\frac{1}{v_1^3}\left(\frac{8753 {v_3}
v_2^2}{13440 \ {v}^2}+\frac{383 {v_5} {v_2}}{4480 {v}}+\frac{689
{v_3} {v_4}}{4480 \, {v}}+\frac{{v_7}}{384}\right)\nn\\
&-\frac{1}{v_1^4}\left(\frac{4303 \, v_2^4}{13440 {v}^2}+\frac{185
{v_4} v_2^2}{448 {v}}+\frac{2607 v_3^2 \, {v_2}}{4480
{v}}+\frac{21\,v_2 v_6}{640}+\frac{103 v_4^2}{2240}+\frac{159 \,
{v_3} {v_5}}{2240}\right) \nn\\
&+\frac{1}{v_1^5}\left(\frac{9343 \, {v_3} v_2^3}{6720
{v}}+\frac{1059 {v_5} v_2^2}{4480}+\frac{3819 {v_3} {v_4}
{v_2}}{4480}+\frac{177 v_3^3}{896}\right)\nn\\
&-\frac{1}{v_1^6}\left.\left(\frac{131 v_2^5}{210 {v}}+\frac{83}{70}
{v_4} \ v_2^3+\frac{2241}{896} v_3^2 v_2^2\right)
+\frac{59}{14} \frac{v_2^4 {v_3}}
{v_1^7}-\frac{5}{3} \frac{v_2^6}{v_1^8}\right] +O(\epsilon^8)
\label{qmt-2}
\end{align}
transforming any monotone solution of the dispersionless equation
$$
v_t = 5 v^2 v_x
$$
to a solution of \eqref{a22}. In this long formula we denote the jet
coordinates by $v_1=v_x$, $v_2=v_{xx}$ etc. According to \cite{LWZ, LZ2},
any local Hamiltonian structure of \eqref{a22} with coefficients depending {\it polynomially} on the jet coordinates $u_x$, $u_{xx}$, \dots must be obtained from some dispersionless Hamiltonian structure of the form
\[\{v(x),v(y)\}=\varphi(v(x))\dl'(x-y)+\frac12\varphi'(v(x))v_x(x)\dl(x-y)\]
by applying the quasitriviality transformation \eqref{qmt-2}. The unknown function $\varphi(v)$ has to be chosen in such a way to ensure cancellation of all the jet dependent denominators in the transformed bracket
\begin{equation}
\{u(x),u(y)\}=\varphi(u)\dl'(x-y)+\frac12\varphi'(u)u_x\dl(x-y)+\e^2\,Z_2+\e^4\,Z_4+\e^6\,Z_6+\cdots.\label{a220-ham}
\end{equation}
Here $Z_2$ is a polynomial for any $\varphi(v)$, while $Z_4$ contains the following term
\[-\frac{3}{160}\frac{u_{xx}^4}{u^2\,u_x^4}\left[3\,u^2\,\varphi''(u)-2\,u\,\varphi'(u)+2\varphi(u)\right]\dl'(x-y).\]
So, to ensure $Z_4$ is a polynomial, we must have
\[\varphi(u)=c_1 u+c_2 u^{\frac23}\]
for some constants $c_1$ and $c_2$.
Next, $Z_6$ contains the following term
\[\frac{5\,c_2\,u_{xx}^5}{432\,u^{10/3}\,u_x^4}\dl'(x-y),\]
which implies $c_2=0$.
So we have $\varphi(u)=c_1\,u$, by taking $c_1=2$, we obtain the Hamiltonian structure \eqref{a22-h}. The Proposition is proved.

\medskip

In a similar way we have analyzed another example of an integrable scalar equation associated with $A_2^{(2)}$. It would be interesting to prove in general that the Drinfeld - Sokolov hierarchies associated with twisted Kac - Moody Lie algebras never admit a local bihamiltonian structure.

\vskip 0.4truecm \noindent{\bf Acknowledgments.}
The authors thank Yassir Dinar for fruitful discussions.
The work of Dubrovin is
partially supported by European Science Foundation Programme ``Methods of
Integrable Systems, Geometry, Applied Mathematics" (MISGAM), Marie Curie RTN ``European Network in
Geometry, Mathematical Physics and Applications"  (ENIGMA),
and by Italian Ministry of Universities and Researches (MUR) research grant PRIN 2006
``Geometric methods in the theory of nonlinear waves and their applications".
He also thanks for hospitality and support the Department of Mathematics of Tsinghua University
where part of this work was done.
The work of Zhang is partially supported by NSFC No.10631050 and the
National Basic Research Program of China (973 Program) No.2007CB814800. He also thanks SISSA,
where part of this work was done, for
hospitality and support.

\appendix

\section{Data for $E_7$}

{\scriptsize
The Weyl basis
\begin{align*}
X_1=& e_{7,8}+e_{9,10}+e_{11,12}+e_{13,15}+e_{16,18}+e_{19,22}
-e_{35,38}-e_{39,41}-e_{42, 44}-e_{45, 46}-e_{47, 48}-e_{49, 50},\\
X_2=& e_{5,6}+e_{7,9}+e_{8,10}-e_{20,23}-e_{24,26}-e_{27,29}
-e_{28,30}-e_{31,33}-e_{34,37}+e_{47,49}+e_{48,50}+e_{51,52},\\
X_3=& e_{5,7}+e_{6,9}+e_{12,14}+e_{15,17}+e_{18,21}+e_{22,25}
-e_{32,35}-e_{36,39}-e_{40,42}-e_{43,45}+e_{48,51}+e_{50,52},\\
X_4=& e_{4,5}-e_{9,11}-e_{10,12}-e_{17,20}-e_{21,24}-e_{25,28}
-e_{29,32}-e_{33,36}-e_{37,40}+e_{45,47}+e_{46,48}-e_{52,53},\\
X_5=& e_{3,4}-e_{11,13}-e_{12,15}-e_{14,17}-e_{24,27}-e_{26,29}
-e_{28,31}-e_{30,33}+e_{40,43}+e_{42,45}+e_{44,46}-e_{53,54},\\
X_6=& e_{2,3}-e_{13,16}-e_{15,18}-e_{17,21}-e_{20,24}-e_{23,26}
+e_{31,34}+e_{33,37}+e_{36,40}+e_{39,42}+e_{41,44}-e_{54,55},\\
X_7=& e_{1,2}+e_{16,19}+e_{18,22}+e_{21,25}+e_{24,28}+e_{26,30}
+e_{27,31}+e_{29,33}+e_{32,36}+e_{35,39}+e_{38,41}+e_{55,56},\\
Y_i=&{X_i}^T,\ H_i=[X_i, Y_i],\quad i=1,\dots,7.
\end{align*}

The invariant bilinear form
\[ \la A, B\ra_\g=\frac1{12}\tr(A\,B).\]

The Dynkin diagram
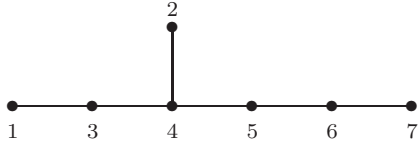
\begin{center}
\setlength{\unitlength}{0.7mm}
\begin{picture}(77,24)
\put(1,6){\circle*{2}} \put(16,6){\circle*{2}} \put(31,6){\circle*{2}} \put(46,6){\circle*{2}}
\put(61,6){\circle*{2}} \put(76,6){\circle*{2}} \put(31,21){\circle*{2}}
\put(1,6){\line(1,0){15}} \put(16,6){\line(1,0){15}} \put(31,6){\line(1,0){15}}
\put(46,6){\line(1,0){15}} \put(61,6){\line(1,0){15}} \put(31,6){\line(0,1){15}}
\put(0,0){$1$} \put(30,23){$2$} \put(15,0){$3$} \put(30,0){$4$}
\put(45,0){$5$} \put(60,0){$6$} \put(75,0){$7$}
\end{picture}
\end{center}

The basis of $V$
\begin{align*}
\gamma_1=&X_7+\frac{52}{27}X_6+\frac{25}{9}X_5+\frac{32}{9}X_4+\frac{22}{9}X_3+\frac{49}{27}X_2+\frac{34}{27}X_1
=\frac1{27}I_+,\\
\gamma_2=&X_{0011111}-\frac{49}{44}X_{0101111}-\frac{49}{54}X_{0111110}+\frac{196}{117}X_{0112100}
+\frac{34}{27}X_{1011110}+\frac{833}{1404}X_{1111100},\\
\gamma_3=&X_{0112111}-\frac{25}{9}X_{0112210}+\frac{17}{24}X_{1111111}-\frac{34}{27}X_{1112110}
-\frac{187}{117}X_{1122100},\\
\gamma_4=&X_{0112221}+\frac{17}{26}X_{1112211}+\frac{187}{325}X_{1122111},
\gamma_5=X_{1122221}-\frac{24}{13}X_{1123211}+\frac{392}{117}X_{1223210},\\
\gamma_6=&X_{1123321}-\frac{49}{75}X_{1223221},\ \gamma_7=X_{2234321}.
\end{align*}

The flat coordinates of the first metric
\begin{align*}
t_1=
&\frac{38484341090380842575u_1^9}{8105110306037952534}-\frac{636900975026373245u_2u_1^6}{58884863779069614}
+\frac{36219697127845u_3u_1^5}{132177023073108}\\
&-\frac{17412855149u_4u_1^4}{16788221190}+\frac{2705810236u_5u_1^3}{5036466357}
+\frac{6468204279499075u_2^2u_1^3}{2566845151777764}\\
&-\frac{4934129232697u_2u_3u_1^2}{2880858756204}-\frac{553112u_6u_1^2}{13286025}-
\frac{2474855971u_3^2u_1}{38799444528}
+\frac{270317u_2u_4u_1}{1108809}\\
&-\frac{624221464415u_2^3}{7944932410272}-\frac{329273u_3u_4}{9240075}-\frac{433415u_2u_5}{16888014}
-\frac{u_7}{81},\\
t_2=
&-\frac{320145055952981u_1^7}{68630377364883}+\frac{2127572498737u_2u_1^4}{460255540932}
-\frac{3585503245u_3u_1^3}{1549681956}+\frac{160921 u_4u_1^2}{328050}\\
&-\frac{19277059375u_2^2u_1}{20062960632}-\frac{1798 u_5u_1}{19683}+\frac{11494397 u_2u_3}{67552056}
+\frac{124 u_6}{2025},\\
t_3= &\frac{333872584054 u_1^6}{847288609443}+\frac{8686821907 u_2u_1^3}{3357644238}
+\frac{24931445 u_3u_1^2}{165809592}
+\frac{7337 u_4u_1}{52650}-\frac{44605625 u_2^2}{495381744}-\frac{725 u_5}{3159},\\
t_4=&\frac{162925u_1^5}{59049}-\frac{149891u_2u_1^2}{69498}+\frac{665u_3u_1}{1404}-\frac{161 u_4}{650},\\
t_5=&\frac{92912470 u_1^4}{43046721}-\frac{382375 u_2u_1}{938223}+\frac{11305 u_3}{12636},\\
t_6=&-\frac{2185 u_2}{1782}+\frac{439622 u_1^3}{531441},\ t_7=-\frac{133u_1}{27}.
\end{align*}

The Euler and the unity vector fields
\begin{align*}
&E=t_1\frac{\p}{\p t_1}+\frac79 t_2 \frac{\p}{\p t_2}+\frac23 t_3 \frac{\p}{\p t_3}+\frac59 t_4 \frac{\p}{\p t_4}+
\frac49 t_5 \frac{\p}{\p t_5}+\frac13 t_6 \frac{\p}{\p t_6}+\frac19 t_7 \frac{\p}{\p t_7},\\
&e=-\frac1{81} \frac{\p}{\p t_1}.
\end{align*}

The potential of the Frobenius manifold is given by $F(t)=\frac{3^7}{2} \tilde{F}$, where $\tilde{F}$ is
obtained from the $E_7$ free energy of \cite{zuber}(given in Appendix D) by the following substitution:
\beq
t_{0}\mapsto t_1,\ t_{4}\mapsto t_2,\ t_{6}\mapsto t_3,\ t_{8}\mapsto t_4,\ t_{10}\mapsto t_5,\ t_{12}\mapsto t_6,\
t_{16}\mapsto t_7.\nn
\eeq

The components of the tensor $K$:
\begin{align*}
&K_2^1 = \frac{1265 t_7}{5301},\ K_3^1 = -\frac{572296 t_7^2}{2355525},\
K_4^1 = \frac{158794 t_6}{211071} - \frac{3974393683 t_7^3}{46960314975},\\
&K_5^1 = -\frac{44538734101583 t_7^4}{4034736342022050}
- \frac{23803046371 t_6 t_7}{45336995445} - \frac{1049501 t_5}{9665775},\\
&K_6^1 = -\frac{135026437662337483 t_7^5}{22921337159027266050}
+ \frac{4216764458944 t_6 t_7^2}{20156828174847}
- \frac{3478865269 t_5 t_7}{6476713635} + \frac{974770 t_4}{1175967},\\
&K_7^1 = \frac{11916123469344407708 t_7^7}{33997824196158072507075}
- \frac{833538047791538069 t_6 t_7^4}{22921337159027266050}
- \frac{12187906171012 t_5 t_7^3}{288195453001575} \\
&- \frac{3127213133 t_4 t_7^2}{9690223725} + \frac{18075424576 t_6^2 t_7}{73834535439}
- \frac{2899844 t_3 t_7}{4616829} + \frac{1883 t_2}{5301}
- \frac{44290188771 t_5 t_6}{75561659075},\\
&K_3^2 = \frac{31}{75},\
K_4^2 = \frac{4029628 t_7}{4817925},\
K_5^2 = -\frac{19829032882 t_7^2}{68991080025},\
K_4^3 = -\frac{29}{42},\
K_5^4 = \frac{27}{34},\\
&K_6^2 = \frac{33964731481462 t_7^3}{460101512686725}
- \frac{33203728 t_6}{22343373},\
K_5^3 = -\frac{411452 t_7}{2706417},\
K_6^3 = -\frac{213705872 t_7^2}{668484999},\\
&K_7^2 = -\frac{27908652530225533 t_7^5}{3488029567678062225}
+ \frac{42670790888702 t_6 t_7^2}{153367170895575}
- \frac{387454306 t_5 t_7}{579756975} + \frac{236468 t_4}{229425},\\
&K_7^3 = -\frac{6028481916881 t_7^4}{456100629967710} - \frac{184848088 t_6 \
t_7}{257109615} - \frac{21373 t_5}{159201},\
K_6^4 = -\frac{4518 t_7}{4199},\\
&K_7^4 = \frac{7249042 t_7^3}{53054365} - \frac{1926 t_6}{1615},\
K_6^5 = \frac{25}{39},\
K_7^5 = -\frac{939370 t_7^2}{2069613},\
K_7^6 = \frac{4922 t_7}{8379}.
\end{align*}
}

\section{Data for $E_8$}
{\scriptsize
The Weyl basis
\begin{align*}
X_1=&e_{8,9}+e_{10,11}+e_{12,13}+e_{14,16}+e_{17,19}+e_{20,23}+e_{24,28}-e_{42,47}-e_{48,53}-e_{54,59}-e_{56,61}\\
&-e_{60,65}-e_{62,67}-e_{66,71}-e_{68,73}-e_{69,75}-e_{72,78}-e_{74,80}-e_{76,82}-e_{81,86}-e_{83,88}-e_{84,90}\\
&-e_{89,94}-e_{91,96}-e_{97,102}-e_{98,104}-e_{105,110}-e_{112,118}-e_{120,127}+e_{135,143}+e_{142,149}\\
&+e_{148,155}+e_{150,157}+e_{154,161}+e_{156,163}+e_{160,167}+e_{162,169}+e_{164,171}+e_{166,172}+e_{168,174}\\
&+e_{170,176}+e_{173,179}+e_{175,181}+e_{177,184}+e_{180,185}+e_{182,188}+e_{186,192}+e_{189,195}+e_{193,198}\\
&+e_{199,205}-e_{224,228}-e_{225,229}-e_{230,232}-e_{233,235}-e_{236,237}-e_{238,239}-e_{240,241}+2e_{127,136}\\
&+\frac12(e_{122,136}-e_{121,136}-e_{123,136}+e_{124,136}-e_{125,136}-3e_{126,136}-e_{128,136}),
\end{align*}
\begin{align*}
X_2=&e_{6,7}+e_{8,10}+e_{9,11}-e_{21,25}-e_{26,29}-e_{30,33}-e_{31,34}-e_{35,38}-e_{36,40}-e_{39,44}-e_{41,45}\\
&-e_{46,51}-e_{52,58}+e_{66,72}+e_{71,78}+e_{74,81}+e_{79,85}+e_{80,86}+e_{83,89}+e_{87,93}+e_{88,94}+e_{91,97}\\
&+e_{95,101}+e_{96,102}+e_{98,105}+e_{103,109}+e_{104,110}+e_{111,117}+e_{119,125}+e_{119,126}+e_{119,127}\\
&+e_{125,134}+e_{126,134}+e_{133,141}+e_{140,147}+e_{142,150}+e_{146,153}+e_{148,156}+e_{149,157}+e_{152,159}\\
&+e_{154,162}+e_{155,163}+e_{158,165}+e_{160,168}+e_{161,169}+e_{166,173}+e_{167,174}+e_{172,179}-e_{196,202}\\
&-e_{200,206}-e_{203,208}-e_{207,211}-e_{209,213}-e_{212,216}-e_{214,217}-e_{215,218}-e_{219,222}-e_{223,227}\\
&+e_{238,240}+e_{239,241}+e_{242,243},
\end{align*}
\begin{align*}
X_3=&e_{6,8}+e_{7,10}+e_{13,15}+e_{16,18}+e_{19,22}+e_{23,27}+e_{28,32}-e_{37,42}-e_{43,48}-e_{49,54}-e_{50,56}\\
&-e_{55,60}-e_{57,62}-e_{63,68}-e_{64,69}-e_{70,76}-e_{71,79}-e_{77,84}-e_{78,85}-e_{80,87}-e_{86,93}-e_{88,95}\\
&-e_{94,101}-e_{96,103}-e_{102,109}-e_{104,111}-e_{110,117}+e_{112,120}-e_{118,125}+e_{118,126}+e_{118,127}\\
&+e_{125,135}-e_{127,135}+e_{133,142}-e_{136,143}+e_{140,148}+e_{141,150}+e_{146,154}+e_{147,156}+e_{152,160}\\
&+e_{153,162}+e_{158,166}+e_{159,168}+e_{165,173}+e_{171,178}+e_{176,183}+e_{181,187}+e_{184,190}+e_{185,191}\\
&+e_{188,194}+e_{192,197}+e_{195,201}+e_{198,204}+e_{205,210}-e_{220,224}-e_{221,225}-e_{226,230}-e_{231,233}\\
&-e_{234,236}-e_{239,242}-e_{241,243},
\end{align*}
\begin{align*}
X_4=&e_{5,6}-e_{10,12}-e_{11,13}-e_{18,21}-e_{22,26}-e_{27,31}-e_{32,36}-e_{33,37}-e_{38,43}-e_{44,49}-e_{45,50}\\
&-e_{51,57}-e_{58,64}-e_{60,66}-e_{65,71}-e_{68,74}-e_{73,80}-e_{76,83}-e_{82,88}-e_{84,91}+e_{85,92}-e_{90,96}\\
&+e_{93,100}+e_{101,108}+e_{105,112}+e_{109,116}+e_{110,118}+e_{111,119}+e_{117,124}+e_{117,125}+e_{124,133}\\
&+e_{125,133}+e_{132,140}+e_{134,141}+e_{135,142}+e_{139,146}+e_{143,149}+e_{145,152}+e_{151,158}-e_{156,164}\\
&-e_{162,170}-e_{163,171}-e_{168,175}-e_{169,176}-e_{173,180}-e_{174,181}-e_{179,185}-e_{190,196}-e_{194,200}\\
&-e_{197,203}-e_{201,207}-e_{204,209}-e_{210,215}-e_{216,220}-e_{217,221}-e_{222,226}-e_{227,231}-e_{236,238}\\
&-e_{237,239}+e_{243,244},
\end{align*}
\begin{align*}
X_5=&e_{4,5}-e_{12,14}-e_{13,16}-e_{15,18}-e_{26,30}-e_{29,33}-e_{31,35}-e_{34,38}-e_{36,41}-e_{40,45}-e_{49,55}\\
&-e_{54,60}-e_{57,63}-e_{59,65}-e_{62,68}-e_{64,70}-e_{67,73}-e_{69,76}-e_{75,82}+e_{91,98}+e_{92,99}+e_{96,104}\\
&+e_{97,105}+e_{100,107}+e_{102,110}+e_{103,111}+e_{108,115}+e_{109,117}+e_{116,123}+e_{116,124}+e_{123,132}\\
&+e_{124,132}+e_{131,139}+e_{133,140}+e_{138,145}+e_{141,147}+e_{142,148}+e_{144,151}+e_{149,155}+e_{150,156}\\
&+e_{157,163}-e_{170,177}-e_{175,182}-e_{176,184}-e_{180,186}-e_{181,188}-e_{183,190}-e_{185,192}-e_{187,194}\\
&-e_{191,197}-e_{207,212}-e_{209,214}-e_{211,216}-e_{213,217}-e_{215,219}-e_{218,222}-e_{231,234}-e_{233,236}\\
&-e_{235,237}+e_{244,245},
\end{align*}
\begin{align*}
X_6=&e_{3,4}-e_{14,17}-e_{16,19}-e_{18,22}-e_{21,26}-e_{25,29}-e_{35,39}-e_{38,44}-e_{41,46}-e_{43,49}-e_{45,51}\\
&-e_{48,54}-e_{50,57}-e_{53,59}-e_{56,62}-e_{61,67}+e_{70,77}+e_{76,84}+e_{82,90}+e_{83,91}+e_{88,96}+e_{89,97}\\
&+e_{94,102}+e_{95,103}+e_{99,106}+e_{101,109}+e_{107,114}+e_{108,116}+e_{115,122}+e_{115,123}+e_{122,131}\\
&+e_{123,131}+e_{130,138}+e_{132,139}+e_{137,144}+e_{140,146}+e_{147,153}+e_{148,154}+e_{155,161}+e_{156,162}\\
&+e_{163,169}+e_{164,170}+e_{171,176}+e_{178,183}-e_{182,189}-e_{186,193}-e_{188,195}-e_{192,198}-e_{194,201}\\
&-e_{197,204}-e_{200,207}-e_{203,209}-e_{206,211}-e_{208,213}-e_{219,223}-e_{222,227}-e_{226,231}-e_{230,233}\\
&-e_{232,235}+e_{245,246},
\end{align*}
\begin{align*}
X_7=&e_{2,3}-e_{17,20}-e_{19,23}-e_{22,27}-e_{26,31}-e_{29,34}-e_{30,35}-e_{33,38}-e_{37,43}-e_{42,48}+e_{46,52}\\
&-e_{47,53}+e_{51,58}+e_{57,64}+e_{62,69}+e_{63,70}+e_{67,75}+e_{68,76}+e_{73,82}+e_{74,83}+e_{80,88}+e_{81,89}\\
&+e_{86,94}+e_{87,95}+e_{93,101}+e_{100,108}+e_{106,113}+e_{107,115}+e_{114,121}+e_{114,122}+e_{121,130}\\
&+e_{122,130}+e_{129,137}+e_{131,138}+e_{139,145}+e_{146,152}+e_{153,159}+e_{154,160}+e_{161,167}+e_{162,168}\\
&+e_{169,174}+e_{170,175}+e_{176,181}+e_{177,182}+e_{183,187}+e_{184,188}+e_{190,194}-e_{193,199}+e_{196,200}\\
&-e_{198,205}+e_{202,206}-e_{204,210}-e_{209,215}-e_{213,218}-e_{214,219}-e_{217,222}-e_{221,226}-e_{225,230}\\
&-e_{229,232}+e_{246,247},
\end{align*}
\begin{align*}
X_8=&e_{1,2}+e_{20,24}+e_{23,28}+e_{27,32}+e_{31,36}+e_{34,40}+e_{35,41}+e_{38,45}+e_{39,46}+e_{43,50}+e_{44,51}\\
&+e_{48,56}+e_{49,57}+e_{53,61}+e_{54,62}+e_{55,63}+e_{59,67}+e_{60,68}+e_{65,73}+e_{66,74}+e_{71,80}+e_{72,81}\\
&+e_{78,86}+e_{79,87}+e_{85,93}+e_{92,100}+e_{99,107}+e_{106,114}+2e_{113,121}-e_{113,122}+e_{113,123}\\
&-e_{113,124}+e_{113,125}-e_{113,126}+e_{113,127}+e_{113,128}+e_{121,129}+e_{130,137}+e_{138,144}+e_{145,151}\\
&+e_{152,158}+e_{159,165}+e_{160,166}+e_{167,172}+e_{168,173}+e_{174,179}+e_{175,180}+e_{181,185}+e_{182,186}\\
&+e_{187,191}+e_{188,192}+e_{189,193}+e_{194,197}+e_{195,198}+e_{200,203}+e_{201,204}+e_{206,208}+e_{207,209}\\
&+e_{211,213}+e_{212,214}+e_{216,217}+e_{220,221}+e_{224,225}+e_{228,229}+e_{247,248}.
\end{align*}
\begin{align*}
Y_1=&{X_1}^T-e_{127,120}-e_{136,127}+\frac12(e_{121,120}-e_{122,120}+e_{123,120}-e_{124,120}+e_{125,120}+3e_{126,120}\\
&e_{128,120}+e_{136,121}-e_{136,122}+e_{136,123}-e_{136,124}+e_{136,125}+3e_{136,126}+e_{136,128}),\\
Y_2=&{X_2}^T+e_{134,127}-e_{127,119},\ Y_3={X_3}^T-e_{126,118}-e_{135,126},\ Y_i={X_i}^T\,(i=4,5,6,7),\\
Y_8=&{X_8}^T-e_{121,113}+e_{122,113}-e_{123,113}+e_{124,113}-e_{125,113}+e_{126,113}-e_{127,113}-e_{128,113}\\
&+e_{129,121}-e_{129,122}+e_{129,123}-e_{129,124}+e_{129,125}-e_{129,126}+e_{129,127}+e_{129,128},\\
H_i=&[X_i, Y_i],\quad i=1,\dots,8.
\end{align*}

The invariant bilinear form
\[\la A,B\ra_\g=\frac1{60}\tr(A\,B).\]

The Dynkin diagram
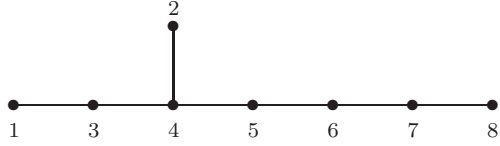
\begin{center}
\setlength{\unitlength}{0.7mm}
\begin{picture}(92,24)
\put(1,6){\circle*{2}} \put(16,6){\circle*{2}} \put(31,6){\circle*{2}} \put(46,6){\circle*{2}}
\put(61,6){\circle*{2}} \put(76,6){\circle*{2}} \put(91,6){\circle*{2}} \put(31,21){\circle*{2}}
\put(1,6){\line(1,0){15}} \put(16,6){\line(1,0){15}} \put(31,6){\line(1,0){15}}
\put(46,6){\line(1,0){15}} \put(61,6){\line(1,0){15}} \put(76,6){\line(1,0){15}} \put(31,6){\line(0,1){15}}
\put(0,0){$1$} \put(30,23){$2$} \put(15,0){$3$} \put(30,0){$4$}
\put(45,0){$5$} \put(60,0){$6$} \put(75,0){$7$} \put(90,0){$8$}
\end{picture}
\end{center}

The basis of $V$
\begin{align*}
\gamma_1=&X_8+\frac{57}{29}X_7+\frac{84}{29}X_6+\frac{110}{29}X_5+\frac{135}{29}X_4+\frac{91}{29}X_3+\frac{68}{29}X_2
+\frac{46}{29}X_1=\frac1{58}I_+.\\
\gamma_2=&X_{01111111}-\frac{135}{29}X_{01121110}+\frac{4950}{551}X_{01122100}-\frac{69}{34}X_{10111111}
-\frac{92}{29}X_{11111110}\\&+\frac{2070}{551}X_{11121100}+\frac{4485}{1102}X_{11221000},\\
\gamma_3=&X_{01122221}+\frac{46}{57}X_{11122211}+\frac{299}{684}X_{11222111}+\frac{2093}{1653}X_{11222210}
-\frac{4485}{2204}X_{11232110}\\&+\frac{25415}{10469}X_{12232100},\\
\gamma_4=&X_{11222221}-\frac{45}{19}X_{11232211}+\frac{4950}{551}X_{11233210}+\frac{510}{133}X_{12232111}
-\frac{3060}{551}X_{12232210},\\
\gamma_5=&X_{12233321}-\frac{45}{28}X_{12243221}+\frac{195}{76}X_{12343211}-\frac{4485}{1102}X_{22343210},\\
\gamma_6=&X_{12244321}-\frac{91}{110}X_{12343321}+\frac{299}{660}X_{22343221},\\
\gamma_7=&X_{22454321}+\frac{68}{91}X_{23354321},\ \gamma_8=X_{23465432}.
\end{align*}

The flat coordinates of the first metric
\begin{align*}
t_1=&\frac{97744996564015238279820370313216u_1^{15}}{1213479667630994687008453125}
-\frac{1288270064507206171704653824u_2u_1^{11}}{12314889544759321146875}\\
&-\frac{2393328775550101435187456u_3u_1^9}{736464332431224703125}
-\frac{2420766825182080240896u_4u_1^8}{1039574576934578125}\\
&+\frac{862526917251409198851419904u_2^2u_1^7}{23648688008069326578125}
-\frac{131350976991158976u_5u_1^6}{1236117213953125}\\
&+\frac{1245108893989927485504u_2u_3u_1^5}{1083721620290046875}-\frac{963126559584u_6u_1^5}{121563921875}\\
&+\frac{14299027732501070976u_2u_4u_1^4}{13767788279546875}
-\frac{33722086688072420883336u_2^3u_1^3}{16483976224411765625}\\
&+\frac{18146063436077468u_3^2u_1^3}{1440211987515625}-\frac{461307312u_7u_1^3}{1195796875}
+\frac{6018680488536u_3u_4u_1^2}{630920828125}\\
&+\frac{357801015222144u_2u_5u_1^2}{16370735171875}+\frac{32227564652712u_4^2u_1}{6740890953125}
-\frac{1845004006277946u_2^2u_3u_1}{44434852609375}\\
&+\frac{130653081u_2u_6u_1}{169468750}-\frac{178420747189383u_2^2u_4}{14647499890625}
+\frac{561100797u_3u_5}{3480942500}+\frac{243u_8}{15625},
\end{align*}
\begin{align*}
t_2=&\frac{3840041279133973715980288u_1^{12}}{663402718510254451875}
-\frac{744229812509520647424u_2u_1^8}{100987244616501875}\\
&-\frac{10589318985400591616u_3u_1^6}{58380050133271875}-\frac{53628582258067968u_4u_1^5}{247223442790625}
+\frac{1500712331u_3^2}{9448272500}\\
&+\frac{280559110180100677056u_2^2u_1^4}{193929132051903125}-\frac{2377142258796u_5u_1^3}{293963665625}
-\frac{96129730861884u_2^3}{7114499946875}\\
&-\frac{998361u_6u_1^2}{2628125}+\frac{4615114574016u_2u_4u_1}{172323528125}
+\frac{383963689278654u_2u_3u_1^2}{8886970521875}-\frac{12879u_7}{284375},
\end{align*}
\begin{align*}
t_3=&\frac{2549753401933430769664u_1^{10}}{2366478187313630625}-\frac{26925773335669824u_2u_1^6}{36655987156625}
+\frac{266785048403496u_2^2u_1^2}{2427299981875}\\
&-\frac{7567792181264u_3u_1^4}{243569894375}-\frac{8241329088u_4u_1^3}{442050625}
-\frac{364203u_5u_1}{525625}+\frac{2292843u_2u_3}{1779730}-\frac{2709u_6}{27500},\\
t_4=&\frac{1522342602122771968u_1^9}{5440179740950875}-\frac{537827687557056u_2u_1^5}{1412705387375}
-\frac{1371394734296u_3u_1^3}{159580275625}\\&-\frac{19495261224u_4u_1^2}{2027335625}
+\frac{2320060511106u_2^2u_1}{83699999375}-\frac{1925073u_5}{4821250},\\
t_5=&\frac{98329903451648u_1^7}{1293740723175}-\frac{20835557808u_2u_1^3}{335958475}
-\frac{1420978u_3u_1}{1308625}-\frac{887778u_4}{482125},\ t_8=\frac{124u_1}{29},\\
t_6=&\frac{567020228224u_1^6}{14870583025}-\frac{2551308804u_2u_1^2}{196941175}-\frac{1271u_3}{1805},\
t_7=\frac{22811536u_1^4}{3536405}-\frac{10323u_2}{2465}.
\end{align*}

The Euler and the unity vector fields
\begin{align*}
&E=t_1\frac{\p}{\p t_1}+\frac45t_2\frac{\p}{\p t_2}+\frac23t_3\frac{\p}{\p t_3}+\frac35t_4\frac{\p}{\p t_4}
+\frac7{15}t_5\frac{\p}{\p t_5}+\frac25t_6\frac{\p}{\p t_6}+\frac4{15}t_7\frac{\p}{\p t_7}
+\frac1{15}t_8\frac{\p}{\p t_8},\\
&e=\frac{3^5}{5^6}\frac{\p}{\p t_1}.
\end{align*}

The potential of the Frobenius manifold is given by $F(t)=-\frac{5^{12}}{2\cdot 3^{11}}\tilde{F}$,
where $\tilde{F}$ is
obtained from the $E_8$ free energy of \cite{zuber}(given in Appendix D) by the following substitution:
\beq
t_{0}\mapsto t_1,\ t_{6}\mapsto t_2,\ t_{10}\mapsto t_3,\ t_{12}\mapsto t_4,\ t_{16}\mapsto t_5,\ t_{18}\mapsto t_6,\
t_{22}\mapsto t_7,\ t_{28}\mapsto t_8.\nn
\eeq

The components of the tensor $K$:
\begin{align*}
K_2^1=&\frac{1606747t_8^2}{1018660},\ K_3^1=\frac{601725604717t_8^4}{555959642000}+\frac{18606712t_7}{8631175},\\
K_4^1=&\frac{23515325188849703t_8^5}{51379248583848000}+\frac{72379436165441t_7t_8}{10901267734900},\\
K_5^1=&\frac{68957356685371184639t_8^7}{4819885123299777312000}+\frac{398354166156373021t_7t_8^3}{82917361810316400}\\
&+\frac{1042486747513t_6t_8}{287607914708}+\frac{33089674283t_5}{9730244164},
\end{align*}

\begin{align*}
K_6^1=&\frac{89894219185919589074208393t_8^8}{740233940665443716429200000}
+\frac{540313685410365273433t_7t_8^4}{372713541337372218000}\\
&+\frac{431357837132874729t_6t_8^2}{61633372838501000}+\frac{56454979942439t_5t_8}{14380395735400}
+\frac{33101974334668t_7^2}{8987747334625},
\end{align*}
\begin{align*}
K_7^1=&\frac{1357568393016120771637825259t_8^{10}}{86751806948718464815666000000}
+\frac{541717459607162768848213339t_7t_8^6}{408322593076744759707720000}\\
&+\frac{63779850971187562357459t_6t_8^4}{35407786427050360710000}
+\frac{85588038590694599881t_5t_8^3}{15754298743960116000}+\frac{4464353t_3}{1726235}\\
&+\frac{6210460349164259675293t_7^2t_8^2}{728636316908155365000}+\frac{2173589973652t_4t_8}{279897414815}
+\frac{302135911297903t_6t_7}{40984127845890},
\end{align*}
\begin{align*}
K_8^1=&\frac{2828405696582354177109987369765353t_8^{13}}{2192552418323772900665150972160000000}
+\frac{104659508314126104547782052831201t_7t_8^9}{540637260904413472731230512000000}\\
&+\frac{102899753677493534416025877059t_6t_8^7}{84386669235860583672928800000}
+\frac{8481094570209522684734123t_5t_8^6}{18773452555252632630240000}\\
&+\frac{6651723858574138184203661671t_7^2t_8^5}{1361075310255815865692400000}
+\frac{1289206380441371567t_4t_8^4}{360580492313672000}+\frac{670532779771t_3t_8^3}{111191928400}\\
&+\frac{155409080204960966846669t_6t_7t_8^3}{17703893213525180355000}
+\frac{48735507458090179297t_5t_7t_8^2}{2625716457326686000}+\frac{47752579078487t_4t_7}{5597948296300}\\
&+\frac{2427437707690958t_6^2t_8}{308166864192505}+\frac{1029963t_2t_8}{203732}
+\frac{3178482572407t_5t_6}{719019786770}+\frac{253103049412822813259t_7^3t_8}{38349279837271335000},
\end{align*}

\begin{align*}
K_3^2=&\frac{6307t_8}{2666},\ K_4^2=\frac{4500609957t_8^2}{1253739820},\
K_5^2=\frac{382063294638363t_8^4}{298801303820960}+\frac{1389601541t_7}{282855935},\\
K_6^2=&\frac{249245177327411089t_8^5}{818970646753180000}+\frac{281315094790593t_7t_8}{35950989338500},\\
K_7^2=&\frac{1688680021615354182905594447t_8^7}{8166451861534895194154400000}
+\frac{367873390741219178203t_7t_8^3}{59078620289850435000}\\&
+\frac{227080200740691t_6t_8}{27322751897260}+\frac{6011337539t_5}{1131423740},\\
K_8^2=&\frac{513287066338916479781267251t_8^{10}}{24864125687376253900656000000}
+\frac{195277479174731922608232983t_7t_8^6}{107453313967564410449400000}\\
&+\frac{766053117047317547t_6t_8^4}{296734015730570800}+\frac{2118272066471217t_5t_8^3}{298801303820960}
+\frac{11494829512942367213t_7^2t_8^2}{1036467022628955000}\\
&+\frac{81503559t_4t_8}{8471215}+\frac{21253t_3}{6665}+\frac{123284247247959t_6t_7}{13661375948630},
\end{align*}

\begin{align*}
K_4^3=&\frac{98}{37},\ K_5^3=\frac{17937969t_8^2}{8818136},\ K_6^3=\frac{5442056846681t_8^3}{1981876066000},\\
K_7^3=&\frac{137374164909966859t_8^5}{361870750890940000}+\frac{683860335753t_7t_8}{77489370100},\\
K_8^3=&\frac{7182452159729390399233t_8^8}{40426843274219975775000}+
\frac{1428397083604124417t_7t_8^4}{542806126336410000}+\frac{388911287667t_6t_8^2}{39637521320}\\
&+\frac{247282497t_5t_8}{44090680}+\frac{58064641684t_7^2}{11623405515},
\end{align*}

\begin{align*}
K_5^4=&\frac{16497t_8}{3844},\ K_6^4=\frac{1908504231t_8^2}{863939000},\
K_7^4=\frac{1510115085289879t_8^4}{946479732060000}+\frac{305059657t_7}{50668725},\\
K_8^4=&\frac{75610072721381020687t_8^7}{1127863107911978400000}+\frac{3665966665482509t_7t_8^3}{473239866030000}
+\frac{99038823t_6t_8}{17278780}+\frac{96867t_5}{19220},\\
K_6^5=&\frac{86}{25},\ K_7^5=\frac{48222393t_8^2}{9129500},\ K_8^5=\frac{19962579297353t_8^5}{21758154760000}
+\frac{108825561t_7t_8}{9129500},\\
K_7^6=&\frac{4633t_8}{1178},\ K_8^6=\frac{613126423t_8^4}{280750384}+\frac{12628t_7}{2945},\
K_8^7=\frac{85803t_8^2}{19220},
\end{align*}
}

\noindent{Email addresses:\newline
dubrovin@sissa.it,
liusq@mail.tsinghua.edu.cn,
youjin@mail.tsinghua.edu.cn
\end{document}